\documentclass[a4paper, 12pt]{article}
\usepackage{amsmath, amsthm, amssymb, amscd, mathrsfs, amsfonts,graphicx, tikz, tikz-cd, pgf, hyperref, mathtools, scalerel, float, subfigure, multicol, xcolor, url, enumitem}
\usepackage[title]{appendix}
\usepackage[margin=2.5cm]{geometry}
\usepackage[utf8]{inputenc}
\usepackage{csquotes}
\usetikzlibrary{shapes.geometric}
\usetikzlibrary{calc}
\usepackage[font=small,labelfont=bf]{caption}
\makeatletter \renewcommand{\fnum@figure}{Fig. \thefigure} \makeatother
\newtheorem{theorem}{Theorem}[section]
\newtheorem{corollary}[theorem]{Corollary}
\newtheorem{definition}[theorem]{Definition}
\newtheorem{example}[theorem]{Example}
\newtheorem{conjecture}[theorem]{Conjecture}
\newtheorem{lemma}[theorem]{Lemma}

\newtheorem{observation}{Observation}

\newtheorem{proposition}[theorem]{Proposition}
\newtheorem{remark}[theorem]{Remark}
\newtheorem{question}[theorem]{Question}

\numberwithin{equation}{section}
\numberwithin{figure}{section}

\begin{document}
\title{The Domino Problem of the Hyperbolic Plane for Regular Polygons}
\author{Arun Maiti} 
\providecommand{\keywords}[1]{\textit{Keywords:} #1}
\maketitle
\begin{abstract}
We provide a definitive classification of all finite sets of regular polygons that admit a tiling of the hyperbolic plane, thereby establishing the decidability of the Domino Problem for this class of prototiles. We show that admissibility is determined by a finite set of local and inductive combinatorial constraints. This classification further leads to the discovery of the first known examples of weakly aperiodic protosets consisting of regular polygons in $\mathbb{H}^2$.
 \end{abstract}

\keywords{hyperbolic tilings, domino problems, homogeneous tilings, pseudo-homogeneous tilings, aperiodic tiles, Heesch problem}

\section{Introduction} \label{intro}

The classification of tilings by regular polygons has historically been a well studied subject within discrete geometry, famously beginning with the five Platonic solids and the eleven Archimedean tilings of the Euclidean plane. In Euclidean space, the possibility of such tilings is strictly limited by the fact that the interior angles of a regular $p$-gon are fixed at $\pi(p-2)/p$. Consequently, the vertex-type--the cyclic sequence of polygons meeting at a point--must satisfy a rigid interior angle sum of exactly $2\pi$, a constraint that allows only a handful of polygon combinations.
\par
In contrast, the hyperbolic plane ($\mathbb{H}^2$) offers a vastly larger variety for tilings by regular polygons. This geometric freedom arises because the interior angle of a regular hyperbolic $p$-gon is not constant; it is a strictly monotonic function of its side length $L$. This allows for the existence of a unique side length $L_0$ such that any combination of regular polygons satisfying a simple condition can be metrically realized around a single vertex to meet edge-to-edge.
\\
However, the ability to fit a set of polygons around a single vertex is merely a local necessary condition and does not guarantee the existence of a global tiling. The challenge is in ensuring that a patch of polygons can be expanded infinitely without eventually encountering a combinatorial mismatch or a \enquote{dead-end} where no available polygon from the set can fit the remaining gap.
\par
A finite collection of polygons in $\mathbb{H}^2$ with matching side lengths will be referred to as a \textit{protoset} here. For a certain class of protosets of regular polygons, admissibility of a tiling can be derived using elementary combinatorics or from known results on homogeneous tilings \cite{GS, DG18}. However, a complete classification requires a systematic approach that has been missing from the literature. To fill that gap we introduce the notion of pseudo-homogeneous tilings and develop tools to classify them, which in turn provides a complete resolution to the classification problem.
 \par
The \textit{type} of a vertex in a tiling is defined to be the cyclic sequence of the sizes (number of sides) of the polygons incident to a vertex. A vertex type and its mirror image are considered to be the same. A tiling of a surface is called \textit{homogeneous} (also known as Archimedean and semi-regular) if all vertices have the same type. It is called \textit{pseudo-homogeneous} if the vertex types are the same up to a permutation. 
The \textit{type} of a homogeneous (resp. pseudo-homogeneous) tiling is defined to be the type of its vertices (resp. an unordered tuple). By a slight abuse of notation, we will use the same symbol to denote a tuple and its cyclic equivalence class (a cyclic tuple). 

A tiling of a given type is assumed to be pseudo-homogeneous throughout this paper unless stated otherwise. 
\subsection{Pseudo-homogeneous tilings}

For a pseudo-homogeneous tiling of $\mathbb{H}^2$ of type $\mathfrak{k}$ by regular polygons, the common edge length is uniquely (up to isometry) determined by the tuple $\mathfrak{k}$ whenever $k_i<\infty$, $\vartheta(\mathfrak{k}) > 2$; see Lemma 2.1 in \cite{DG18}. This essentially allows us to view pseudo-homogeneous tilings of $\mathbb{H}^2$ as topological tilings of the plane (faces with continuous curved sides replacing regular tiles), and vice-versa: a pseudo-homogeneous tiling of the plane of type $\mathfrak{k}$ satisfying $\vartheta(\mathfrak{k}) > 2$ (or $\vartheta(\mathfrak{k}) = 2$) can be realized as a geometric tiling of $\mathbb{H}^2$ (respectively of $\mathbb{E}^2$) of type $\mathfrak{k}$ by use of the Cartan-Hadamard theorem; see Lemma 2.5 in \cite{DG18}. Thus, the classification of protosets of regular polygons that admit a pseudo-homogeneous tiling of $\mathbb{H}^2$ amounts to answering the following question:
\begin{question} \label{questionps} Given a tuple $\mathfrak{k}=[k_1, k_2,$ $ \cdots, k_d]$ with $3 \leq k_i < \infty$ satisfying $\vartheta(\mathfrak{k}) \geq 2$, is there a pseudo-homogeneous tiling of the plane of type $\mathfrak{k}$? 
\end{question}
Our answer to the above question is primarily an exhaustive verification of all tuples. For tuples that admit a tiling of the plane, a tiling is constructed inductively layer-by-layer. The zero-\textit{layer}, $X_0$, is a point on the plane, and next, by induction, $k$-layered tiling $X_k$ is a tiling of $D^k(X_0)$ (the disc of radius $k$ centered at $X_0$) such that no vertex lies in the annulus region $D^k(X_0) \setminus D^{k-1}(X_0)$; see Fig. \ref{tiling4455}. The boundary of layer $X_k$, $\partial X_k$ consists of edges of tiling. The inductive construction involves prescribing a rule to attach faces to $\partial X_k$ in circular order (taken to be clockwise) to form layer $X_K$ that applies for any $k$. 
\begin{figure}[H]
\centering
\begin{tikzpicture} [scale=0.9]
 \def\rThree{3cm} 
 \def\rTwo{2.2cm} 
 \def\rOne{1.2cm} 
 \draw (0,0) circle (\rTwo);

 \draw (0,0) circle (\rOne);

 \node[fill=black,circle,inner sep=1.5pt] (V) at (0,0) {}; % V is at the center

 % --- Spokes from V to S^{i-1} region (innermost to second) ---
 % These appear to be segments from V to the second circle
 \draw (V) -- (-10:\rTwo); \fill (-10:\rTwo) circle (1.5pt); % Dot at intersection
 
 \draw (V) -- (80:\rTwo); \fill (80:\rTwo) circle (1.5pt); % Dot at intersection

 \draw (V) -- (240:\rTwo); \fill (240:\rTwo) circle (1.5pt); % Dot at intersection
% -- Lines connecting first to second circle
 \draw (195:\rOne) -- (180:\rTwo); \fill (195:\rOne) circle (1.5pt);
 \draw (195:\rOne) -- (210:\rTwo); 

 \draw (270:\rOne) -- (280:\rTwo); \fill (280:\rTwo) circle (1.5pt); \fill (270:\rOne) circle (1.5pt);
 \draw (270:\rOne) -- (250:\rTwo); \fill (250:\rTwo) circle (1.5pt);
 \draw (310:\rOne) -- (300:\rTwo); \fill (300:\rTwo) circle (1.5pt); \fill (310:\rOne) circle (1.5pt);
 \draw (310:\rOne) -- (330:\rTwo); \fill (330:\rTwo) circle (1.5pt);
 \draw (120:\rOne) -- (100:\rTwo); \fill (120:\rOne) circle (1.5pt); \fill (100:\rTwo) circle (1.5pt);
 \draw (120:\rOne) -- (140:\rTwo); \fill (140:\rTwo) circle (1.5pt);
 
 \draw (60:\rOne) -- (70:\rTwo); \fill (70:\rTwo) circle (1.5pt); \fill (60:\rOne) circle (1.5pt);
 \draw (60:\rOne) -- (50:\rTwo); \fill (50:\rTwo) circle (1.5pt);
 \draw (20:\rOne) -- (00:\rTwo); \fill (00:\rTwo) circle (1.5pt); \fill (20:\rOne) circle (1.5pt);
 \draw (20:\rOne) -- (30:\rTwo); \fill (30:\rTwo) circle (1.5pt);
 \fill (-10:\rOne) circle (1.5pt);

 \draw (V) -- (160:\rTwo); \fill (160:\rTwo) circle (1.5pt); % Dot at intersection
 % --- Lines connecting between circles for P, Q, R ---
 % These connect points on rTwo to points on rThree, then rThree to rFour
 % P

 \fill (170:\rTwo) circle (1.5pt); % Dot on rTwo

 % Q (The X1, X2 points are on this spoke)

 \fill (210:\rTwo) circle (1.5pt); % Dot on rTwo
 \fill (160:\rTwo) circle (1.5pt); % Dot on rThree
 % R
 \fill (240:\rTwo) circle (1.5pt); % Dot on rTwo
 \fill (180:\rTwo) circle (1.5pt); % Dot on rThree
 
 \fill (220:\rTwo) circle (1.5pt); \fill (250:\rTwo) circle (1.5pt); \fill (260:\rTwo) circle (1.5pt); \fill (270:\rTwo) circle (1.5pt); \fill (310:\rTwo) circle (1.5pt); \fill (320:\rTwo) circle (1.5pt); \fill (195:\rTwo) circle (1.5pt); \fill (75:\rTwo) circle (1.5pt); \fill (80:\rTwo) circle (1.5pt); \fill (10:\rTwo) circle (1.5pt);
\fill (150:\rTwo) circle (1.5pt); \fill (130:\rTwo) circle (1.5pt); \fill (115:\rTwo) circle (1.5pt);

 % --- Labels for segments ---

 \fill (160:\rOne) circle (1.5pt);
 \fill (240:\rOne) circle (1.5pt);
 \fill (80:\rOne) circle (1.5pt);
 % On rTwo
 \fill (15:\rTwo) circle (1.5pt);

 % On rThree
 \fill (50:\rTwo) circle (1.5pt);
 \fill (300:\rTwo) circle (1.5pt);

\path (110:\rTwo) arc (110:70:\rTwo) 
 node[pos=0.5, sloped, above, fill=white, inner sep=1pt, font=\scriptsize] {$\partial X_2$};

\path (110:\rOne) arc (110:70:\rOne) 
 node[pos=0.35, sloped, above, fill=white, inner sep=1pt, font=\scriptsize] {$\partial X_1$};

\node[fill=black,circle,inner sep=0.2pt,label={[font=\scriptsize]below:$X_0$}] (X) at (0.1,0) {};

\end{tikzpicture}
\caption{Two layers of tiling of type $[4, 4, 5, 5]$}
\label{tiling4455}
\end{figure}

It turns out that the existence is straightforward for the class of all $6$-tuples (tuples of size $6$) and $5$-tuples without triangles. This is because of ample flexibility in permuting five or more faces to cover consecutive vertices along $\partial X_{k}$ while extending $X_k$ to $X_{k+1}$ irrespective of the tiling $X_i$ for $i<k-1$. A detailed construction of this is presented in Proposition \ref{pseudodeg5}, both for the sake of completeness and to prepare for the more intricate construction required later.
For $3$-tuples, the classification is also straightforward and known; see Proposition \ref{dg3} below. 
\par
However, it proved impossible to formulate a reasonably simple inductive hypothesis that applies to the remaining tuples that admit tilings, or to determine which tuples do not admit one, particularly when triangles are present. We therefore divide the tuples into subclasses according to the number of triangles they contain. For $5$-tuples with triangles and $4$-tuples without any triangles, a non-trivial inductive hypothesis is required for the layer-by-layer construction; see Proposition \ref{pseudo4prop}. This hypothesis must take into account not only the current layer and fans around each vertex but also two to three consecutive layers. 
\par
 The remaining -- and certainly the most interesting -- case of $4$-tuples with triangles is classified in Theorem \ref{431prop}. In this case, we take a two-pronged approach. On the one hand, we attempt to formulate an inductive hypothesis for the layer-by-layer construction; on the other hand, we track the combinatorial constraints that arise in constructing neighborhoods around each of the tiles as successive layers are added. Here, a \textit{neighborhood} of a face $P$ in a tiling is the cyclic sequence of faces adjacent to $P$ up to a cyclic rotation and reversal, and they are expressed as cyclic sequence of the sizes of the faces. Our approach fortunately yields a complete classification without requiring verification over too many layers. 
 \\
 For the existence part, in the special case of tuples with a single triangle (e.g., for tuples $[3, 5, p, p]$, $[3, 5, k_3, k_4]$), we augment the layers -- calling this the \textit{$t$-layer} -- to include all the faces adjacent to the triangles that are adjacent to the boundary of the previous layers. This helps us to formulate a significantly simpler inductive hypothesis for this class. For the remaining cases, the existence is established using known results on homogeneous tilings and applying basic tiling operations such as rectification, truncation, or suitable combinations thereof. 
\par
We then complete the classification of protosets of regular polygons by proving the following theorem in \S \ref{regtile}.
 
 \begin{theorem}\label{solv1} If a protoset of regular polygons admits a tiling of $\mathbb{H}^2$ then it has a subset that admits a pseudo-homogeneous tiling.
 \end{theorem}
 The proof is established by showing that if two vertex types, neither of which admits a tiling on its own, are combined (in a tiling of mixed vertex types), the resulting mixed-type arrangement also does not admit a tiling. This failure is due to either a combinatorial impasse or a mismatch in the side lengths of the associated regular polygons. The statement and proof method in this theorem can be extended to include apeirogons; however, due to the technical complexities, we defer this discussion to a future article.

\subsection{Aperiodicity in $\mathbb{H}^2$}
In \S \ref{aperiodic}, we examine an important variant of the domino problem in $\mathbb{H}^2$: the periodic domino problem, which asks whether a given protoset of tiles can produce a periodic tiling. In the hyperbolic plane, there are two distinct notions of periodicity: weak and strong. A tiling of $\mathbb{H}^2$ is called \textit{strongly periodic} if it quotients to a compact domain under the action of its symmetry group, and \textit{weakly periodic} if its symmetry group contains a subgroup of infinite cyclic symmetry \cite{GS05}. In $\mathbb{E}^2$, however, these two notions coincide. A set of hyperbolic tiles is called \textit{weakly aperiodic} (or strongly aperiodic) if none of the tilings of $\mathbb{H}^2$ by isometric copies of them is strongly periodic (respectively, weakly periodic).
\par
While numerous examples of protosets that admit periodic tilings can be found in the literature \cite{GS79}, these are typically vertex-transitive (also known as uniform) constructions, which are necessarily homogeneous. Only recently have examples of regular polygons that admit only tilings with multiple vertex orbits been presented in the author's work \cite{AM20} and in numerous informal notes of Marek \v{C}trn\'{a}ct. Nevertheless, the existence of a weakly aperiodic protoset of regular polygons in $\mathbb{H}^2$  has remained a long-standing open question in the field.
 \par
Here, in \S \ref{aperiodic}, we employ a double-counting argument to first show that there does not exist any weakly periodic tiling of type $[3, 5, k_3, k_4]$ for $10 \leq k_3 < k_4$, $k_3, k_4 \neq 11$. More precisely, the argument compares two different counts of incidences between triangles and pentagons that would arise in any strongly periodic tiling of this type. Combined with Theorem \ref{431prop}, which already provides an inductive construction of such tilings, this yields a bi-infinite family of aperiodic protosets of regular tiles: 
\[\{3, 5, k_3, k_4\}, 10 \leq k_3 <k_4, k_3 \neq 11\]
with side length chosen so that they form a complete configuration around a vertex.
\\
Tiling spaces associated to aperiodic tile sets are important objects in symbolic dynamics \cite{B13} non-commutative geometry \cite{KP00, OYO11} and in physics \cite{BDF20, ZY22, LZ22}. A few properties of the tiling space associated to aperiodic tile set $\{3, 5, k_3, k_4\}$ is studied in \S \ref{tilingspace} and also in Appendix \S \ref{appen1} for a representative class ($[3, 5, 12, 14]$). 
\par
In \S\ref{conclu}, we present a brief discussion of the implications of our results together with additional observations concerning a few unresolved problems.

\section{Pseudo-homogeneous tilings}\label{planar}
Let us first introduce a few definitions and notations that will be used throughout this paper. 
\begin{definition} \label{fan}
For a $d$-tuple $\mathfrak{k}=[k_1, k_2,$ $ \cdots, k_d]$, a \textit{fan} of type $\mathfrak{k}$ around a vertex $v$ is a configuration of $d$ faces (in any order) around the vertex with side counts $ k_1, k_2,$ $ \cdots, k_d$, such that all the edges incident to $v$ are shared by two faces. A \textit{partial fan} around $v$ of type $\mathfrak{k}$ is a configuration of faces around $v$ such that all but two edges incident to $v$ are shared by two faces, and can be extended to a fan of type $\mathfrak{k}$ around the vertex.
\end{definition}
The boundary $\partial X_k$ of the $k$-th layer $X_k$ consists of a cycle of edges. Consequently, the vertices on $\partial X_k$ have a natural cyclic ordering.

\begin{definition}\label{innerlayer}
Let $C_k$ denote the $k$-th corona of the tiling, defined as the cyclic sequence of faces in $X_k$ that share at least one vertex (or edge) with the boundary $\partial X_k$. Note that $X_k = X_{k-1} \cup C_k$.
\end{definition}

\begin{definition} \label{indegree}
We say that a vertex $x$ on $\partial X_k$ has \textit{in-degree} $r$ if the number of edges of the form $(x, y)$ with $y \in X_{k} \setminus \partial X_{k}$ is $r$. A vertex of in-degree $0$ will usually be referred to as a \textit{free} vertex.
\end{definition}

In what follows, we describe a procedure to extend the layer $X_k$ to $X_{k+1}$ by starting with a fan around a suitably chosen vertex on $\partial X_k$, and then (circularly) inductively construct fans around all the vertices on $\partial X_k$.

\subsection{Tilings of degree $\geq 5$ } 
 \begin{proposition} \label{pseudodeg5} 
For $d \geq 5$ and a $d$-tuple $\mathfrak{k}=[k_1, k_2,$ $ \cdots, k_d]$ with $\vartheta(\mathfrak{k}) \geq 2 $, there exists a tiling of the plane of type $\mathfrak{k}$.
\end{proposition} 
\begin{proof}

\begin{figure}[ht!]
\tikzstyle{ver}=[]
\tikzstyle{vert}=[circle, draw, fill=black!100, inner sep=0pt, minimum width=4pt]
\tikzstyle{vertex}=[circle, draw, fill=black!00, inner sep=0pt, minimum width=4pt]
\tikzstyle{edge} = [draw,thick,-]
 \begin{minipage}{.60\textwidth} 
\centering
\begin{tikzpicture}[scale=0.20]
%\begin{scope}[shift={(-50,13)}]

\draw[edge, thick](5,5)--(32,5);
\draw[edge, thick](2,10)--(32,10);
\draw[edge, thick](2,15)--(32,15);

\draw[edge, thick](14, 5)--(5,10);
\draw[edge, thick](21, 5)--(15,10);
\draw[edge, thick](21, 5)--(27,10);
\draw[edge, thick](30, 5)--(27,10);
\draw[edge, thick](15, 10)--(20,15);
\draw[edge, thick](27, 10)--(30,15);
\draw[edge, thick](15, 10)--(10,15);
\draw[edge, thick](15, 10)--(15,15);

\node[ver] () at (16.5, 12.5){\scriptsize $S$};
\node[ver] () at (9.5, 12.5){\scriptsize ${U_{3}}$};
\node[ver] () at (23.5, 12.5){\scriptsize ${T}$};
\node[ver] () at (14, 12.5){\scriptsize ${R}$};
\node[ver] () at (8.0, 9.3){\scriptsize ${v_{i-1}}$};
\node[ver] () at (14.3, 9.4){\scriptsize ${v_{i}}$};
\node[ver] () at (29, 9.3){\scriptsize ${v_{i+1}}$};
\node[ver] () at (26.5, 7.5){\scriptsize ${W_1}$};
\node[ver] () at (30.5, 7.5){\scriptsize ${W_2}$};
\node[ver] () at (20.8,8){\scriptsize ${U_1}$};
\node[ver] () at (12.8,8){\scriptsize ${U_2}$};
\node[ver] () at (34,10){\scriptsize ${\partial X_{k}}$};
\node[ver] () at (34.5,5){\scriptsize ${\partial X_{k-1}}$};
\node[ver] () at (34.5, 15){\scriptsize ${\partial X_{k+1}}$};

\end{tikzpicture}
\captionof{figure}{Layer construction for $d=6$, case (a)}
\label{fig:pcu} 
\end{minipage}%
 \begin{minipage}{.45\textwidth} 
 \hspace{0.5cm}
\begin{tikzpicture}[scale=0.2]

\draw[edge, thick](8,5)--(30,5);
\draw[edge, thick](8,10)--(31,10);
\draw[edge, thick](8,15)--(32,15);

\draw[edge, thick](10, 5)--(15,10);
\draw[edge, thick](21, 5)--(15,10);
\draw[edge, thick](21, 5)--(27,10);
\draw[edge, thick](27, 10)--(30,15);

\draw[edge, thick](15, 10)--(10,15);
\draw[edge, thick](15, 10)--(20,15);

\node[ver] () at (15.5, 12.5){\scriptsize $S$};
\node[ver] () at (9.5, 12.5){\scriptsize ${U_{4}}$};
\node[ver] () at (23.5, 12.5){\scriptsize ${T}$};
\node[ver] () at (13.3, 9.4){\scriptsize ${v_{i}}$};
\node[ver] () at (28.5, 9.2){\scriptsize ${v_{i+1}}$};
\node[ver] () at (26.5, 7.5){\scriptsize ${W_1}$};
\node[ver] () at (20.8,8){\scriptsize ${U_1}$};
\node[ver] () at (14.8,8){\scriptsize ${U_2}$};
\node[ver] () at (9.8,8){\scriptsize ${U_3}$};
\node[ver] () at (34,10){\scriptsize ${\partial X_{k}}$};
\node[ver] () at (34,5){\scriptsize ${\partial X_{k-1}}$};
\node[ver] () at (34.5, 15){\scriptsize ${\partial X_{k+1}}$};

\end{tikzpicture}

\captionof{figure}{Case (b)-part-1}
\label{fig:6l34}

\end{minipage}

\end{figure}

\textbf{Degree $\geq 6$.} Let $\mathfrak{k}=[k_1, k_2, k_3, k_4, k_5, k_6]$ be a $6$-tuple. We may assume that $k_i \geq 4$ for at least one $i$, as the existence of a tiling of type $[3^6]$ is trivial. 
\newline
Assume inductively that $X_k$ is a tiling of type $\mathfrak{k}$ on $D^k(X_0)$. By induction on layers, the possible in-degrees of a vertex on $\partial X_k$ are $0$, $1$, or $2$. Moreover, either there is at least one free vertex (say $u$) or there are two consecutive vertices (say $\{v_1, v_2\}$) of in-degree $1$ on $\partial X_k$. In the first case, we construct a fan around the vertex succeeding $u$ on $\partial X_k$; in the latter, we construct a fan around a vertex following the edge $\{v_1, v_2\}$. 
\\
Let $U_1, U_2, U_3, U_4, U_5, U_6$ denote faces of sizes $k_1, k_2, k_3, k_4, k_5, k_6$, respectively. Assume that we have a complete fan around $v_{i-1} \in \partial X_k$ and an induced partial fan around the succeeding vertex $v_{i} \in \partial X_k$. Then either \textbf{(a)} the in-degree of $v_i$ is $1$ (see Fig.~\ref{fig:pcu}), in which case, we may assume, without loss of generality, that the partial fan around $v_i$ is of the form $[k_1, k_2, k_3, \cdots]$, or \textbf{(b)} the in-degree of $v_i$ is $2$ (see Fig.~\ref{fig:6l34}) and the partial fan around $v_i$ of the form $[k_1, k_2, k_3, k_4, \cdots]$. 
\\
In case (a), if further the in-degree of $v_{i+1}$ is $2$ (see Fig.~\ref{fig:pcu}), then we can choose faces $R$, $S$ and $T$ so that $T \notin \{U_1, U_2, U_3, W_1, W_2\} $. With this choice, the fan around $v_i$ determines a partial fan around $v_{i+1}$. If the in-degree of $v_{i+1}$ is $1$, there is greater flexibility in selecting $R$, $S$, and $T$ so as to obtain a partial fan at $v_{i+1}$. 
\\
In case (b), if the in-degree of $v_{i+1}$ is $1$, then we can choose faces $S$, $T$ such that $\{W_1, U_1, T\}$ is a partial fan around $v_{i+1}$ (see Fig.~\ref{fig:6l34}). If $v_{i+1}$ has in-degree $\leq 4$ then we observe that $U_1$, $U_2$ and $W_1$ must all be triangles (see Fig.~\ref{fig:6l44}). Thus, treating $W_1$ in the role of $U_2$, we may select $S$ and $T$ so that $\{W_2, W_1, U_1, T\}$ becomes a partial fan at $v_{i+1}$.
\\
Due to the choice of the starting vertex, the above induction can be continued up to the final vertex on $\partial X_k$. This completes the induction when the in-degree of $v_{i}$ is $1$ or $2$. The case when $v_i$ is free is simpler and is treated similarly.

 \begin{figure}[ht!]
\tikzstyle{ver}=[]
\tikzstyle{vert}=[circle, draw, fill=black!100, inner sep=0pt, minimum width=4pt]
\tikzstyle{vertex}=[circle, draw, fill=black!00, inner sep=0pt, minimum width=4pt]
\tikzstyle{edge} = [draw,thick,-]
\centering
\begin{tikzpicture}[scale=0.2]

\draw[edge, thick](8,5)--(32,5);
\draw[edge, thick](8,10)--(32,10);
\draw[edge, thick](8,15)--(32,15);

\draw[edge, thick](10, 5)--(15,10);
\draw[edge, thick](21, 5)--(15,10);
\draw[edge, thick](21, 5)--(27,10);
\draw[edge, thick](27, 10)--(30,15);
\draw[edge, thick](30, 5)--(27,10);

\draw[edge, thick](15, 10)--(10,15);
\draw[edge, thick](15, 10)--(20,15);

\node[ver] () at (15.5, 12.5){\scriptsize $S$};
\node[ver] () at (9.5, 12.5){\scriptsize ${U_{4}}$};
\node[ver] () at (23.5, 12.5){\scriptsize ${T}$};
\node[ver] () at (13.3, 9.4){\scriptsize ${v_{i}}$};
\node[ver] () at (29.1, 9.2){\scriptsize ${v_{i+1}}$};
\node[ver] () at (26.5, 7.5){\scriptsize ${V_1}$};
\node[ver] () at (30.5, 7.5){\scriptsize ${V_2}$};
\node[ver] () at (20.8,8){\scriptsize ${U_1}$};
\node[ver] () at (15, 7.5){\scriptsize ${U_2}$};
\node[ver] () at (9.8,8){\scriptsize ${U_3}$};
\node[ver] () at (34,10){\scriptsize ${\partial X_{k}}$};
\node[ver] () at (34,5){\scriptsize ${\partial X_{k-1}}$};
\node[ver] () at (34.5, 15){\scriptsize ${\partial X_{k+1}}$};

\end{tikzpicture}
\captionof{figure}{Degree-6: case (b), part-2 }
\label{fig:6l44} 

\end{figure}
A construction similar to the $6$-tuple case can be performed for a $d$-tuple with $d>6$; the higher degree provides additional flexibility in permuting faces within a fan.
\par
\textbf{Degree 5.} For a $5$-tuple $\mathfrak{k}=[k_1, k_2, k_3, k_4, k_5]$ with either $k_3$ or $k_4 \geq 4$, we follow the inductive construction used in degree $6$ case. Given a tiling of type $\mathfrak{k}$ on $X_k$, the only possible in-degrees of a vertex on $\partial X_k$ is $0$ or $1$ (the possibility of degree $2$ and above is ruled out by the absence of triangles in $\mathfrak{k}$). Moreover, there must exist at least one free vertex, say $v$. Starting with a fan around $v$, the $(k+1)$-layered tiling $X_{k+1}$ can be constructed as in degree $6$ case, completing the induction.
\newline
For $5$-tuples with triangles, we divide the discussion into several subcases in terms of the number of triangles present. 
\\
The tuple $[3^5]$ is ruled out by the angle-sum condition.
\newline
For types $[3^3, k_4, k_5]$ and $[3^4, k_5]$ satisfying the angle-sum condition, we note that there is a unique homogeneous tiling (indeed, a vertex-transitive one) of type $[3^2, k_4, 3, k_5]$ for all $k_4, k_5 \geq 4$. The same construction also applies when $k_4=3$ (with $k_5 \geq 6$), yielding a tiling of type $[3^4, k_5]$. 
\par
For types $[3^2, k_3, k_4, k_5]$ with $k_3, k_4, k_5 >3$, the possible in-degrees of a vertex on $\partial X_k$ is $0$, $1$ or $2$. A vertex of in-degree $3$ on $\partial X_k$ can arise only from a vertex of type $[3^3, \dots ]$, $\partial X_{k-1}$( see Fig. \ref{fig:deg53}).

\begin{figure}[ht!]
\tikzstyle{ver}=[]
\tikzstyle{vert}=[circle, draw, fill=black!100, inner sep=0pt, minimum width=4pt]
\tikzstyle{vertex}=[circle, draw, fill=black!00, inner sep=0pt, minimum width=4pt]
\tikzstyle{edge} = [draw,thick,-]
\begin{minipage}{.50\textwidth} 
\centering
\begin{tikzpicture}[scale=0.2]

\draw[edge, thick](12,5)--(25,5);
\draw[edge, thick](12,10)--(25,10);
\draw[edge, thick](12,15)--(25,15);

\draw[edge, thick](15, 5)--(19,10);
\draw[edge, thick](23, 5)--(19,10);

\draw[edge, thick](15, 10)--(19,15);
\draw[edge, thick](19, 10)--(19,15);
\draw[edge, thick](23, 10)--(19,15);

\node[ver] () at (20.4, 14.3){\scriptsize ${v_{i}}$};
\node[ver] () at (21.4, 9.3){\scriptsize ${v_{i+1}}$};
\node[ver] () at (27.5,10){\scriptsize ${\partial X_{k-1}}$};
\node[ver] () at (27.5,5){\scriptsize ${\partial X_{k-2}}$};
\node[ver] () at (27.5, 15){\scriptsize ${\partial X_{k}}$};

\end{tikzpicture}
\captionof{figure}{in-degree for type $[3^2, k_3, k_4, k_5]$}
\label{fig:deg53}

\end{minipage}%
\begin{minipage}{.50\textwidth} 
\centering
\begin{tikzpicture}[scale=0.2]

\draw[edge, thick](10,5)--(30,5);
\draw[edge, thick](10,10)--(30,10);
\draw[edge, thick](10,15)--(30,15);

\draw[edge, thick](14, 5)--(16,10);
\draw[edge, thick](19, 5)--(16,10);
\draw[edge, thick](22, 5)--(24,10);
\draw[edge, thick](24, 10)--(26,15);
\draw[edge, thick](28, 10)--(26,15);
\draw[edge, thick](26, 5)--(24,10);

\draw[edge, thick](16, 10)--(13,15);

\draw[edge, thick](12, 10)--(13,15);

\node[ver] () at (19, 15){\scriptsize $\bullet$};

\node[ver] () at (15, 14.3){\scriptsize ${v_{i-1}}$};
\node[ver] () at (19, 14){\scriptsize ${v_{i}}$};
\node[ver] () at (24.3, 14.3){\scriptsize ${v_{i+1}}$};
\node[ver] () at (32.5,10){\scriptsize ${\partial X_{k-1}}$};
\node[ver] () at (32.5,5){\scriptsize ${\partial X_{k-2}}$};
\node[ver] () at (32, 15){\scriptsize ${\partial X_{k}}$};

\end{tikzpicture}
\captionof{figure}{Type $[3^2, k_3, k_4, k_5]$: case (a)}
\label{deg5l31y} 
\end{minipage}

\end{figure}

 Moreover, there is at least one free vertex on $\partial X_k$ along with either the preceding or the succeeding vertex (or both) of in-degree $1$. This assertion is verified by noting that a free vertex $v_{i}$ flanked by two adjacent vertices of in-degree $2$ implies either \textbf{(a)} two consecutive vertices of in-degree $4$ on $\partial X_{k-1}$ (see Fig.~\ref{deg5l31y}) or \textbf{(b)} a fan of type $[4, 4, 3, 4, 3]$ around a vertex (say $y$) on $\partial X_{k-1}$ (see Fig.~\ref{fig:deg532y}). It can be shown that the case (a) cannot arise by arguing inductively on the layers. For the case (b), it can be shown that there exists a homogeneous tiling of type $[4, 4, 3, 4, 3]$. With these observations, it is straightforward to verify that there is no further obstruction in constructing layer $X_{k+1}$ by induction, as established for degree $6$ case above.

\begin{figure}[ht!]
\tikzstyle{ver}=[]
\tikzstyle{vert}=[circle, draw, fill=black!100, inner sep=0pt, minimum width=4pt]
\tikzstyle{vertex}=[circle, draw, fill=black!00, inner sep=0pt, minimum width=4pt]
\tikzstyle{edge} = [draw,thick,-]
\centering
\begin{tikzpicture}[scale=0.2]

\draw[edge, thick](10,5)--(32,5);
\draw[edge, thick](10,10)--(32,10);
\draw[edge, thick](10,15)--(32,15);

\draw[edge, thick](14, 5)--(14,10);
\draw[edge, thick](22, 5)--(22,10);
\draw[edge, thick](22, 10)--(26,15);
\draw[edge, thick](29, 10)--(26,15);
\draw[edge, thick](29, 5)--(29,10);

\draw[edge, thick](22, 10)--(18,15);

\draw[edge, thick](14, 10)--(18,15);

\node[ver] () at (22, 15){\scriptsize $\bullet$};

\node[ver] () at (15.7, 14.3){\scriptsize ${v_{i-1}}$};
\node[ver] () at (22, 14){\scriptsize ${v_{i}}$};
\node[ver] () at (28, 14.3){\scriptsize ${v_{i+1}}$};
\node[ver] () at (21, 9){\scriptsize ${y}$};

\node[ver] () at (34.5,10){\scriptsize ${\partial X_{k-1}}$};
\node[ver] () at (34.5,5){\scriptsize ${\partial X_{k-2}}$};
\node[ver] () at (34, 15){\scriptsize ${\partial X_{k}}$};

\end{tikzpicture}

\captionof{figure}{Type $[3^2, k_3, k_4, k_5]$: case (b)}
\label{fig:deg532y}

\end{figure}

For types $[3, k_2, k_3, k_4, k_5]$ with $k_i \geq 4$ for all $i$, the possible in-degree of a vertex on $\partial X_k$ is $0$, $1$, or $2$. Moreover, there is at least one vertex of degree $2$ on $\partial X_k$ with at least one of its two adjacent vertices having degree $3$; otherwise, it leads to a vertex type of the form
$[3^2,..]$ on $\partial X_{k-1}$ as shown in Fig.~\ref{deg5l31y}. Thus the situation is equivalent to constructing tilings for tuples with only two triangles.

\end{proof}

\subsection{Tilings of degree $4$.} \label{deg4pseudo} We will first prove that the above approach can also be used to construct tilings for $4$-tuples without triangles. To do so, we require the following lemma to address a special class of tuples. 
 \begin{lemma}\label{twoeven}There exists a tiling of type $\mathfrak{k}=[2p, q, 2p, s]$, $\alpha (\mathfrak{k}) \geq 2$, on a closed surface, and hence on the plane.
\end{lemma}
\begin{proof} By Theorem 1.3 of \cite{EEK2}, there exists a homogeneous tiling of type $\mathfrak{k}=[2p, 2q, 2s]$ for $\alpha (\mathfrak{k}) \geq 2$ on an orientable closed surface. For $ q, s \geq 3$, we can contract common edges of the $2q$-gons and $2s$-gons of this tiling to a point to obtain a homogeneous tiling of type $[2p, q, 2p, s]$.
\end{proof}
 \begin{proposition} \label{pseudo4prop} 
For a cyclic-tuple $\mathfrak{k}=[k_1, k_2, k_3, k_4]$ with $k_i \geq 4$ for all $i$ and $\vartheta(\mathfrak{k}) \geq 2 $, there exists a pseudo-homogeneous tiling on the plane of type $\mathfrak{k}$.
\end{proposition} 
\begin{proof}
The construction of tilings for tuples $\mathfrak{k}$ with $k_i \geq 5$ is straightforward using the method described in Proposition \ref{pseudodeg5} for degree greater than $5$. This relies on the fact that the in-degree of any vertex on $\partial X_k$ is either $0$ or $1$, and furthermore, no two consecutive vertices have degree $3$. For a tuple $\mathfrak{k}$ with at least two of the $k_i$'s equal to $4$, the existence of a homogeneous tiling is guaranteed by Lemma \ref{twoeven}.
\newline
For tuples $\mathfrak{k}$ with $k_1=4$ and $k_2, k_3, k_4>4$, the in-degrees of the vertices on $\partial X_k$ are either $0$ or $1$. Moreover, we have one of the following two scenarios: \textbf{(a)} no two consecutive vertices on $\partial X_k$ have in-degree $1$ \textbf{(b)} there are vertices of a $4$-gon as illustrated in Fig.~\ref{fig:deg4l} by vertices $v_i$ and $v_{i+1}$. In the latter case, these vertices of in-degree $1$ are preceded and followed by vertices of in-degree $0$, $v_{i-1}$ and $v_{i+2}$ in the figure. For case (a), it is straightforward to verify that there is no obstacle in extending $X_k$ to $X_{k+1}$ inductively as done for degree $\geq 5$ above. For case (b), a potential obstacle may arise if a fan around a vertex $v_{i-1}$ completely determines the fan around the succeeding vertex $v_i$, which in turn determines a sequence of faces around the vertex $v_{i+1}$ that does not constitute a partial fan. This situation, however, can be avoided by permuting the faces appropriately in the fan of the free vertex $v_{i-2}$ preceding $v_{i-1}$(as $k_3>4$); more precisely, by permuting $P$ and the $k_2$-gon, and consequently the $k_4$ and $k_2$-gon in the fan of $v_{i}$.

\begin{figure}[ht!]
\tikzstyle{ver}=[]
\tikzstyle{vert}=[circle, draw, fill=black!100, inner sep=0pt, minimum width=4pt]
\tikzstyle{vertex}=[circle, draw, fill=black!00, inner sep=0pt, minimum width=4pt]
\tikzstyle{edge} = [draw,thick,-]
\centering
\begin{tikzpicture}[scale=0.2]
%\begin{scope}[shift={(-50,13)}]

\draw[edge, thick](5,5)--(35,5);
\draw[edge, thick](2,10)--(35,10);
\draw[edge, thick](2,15)--(35,15);

\draw[edge, thick](20, 5)--(5,10);
\draw[edge, thick](20, 5)--(20,10);

\draw[edge, thick](20, 10)--(20,15);
\draw[edge, thick](30, 10)--(32,15);

\draw[edge, thick](16, 10)--(11,15);
\draw[edge, thick](16, 10)--(15,15);

\draw[edge, thick](10, 10)--(9,15);
\draw[edge, thick](10, 10)--(5,15);

\draw[edge, thick](30, 5)--(30,10);

\node[ver] () at (18.5, 12.5){\scriptsize ${K_{4}}$};

\node[ver] () at (23.5, 12.5){\scriptsize ${K_{2}}$};

\node[ver] () at (11.5, 12.5){\scriptsize ${K_{2}}$};
\node[ver] () at (8.6, 13){\scriptsize ${P}$};
\node[ver] () at (14.5, 13){\scriptsize ${4}$};

\node[ver] () at (9.5, 9.3){\scriptsize ${v_{i-2}}$};
\node[ver] () at (15.5, 9.3){\scriptsize ${v_{i-1}}$};

\node[ver] () at (28.5, 9.3){\scriptsize ${v_{i+1}}$};
\node[ver] () at (21, 9.3){\scriptsize ${v_{i}}$};

\node[ver] () at (25.5, 7.5){\scriptsize ${4}$};

\node[ver] () at (32.5, 7.5){\scriptsize ${K_{2}}$};

\node[ver] () at (16.8,8){\scriptsize ${K_{3}}$};

\node[ver] () at (37.5,10){\scriptsize ${\partial X_{k}}$};
\node[ver] () at (37,5){\scriptsize ${\partial X_{k-1}}$};
\node[ver] () at (37.5, 15){\scriptsize ${\partial X_{k+1}}$};

\end{tikzpicture}
\captionof{figure}{An obstacle for layer construction for type $[4, k_2, k_3, k_4]$}
\label{fig:deg4l}
\end{figure}

\end{proof}
The classification of $4$-tuples with triangles is considerably more complex as we shall see below. For some classes, we use known results to establish the existence of tilings, as shown in the following examples.

\begin{example} \label{tiling3369} We construct a weakly periodic tiling of type $[3, 3, 6, 9]$ on $\mathbb{H}^2$. First, there exists a weakly periodic homogeneous tiling $\mathcal{T}$ of type $[4, 6, 4, 6]$; see Theorem 1.3 in \cite{EEK2}. The dual of $\mathcal{T}$ is obtained by successive reflections of a single quadrilateral across its sides. Note that the edges of a quadrilateral admit a proper two-label edge labelling in which adjacent edges receive distinct labels. Since the degree of $\mathcal{T}$ is even, the reflections preserve the labeling. Consequently, each face of the dual tiling of $\mathcal{T}$ inherits the two-labelling, and so do the faces of $\mathcal{T}$ themselves. This is illustrated in Figure~\ref{fig:3369}, where the two labels are indicated by different line styles (dotted versus solid). 
\begin{figure}[H] 
\begin{minipage}{.45\textwidth} 
\centering
	\begin{tikzpicture}[scale=0.5,
 x=1cm, y=1cm, line cap=round, line join=round,
 Thick/.style={line width=0.6pt},
 Thin/.style={line width=0.5pt},
 Dashed/.style={dash pattern=on 1.5pt off 3.0pt}
]
 % Thick Dashed Lines
 
 \draw[Thick, Dashed] 
 (-1.0,-2.2) -- (2.5,-2.2)
 (1.1,2.2) -- (-1.0,3.3)
 (1.1,2.2) -- (3.7,3.0)
 (2.5,0.0) -- (4.9,0.8)
 (-3.7,-3.1) -- (-4.3,-1.8);

% Thick Solid Lines
\draw[Thick] 
 (-1.0,-2.2) -- (-1.0,0.0) -- (-3.0,1.3)
 (2.5,-2.2) -- (2.5,0.0) -- (1.1,2.2)
 (-3.0,2.8) -- (-1.0,3.3)
 (3.7,3.0) -- (4.9,0.8)
 (-3.7,-3.1) -- (-2.0,-3.3)
 (-4.3,-1.8) -- (-3.3,-0.7) -- (-4.9,0.4);

% Thin Dashed Lines
\draw[Thin, Dashed] 
 (2.5,0.0) -- (-1.0,0.0)
 (-3.0,1.3) -- (-3.0,2.8)
 (-1.0,-2.2) -- (-2.0,-3.3)
 (-3.0,1.3) -- (-4.9,0.4)
 (-3.3,-0.7) -- (-1.0,0.0);

\end{tikzpicture}
\captionof{figure}{Tiling of type $[4^2,6^2]$}
\end{minipage}%
\begin{minipage}{.45\textwidth}
\centering
\includegraphics[scale=0.40]{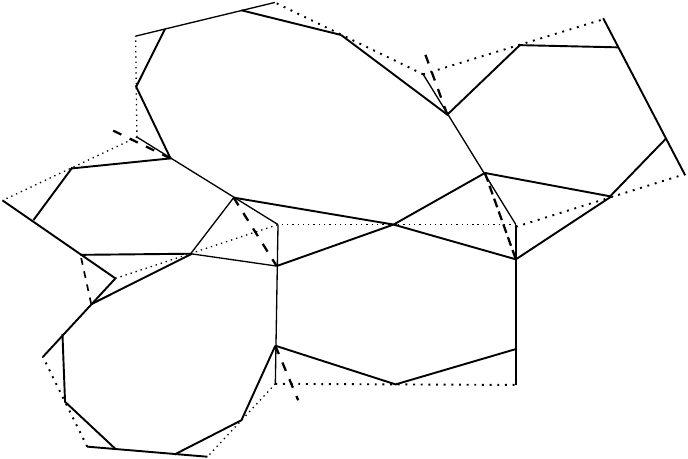}
\captionof{figure}{Tiling of type $[3^2,6, 9]$}
\label{fig:3369} 
\end{minipage}
\end{figure}
We then perform a sequence of operations on $\mathcal{T}$ resembling a rectification along the dotted edges and truncation along the solid edges. More precisely, we form quadrilaterals centered around each vertex of $\mathcal{T}$ by joining edges incident to them so that the vertices (of the quadrilaterals) on the dotted edges are precisely the midpoints and vertices on the solid edges lie before the midpoint; see Fig.~\ref{fig:3369}. Next, we remove the dotted edges together with the original vertices of $T$ and edges incident to them inside the quadrilaterals. Finally, we introduce edges between opposite vertices of the quadrilaterals that have degree less than $4$, shown by dashed edges in the figure. This produces a periodic tiling of type $[3,3,6,9]$ on $\mathbb{H}^2$. 
\end{example}
 
\begin{corollary} \label{33pq} Using the same method as above, we can construct a homogeneous tiling of type $[3, 3, 3p, 3q]$ from a homogeneous tiling of type $[2p, 2q, 2p, 2q]$ for $p, q \geq 2$ on $\mathbb{H}^2$.

\end{corollary}
Before establishing the classification of 4-tuples, we make two fundamental observations regarding the valid neighborhoods of triangles and pentagons. These constraints will be crucial for the subsequent proofs.

\begin{observation} \label{A1}
For a tiling of type $[3, k_2, k_3, k_4]$ with $3 < k_2 < k_3 < k_4$, the neighborhoods of all triangles are the same. In particular, for a valid neighborhood of a $k_2$-gon with $m$ $\Delta$-triangles around it, $k_2-m$ must be even, and likewise for $k_3$- and $k_4$-gon.

\end{observation}
\begin{observation} \label{A2} For a tiling of type $[3, 5, k_3, k_4]$ with $5<k_3<k_4$, for a valid neighborhood, all pentagons must be of type $4$. This implies, in particular, that there cannot be $\nabla$-triangles attached to a $k_3$-gon (or a $k_4$-gon) on two consecutive vertices which are also incident to an edge of a pentagon; see Fig.~\ref{3triangle}.
\end{observation}
\begin{figure}[H] 
\tikzstyle{ver}=[]
\tikzstyle{vert}=[circle, draw, fill=black!100, inner sep=0pt, minimum width=4pt]
\tikzstyle{vertex}=[circle, draw, fill=black!00, inner sep=0pt, minimum width=4pt]
\tikzstyle{edge} = [draw,thick,-]
\centering
\begin{tikzpicture}[scale=0.10]
\draw[edge, thick](1,35)--(35,35);
\draw[edge, thick](5,35)--(1,42)--(9,42)--(5,35);
\draw[edge, thick](16,35)--(12,42)--(18.5,42)--(16,35);
\draw[edge, thick](27,35)--(24.5,42)--(31,42)--(27,35);
\draw[edge, thick](13,47)--(18.5,42)--(21.5,47)--(24.5,42)--(29,47);
\draw[edge, thick](3,35)--(3,28);
\draw[edge, thick](33,35)--(33,28);
\node[ver] () at (21.5,37){\scriptsize ${5}$};
\node[ver] () at (21.5,48.2){\scriptsize ${v}$};
\node[ver] () at (10.5,38){\scriptsize ${ k_{4}}$};
\node[ver] () at (18,30){\scriptsize ${ k_{3}}$};
\node[ver] () at (18.2,45.2){\scriptsize ${ k_{4}}$};
\node[ver] () at (24.7,45.2){\scriptsize ${ k_{4}}$};
\end{tikzpicture}
\captionof{figure}{Consecutive $\nabla$-triangles for type $[3, 5, k_3, k_4]$}
\label{3triangle} 
\end{figure}

 \begin{theorem} \label{431prop} For a $4$-tuple $ \mathfrak{k}$ with $\vartheta(\mathfrak{k}) \geq 2$ with at least one triangle, there exists a tiling on the plane of type $\mathfrak{k}$ if and only if $\mathfrak{k}$ is none of the following with $5<k_3 < k_4$:
 \newline
 1. $[3, 3, k_3, k_4]$, $k_3$ or $k_4$ is not a multiple of $3$;
\newline
2. a) $[3, 4, 5, 5]$, $[3, 4, 6, 8]$ ; b) $[3, 4, k_3, k_4]$, $k_3$ or $k_4$ is not even;
\newline
3. a) $[3, 5, k_3, k_4]$, $k_3 \leq 11$, $k_3 \neq 10$. 
\end{theorem}

\begin{proof}
We first prove the \enquote{only if} direction of the theorem, i.e., none of the types (1)--(3) listed in the statement admits a tiling.
\par
\textbf{Type $[3, 3, k_3, k_4]$.} A vertex in any tiling of this type can only be one of the two cyclic vertex-types: $[3^2, k_3, k_4]$ and $[3, k_3,3, k_4]$. By inspection, the neighborhood of a $k_3$-gon (and, a $k_4$-gon) in such a tiling has the same (cyclic) repeating pattern of fans of type $[3^2, k_3, k_4]$ around two consecutive vertices followed by a fan of type $[3, k_3, 3, k_4]$ around a third vertex. Hence the boundary edges of a $k_3$–gon are partitioned into blocks of length $3$, so $k_3$ is a multiple of $3$. The same argument gives $k_4 \equiv 0 \pmod 3$.
\par
\textbf{Type $[3, 4, 5, 5]$.} Type-$5$ pentagon is ruled out in such a tiling as we cannot form a neighborhood of it. We verify that a type-$4$ pentagon $P$ leads to an invalid fan in its second layer, as shown in Fig.~\ref{fig:3455}-(a) around the encircled vertex. So, there can only be type-$3$ pentagons in such a tiling. On the other hand, a type-$2$ quadrilateral $Q$ leads, without loss of generality, to a type-$4$ pentagon $P$, as shown in Fig.~\ref{fig:3455}-(b). With these constraints, starting with a type-$3$ pentagon, $P$, we arrive at an invalid fan around the encircled vertex shown in Fig.~\ref{fig:3455}-(c).

\begin{figure}[H]
\centering
\subfigure[Around a type-4 pentagon]{
\begin{tikzpicture}[scale=0.25,
 x=0.05cm, y=-0.05cm,
 line cap=round,
 line join=round,
 thick
 ]

\coordinate (P_Center) at (232,375);
\coordinate (A) at (186,359);
\coordinate (B) at (238,416);
\coordinate (C) at (290,376);
\coordinate (D) at (274,328);
\coordinate (E) at (215,317);
\coordinate (CircleVert) at (317,424);
\coordinate (F) at (193,286);
\coordinate (G) at (230,287);
\coordinate (H) at (285,289);
\coordinate (I) at (328,320);
\coordinate (J) at (343,388);
\coordinate (K) at (360,342);
\coordinate (L) at (391,303);
\coordinate (M) at (410,355);
\coordinate (N) at (386,412);
\coordinate (O) at (413,394);
\coordinate (P) at (365,450);
\coordinate (Q) at (252,450);
\coordinate (R) at (173,412);
\coordinate (S) at (296,262);
\coordinate (T) at (327,262);
\coordinate (U) at (352,282);
\coordinate (V) at (264,251);
\coordinate (W) at (231,259);

\draw (A) -- (B) -- (C) -- (D) -- (E) -- cycle;
\draw (E) -- (F) -- (G) -- (E);
\draw (D) -- (H) -- (I) -- (D);
\draw (C) -- (J) -- (CircleVert) -- (C);
\draw (CircleVert) -- (P) -- (N) -- (O) -- (M) -- (L) -- (K) -- (I);
\draw (K) -- (J);
\draw (J) -- (N);
\draw (CircleVert) -- (Q) -- (B);
\draw (B) -- (R) -- (A);
\draw (G) -- (H);
\draw (H) -- (S) -- (T) -- (U) -- (L);
\draw (I) -- (U);
\draw (S) -- (V) -- (W) -- (G);
\draw (K) -- (M);

\node[font=\bfseries] at (P_Center) {P};
\node[draw, circle, inner sep=3pt, thick] at (CircleVert) {};

\end{tikzpicture}
}\hspace{1cm} \centering
\subfigure[Around a type-2 quadrilateral]{
\begin{tikzpicture}[scale=0.25,
 x=0.05cm, y=-0.05cm,
 line cap=round,
 line join=round,
 thick
 ]

\coordinate (Q_TR) at (326,284);
\coordinate (Q_TL) at (266,284);
\coordinate (Q_BL) at (267,349);
\coordinate (Q_BR) at (324,348);
\coordinate (P_R) at (333,386);
\coordinate (P_Top) at (293,402);
\coordinate (P_L) at (246,385);
\coordinate (P_FarL_1) at (204,382);
\coordinate (P_FarL_2) at (192,352);
\coordinate (P_Ext_L) at (233,316);
\coordinate (P_Ext_R) at (365,315);
\coordinate (Out_R1) at (403,293);
\coordinate (Out_R2) at (414,350);
\coordinate (Out_R3) at (443,318);
\coordinate (Out_R4) at (403,252);
\coordinate (Out_R5) at (349,251);
\coordinate (Out_R6) at (399,402);
\coordinate (Out_L1) at (197,286);
\coordinate (Out_L2) at (269,246);
\coordinate (Out_L3) at (156,312);
\coordinate (Out_L4) at (214,226);
\coordinate (Tip_Bottom) at (313,220);
\coordinate (Tip_R_Top) at (381,225);
\coordinate (Tip_R_Far) at (446,291);
\coordinate (Tip_R_Low) at (431,383);
\coordinate (Tip_P_Top) at (330,410);
\coordinate (Tip_P_Left) at (219,404);
\coordinate (Label_Q) at (291,324);
\coordinate (Label_P) at (290,382);

\draw (Q_TR) -- (Q_TL) -- (Q_BL) -- (Q_BR) -- cycle;
\draw (Q_TR) -- (P_Ext_R) -- (Q_BR);
\draw (Q_TL) -- (P_Ext_L) -- (Q_BL);
\draw (Q_BL) -- (P_L);
\draw (Q_BR) -- (P_R) -- (P_Top) -- (P_L) -- (P_FarL_1) -- (P_FarL_2) -- (P_Ext_L) -- (Out_L1);
\draw (Q_TL) -- (Out_L2);
\draw (Q_TR) -- (Out_R5);
\draw (Out_L2) -- (Tip_Bottom) -- (Out_R5);
\draw (Out_R5) -- (Out_R4) -- (Out_R1) -- (P_Ext_R);
\draw (Out_R1) -- (Out_R3) -- (Out_R2) -- (P_Ext_R);
\draw (Out_R2) -- (Out_R6) -- (P_R);
\draw (Out_L1) -- (Out_L3) -- (P_FarL_2);
\draw (Out_L2) -- (Out_L4) -- (Out_L1);
\draw (Out_R4) -- (Tip_R_Top) -- (Out_R5);
\draw (Out_R1) -- (Tip_R_Far) -- (Out_R3);
\draw (Out_R2) -- (Tip_R_Low) -- (Out_R6);
\draw (P_R) -- (Tip_P_Top) -- (P_Top);
\draw (P_L) -- (Tip_P_Left) -- (P_FarL_1);

\node[font=\bfseries] at (Label_P) {P};
\node[font=\bfseries] at (Label_Q) {Q};

\end{tikzpicture}
} \hspace{1cm} \centering
\subfigure[Around a type-3 pentagon]{
\begin{tikzpicture}[scale=0.25,
 x=0.05cm, y=-0.05cm,
 line cap=round,
 line join=round,
 thick
 ]

\coordinate (C_TL) at (242,383);
\coordinate (C_TR) at (289,399);
\coordinate (C_BR) at (274,461);
\coordinate (C_BL) at (246,426);
\coordinate (L_Tip1) at (163,426);
\coordinate (L_Tip2) at (154,382);
\coordinate (L_Tip3) at (169,357);
\coordinate (L_Low) at (191,470);
\coordinate (L_Conn) at (207,440);
\coordinate (L_M) at (200,388);
\coordinate (B_L) at (251,497);
\coordinate (B_R) at (309,499);
\coordinate (B_Pk) at (317,433);
\coordinate (B_F) at (353,470);
\coordinate (B_Tip) at (377,493);
\coordinate (R_InT) at (318,405);
\coordinate (R_InB) at (349,413);
\coordinate (R_M) at (389,424);
\coordinate (R_Tip1) at (392,466);
\coordinate (R_Tip2) at (430,464);
\coordinate (R_Tip3) at (434,437);
\coordinate (R_V) at (421,412);
\coordinate (R_Circ) at (406,382);
\coordinate (T_RM) at (357,363);
\coordinate (T_RT) at (394,347);
\coordinate (T_MR) at (342,341);
\coordinate (T_ML) at (319,360);
\coordinate (T_TP) at (323,325);
\coordinate (T_LP) at (301,316);
\coordinate (T_LM) at (293,350);
\coordinate (T_FL) at (275,322);
\coordinate (T_C) at (242,344);
\coordinate (TL_P) at (252,317);
\coordinate (TL_M) at (220,322);
\coordinate (TL_L) at (192,340);
\coordinate (Label_P) at (240,471);

\draw (L_Low) -- (L_Conn) -- (L_Tip1) -- cycle;
\draw (B_L) -- (C_BR) -- (B_R) -- cycle;
\draw (L_Conn) -- (C_BL) -- (C_BR);
\draw (L_Low) -- (B_L);
\draw (C_BL) -- (C_TL) -- (C_TR) -- cycle;
\draw (L_Conn) -- (L_M);
\draw (L_M) -- (C_TL);
\draw (L_Tip1) -- (L_Tip2) -- (L_Tip3) -- (L_M);
\draw (L_Tip3) -- (TL_L) -- (L_M);
\draw (C_BR) -- (B_Pk);
\draw (B_R) -- (B_F);
\draw (B_Pk) -- (B_F);
\draw (B_F) -- (B_Tip);
\draw (B_F) -- (R_Tip1) -- (B_Tip);
\draw (C_TR) -- (R_InT) -- (B_Pk);
\draw (R_InT) -- (R_InB) -- (B_Pk);
\draw (R_InB) -- (R_M) -- (R_Tip1);
\draw (R_M) -- (R_Circ) -- (R_V) -- cycle;
\draw (R_V) -- (R_Tip3) -- (R_Tip2) -- (R_Tip1);
\draw (R_InB) -- (T_RM) -- (R_Circ);
\draw (R_InT) -- (T_ML) -- (T_MR) -- (T_RM);
\draw (T_RM) -- (T_RT) -- (R_Circ);
\draw (T_ML) -- (T_LM) -- (C_TR);
\draw (T_ML) -- (T_TP) -- (T_MR);
\draw (C_TL) -- (T_C) -- (T_FL) -- (T_LM);
\draw (T_FL) -- (T_LP) -- (T_LM);
\draw (TL_L) -- (TL_M) -- (T_C);
\draw (T_C) -- (TL_P) -- (TL_M);

\node[font=\bfseries] at (Label_P) {P};
\node[draw, circle, thick, inner sep=2.5pt] at (R_Circ) {};

\end{tikzpicture}
}
\captionof{figure}{Dead ends in the tiling of type $[3, 4, 5, 5]$}
\label{fig:3455} 
\end{figure}

\textbf{Type $[3, 4, k_3, k_4]$, $k_3$ and $k_4$ odd and $k_3<k_4$.} By Observation \ref{A1}, in a tiling of this type all the 4-gons must be either of type-2 or type-4. It follows that a neighborhood of a $k_3$-gon must be a composition of two kinds of partial neighborhoods shown in Fig.~\ref{fig:34k3k4}-(a) and Fig.~\ref{fig:34k3k4}-(b). Note that a partial neighborhood of the first kind covers four edges and the second kind covers two edges of the $k_3$-gon. Thus, for a valid neighborhood of a $k_3$-gon, $k_3$ must be even. The same parity argument applies to the $k_4$-gons.

\begin{figure}[H]
\tikzstyle{ver}=[]
\tikzstyle{vert}=[circle, draw, fill=black!100, inner sep=0pt, minimum width=4pt]
\tikzstyle{vertex}=[circle, draw, fill=black!00, inner sep=0pt, minimum width=4pt]
\tikzstyle{edge} = [draw,thick,-]
\centering
\subfigure[Type-1]{
\begin{tikzpicture}[scale=0.12]
\draw[edge, thick](70,25)--(75.5,25);
\draw[edge, thick](70,25)--(70,28)--(75.5,28)--(75.5,25);
\draw[edge, thick](70,25)--(65,22)--(64,17.5)--(65,14.5)--(67,12.8)--(70,11.2)--(72.7,10.8);
\draw[edge, thick](69.9,25)--(65.2,26.2)--(65,22)--(60.8,20.8);
\draw[edge, thick](75.5,25)--(80,22)--(81,17.7)--(80,14.5)--(78.5,12.8)--(75.5,11.2)--(72.7,10.8);
\draw[edge, thick](75.6,25)--(79.9,26.2)--(80,22)--(84,21.2);
\draw[edge, thick](83.5,16)--(81.1,17.7)--(84, 18.5);
\draw[edge, thick](61,16)--(63.9,17.5)--(60.8, 18.3);
\node[ver] () at (72.6,26.5){\scriptsize ${4}$};
\node[ver] () at (66.8,24.3){\scriptsize ${3}$};
\node[ver] () at (78.3,24.5){\scriptsize ${3}$};
\node[ver] () at (82.5,20.1){\scriptsize ${k_{4}}$};
\node[ver] () at (82,22.7){\scriptsize ${4}$};
\node[ver] () at (63,20){\scriptsize ${k_{4}}$};
\node[ver] () at (68.8,26.6){\scriptsize ${k_{4}}$};
\node[ver] () at (77,26.8){\scriptsize ${ k_{4}}$};
\node[ver] () at (73.1,17.5){\scriptsize ${ k_{3}}$};
\end{tikzpicture}}
\subfigure[Type-2]{
\begin{tikzpicture}[scale=0.15]
\draw[edge, thick](90,45)--(95,45);
\draw[edge, thick](90,45)--(90,48.8)--(95,48.8)--(95,45);
\draw[edge, thick](90,45)--(87,46.8)--(90,48.8);
\draw[edge, thick](95,48.8)--(98,46.8)--(95,45);
\draw[edge, thick](90,45)--(87,43.2)--(86.0,40.5)--(86.3,39)--(87.5,37.5)--(89.5,36)--(92.5,35.4);
\draw[edge, thick](95,45)--(98,43.2)--(99.0,40.5)--(98.7,39)--(97.3,37.3)--(95.3,36)--(92.5,35.4);
\draw[edge, thick] (84.5, 43.2)--(87.0,43.2)--(85,45.2);
\draw[edge, thick] (100.5, 43.2)--(98.0,43.2)--(100,44.8);
\node[ver] () at (92.5,46.8){\scriptsize ${4}$};
\node[ver] () at (88.9,46.8){\scriptsize ${3}$};
\node[ver] () at (96.2,46.8){\scriptsize ${3}$};
\node[ver] () at (92.7,40.5){\scriptsize ${ k_{3}}$};
\node[ver] () at (87.7,45){\scriptsize ${ k_{4}}$};
\node[ver] () at (97.3,45){\scriptsize ${ k_{4}}$};

\end{tikzpicture}}
\captionof{figure}{Partial neighborhoods of $k_3$-gon for the type $[3, 4, k_3, k_4]$}
\label{fig:34k3k4}
\end{figure}

\textbf{Type $[3, 4, 6, 8]$.} By Observation \ref{A1}, at least one $k_4=8$-gon must have a partial neighborhood of Type-2 above (Fig.~\ref{fig:34k3k4} -(b)) in the face-cycle of any triangle. To form a valid neighborhood of such an $8$-gon, at least two consecutive partial neighborhood of Type-2 is required. Consequently, we cannot form a valid neighborhood around the adjacent 6-gon opposite to the type-4 face adjacent to the $8$-gon, as shown by the mismatch around the encircled vertex in Fig.~\ref{fig:3568}. 
\begin{figure}[H] 
\centering
\begin{minipage}{.6\textwidth} 
\centering
\begin{tikzpicture}[scale=0.6,
 x=0.05cm, y=-0.05cm,
 line cap=round,
 line join=round,
 thick
 ]

\coordinate (C1) at (254.1,125.3);
\coordinate (C2) at (281.6,99.4);
\coordinate (C3) at (313.4,99.1);
\coordinate (C4) at (339.6,121.5);
\coordinate (C5) at (339.6,162.5);
\coordinate (C6) at (314.4,184.5);
\coordinate (C7) at (280.3,184.5);
\coordinate (C8) at (254.4,158.7);

\coordinate (R_TT) at (346.8,171.0);
\coordinate (R_CM) at (356.0,161.5);
\coordinate (R_QT) at (355.7,120.2);
\coordinate (R_M1) at (371.8,130.9);
\coordinate (R_M2) at (371.0,153.4);

\coordinate (R_ET1) at (363.8,113.4);
\coordinate (R_ET2) at (380.1,124.8);
\coordinate (R_EB1) at (378.1,160.0);
\coordinate (R_EB2) at (362.9,168.0);

\coordinate (R_V) at (381.6,143.5);
\coordinate (R_T1) at (392.8,136.9);
\coordinate (R_T2) at (392.3,149.6);

\coordinate (T_H) at (313.2,87.7);
\coordinate (T_D1) at (322.6,91.1);
\coordinate (T_D2) at (347.8,113.3);

\coordinate (B_L) at (313.7,194.9);
\coordinate (B_D) at (321.9,192.1);

\coordinate (L_T) at (281.6,88.8);
\coordinate (L_MT) at (254.5,99.8);
\coordinate (L_FL1) at (239.0,125.3);
\coordinate (L_FL2) at (239.0,158.7);
\coordinate (L_MB) at (253.8,184.9);
\coordinate (L_B) at (280.3,195.5);

\draw (C1) -- (C2) -- (C3) -- (C4) -- (C5) -- (C6) -- (C7) -- (C8) -- cycle;

\draw (C4) -- (R_QT) -- (R_M1) -- (R_M2) -- (R_CM) -- (C5);

\draw (C5) -- (R_CM) -- (R_TT) -- cycle;
\draw (R_TT) -- (B_D) -- (C6);

\draw (R_QT) -- (R_ET1) -- (R_ET2) -- (R_M1);

\draw (R_CM) -- (R_EB2) -- (R_EB1) -- (R_M2);

\draw (R_M1) -- (R_V) -- (R_M2);

\draw (R_V) -- (R_T1);
\draw (R_V) -- (R_T2);

\draw (C3) -- (T_H);
\draw (C3) -- (T_D1) -- (T_D2);
\draw (R_QT) -- (T_D2);
\draw (C4) -- (T_D2);

\draw (C6) -- (B_L) -- (B_D);

\draw (L_T) -- (C2);
\draw (C2) -- (L_MT) -- (C1);
\draw (C1) -- (L_FL1) -- (L_FL2) -- (C8);
\draw (C8) -- (L_MB) -- (C7);
\draw (C7) -- (L_B);

\node[draw, circle, inner sep=1.7pt] at (381.6,143.5) {};

\node[font=\small] at (295.0,145.0) {8};
\node [font=\small] at (350.0,145.0) {6};
\node [font=\small] at (355.5,114.0) {8};
\node [font=\small] at (355.5,169.0) {8};
\node [font=\small] at (293.0,94.0) {6};
\node [font=\small] at (293.0,191.0) {6};
\node [font=\small] at (382.0,135.0) {8};
\node [font=\small] at (382.0,151.0) {8};

\end{tikzpicture}

\captionof{figure}{Neighborhood of a $8$-gon for type $[3, 4, 6, 8]$ }
\label{fig:3568} 
\end{minipage}
\end{figure}
\textbf{Type $[3, 5, k_3, k_4]$.} We divide this case into further subcases depending on values of $k_3$. Along with Observation \ref{A2}, we note that a type-$4$ pentagon forces a partial neighborhood of the $k_3$-gon (or $k_4$-gon) with two $\Delta$-triangles attached to two alternating edges; see Fig.~\ref{35all}.

\begin{figure}[H] 
\tikzstyle{ver}=[]
\tikzstyle{vert}=[circle, draw, fill=black!100, inner sep=0pt, minimum width=4pt]
\tikzstyle{vertex}=[circle, draw, fill=black!00, inner sep=0pt, minimum width=4pt]
\tikzstyle{edge} = [draw,thick,-]
\centering
\begin{tikzpicture}[scale=0.12]
\draw[edge, thick](45,42)--(60,42);
\draw[edge, thick](45,42)--(43.9,46)--(45.2,50);
\draw[edge, thick](60,42)--(61.1,46);
\draw[edge, thick](45,42)--(41.2,43.7)--(43.9,46)--(42,48.2);
\draw[edge, thick](60,42)--(63.8,43.7)--(61.1,46)--(62.5,48.2);
\draw[edge, thick](45,42)--(46,38)--(52.5,36)--(58.9,38)--(60,42);
\draw[edge, thick](46,38)--(42.2,38.2)--(46.3,35)--(46,38);
\draw[edge, thick](52.5,36)--(57,35.3)--(58.9,38);
\draw[edge, thick](54,42)--(61.1,46)--(59,50);
\draw[edge, thick](45.2,50)--(46.5,52)--(48,53)--(50.5,54)--(53,54)--(55,53.5)--(57,52.5)--(59,50);
\node[ver] () at (52.5,46.5){\scriptsize ${ k_{3}}$};
\node[ver] () at (43.6,49){\scriptsize ${5}$};
\node[ver] () at (61.2,49){\scriptsize ${5}$};
\node[ver] () at (51,39.5){\scriptsize ${5}$};
\node[ver] () at (54,39.5){\scriptsize ${P}$};
\end{tikzpicture}
\captionof{figure}{A neighborhood of a $k_3$-gon for $[3, 5, k_3, k_4]$}
\label{35all} 
\end{figure}
\textbf{For $k_3=6, 9$.} 
It is straightforward to verify that this partial neighborhood of the $k_3$-gon cannot be extended to a (full) neighborhood without contradicting Observations \ref{A1}--\ref{A2}.
\par
\textbf{For $k_3=8$.} Note that there is an $8$-gon in a neighborhood of a type-$4$ pentagon with a (at least one) $\nabla$-triangle around it. It is again straightforward to verify that we cannot form a neighborhood of the $8$-gon which is consistent with Observations \ref{A1}--\ref{A2}. 
\par
\textbf{For $[3, 5, 7, k_4]$.} In light of Observations \ref{A1}--\ref{A2}, it is straightforward to verify that for a valid neighborhood, each $7$-gon in such a tiling must have only one $\nabla$-triangle attached to it. This along with observations \ref{A1}--\ref{A2} forces a unique extension of the partial neighborhood of a particular $k_4$-gon adjacent to a type-$4$ pentagon used above for $k_3$-gon for $k_3=6, 9$; see Fig.~\ref{3578k}. One can also see that the orientation of the same repetitive sequence of faces on two sides of the type-$4$ pentagon $P_5$ is opposite, so they can never meet to form a neighborhood of the $k_4$-gon. 

\begin{figure}[H] 
\tikzstyle{ver}=[]
\tikzstyle{vert}=[circle, draw, fill=black!100, inner sep=0pt, minimum width=4pt]
\tikzstyle{vertex}=[circle, draw, fill=black!00, inner sep=0pt, minimum width=4pt]
\tikzstyle{edge} = [draw,thick,-]
\begin{minipage}{.5\textwidth} 
\centering
\begin{tikzpicture}[scale=0.15]
\draw[edge, thick](45,42)--(59,42);
\draw[edge, thick](45,42)--(44,45.5)--(44.2,47.5)--(45,49.5);
\draw[edge, thick](59,42)--(60.1,45.5);
\draw[edge, thick](45,42)--(46,40)--(52,38)--(58,40)--(59,42);
\draw[edge, thick](45,49.5)--(46.5,52)--(48,53)--(50,54)--(53.5,54.1)--(55,53.7)--(57,52.7)--(58.5,51.5)--(59.5, 49.5)--(60.2,47.5)--(60.1,45.5);
\draw[edge, thick](59,42)--(61.5,43.6)--(60.1,45.5)--(62,45.6);
\draw[edge, thick](60.2,47.5)--(61.5,48.6)--(61.9,46.8)--(60.2,47.5);
\draw[edge, thick](61,49.6)--(59.5, 49.5)--(60.5, 51.3)--(58.5, 51.5)--(59.5, 52.5);
\draw[edge, thick](48, 53)--(46.3,53.8)--(47.5,54.7)--(48, 53);
\draw[edge, thick](45.3,53.1)--(46.5,52)--(44.2,51.7)--(45,49.5)--(43.5,50);
\draw[edge, thick](44.2, 47.5)--(42.3,47)--(43,49)--(44.2, 47.5);
\draw[edge, thick](45,42)--(42.5,43.7)--(44,45.5)--(42.1,45.6);
\node[scale=1.25] at (52.5,46.5){\scriptsize ${ k_{4}}$};
\node[ver] () at (44,49){\scriptsize ${5}$};
\node[ver] () at (43.5,46.5){\scriptsize ${7}$};
\node[ver] () at (48.5,54){\scriptsize ${5}$};
\node[ver] () at (46.5,52.8){\scriptsize ${7}$};
\node[ver] () at (60.5,48.8){\scriptsize ${5}$};
\node[ver] () at (60.7,46.3){\scriptsize ${7}$};
\node[ver] () at (58.3,52.5){\scriptsize ${7}$};
\node[scale=1.2] at (52,39.8){\scriptsize ${P_5}$};
\end{tikzpicture}
\captionof{figure}{Neighborhood of a $k_4$-gon for $[3, 5, 7, k_4]$}
\label{3578k} 
\end{minipage}%
\begin{minipage}{.6\textwidth} 
\centering
\begin{tikzpicture}[scale=0.15]
\draw[edge, thick](47,42)--(58,42);
\draw[edge, thick](58,42)--(60.1,45);
\draw[edge, thick](45.5,50.5)--(48,53)--(52,54.1)--(56,53.3)--(59, 51)--(60.2,48)--(60.1,45);
\draw[edge, thick](45.5,50.5)--(44.5,47.5)--(45,45)--(47,42);
\draw[edge, thick](47,42)--(45,40)--(46.5,38.5)--(48,37.5)--(51,37)--(53.5,37)--(56.5,37.5)--(58.5,38.5)--(60,40)--(58,42);
\draw[edge, thick](47,42)--(44,42.5)--(42,44)--(43,45)--(45,45);
\draw[edge, thick](42.5,49)--(44.5,47.5)--(42.5,47)--(45,45);
\draw[edge, thick](45,40)--(44,42.5);
\draw[edge, thick](62.2,48.7)--(60.2,48)--(62.3,47)--(60.1,45);
\draw[edge, thick](58,42)--(61,42)--(63,43.5)--(62,45)--(60.1,45);
\draw[edge, thick](61,42)--(60,40);
\node[scale=1.25] at (52.5,46.5){\scriptsize ${11}$};
\node[ver] () at (44,49){\scriptsize ${k_4}$};
\node[ver] () at (44.5,43.5){\scriptsize ${5}$};
\node[ver] () at (60.8,49.5){\scriptsize ${k_4}$};
\node[ver] () at (61,43.5){\scriptsize ${5}$};
\node[scale=1.25] () at (53,39.8){\scriptsize ${k_4}$};
\node[scale=1.5] at (50,39.8){\scriptsize ${U}$};
\end{tikzpicture}
\caption{Partial neighborhood of a $11$-gon}
\label{fig:3511k4}
\end{minipage}
\end{figure}

\textbf{For $[3, 5, 10, 11]$ and $[3, 5, 11, k_4]$.} We observe that for tuple  $[3, 5, 11, k_4]$, $\nabla$-triangles cannot be attached to a $k_4$-gon at two consecutive vertices which are also incident to an edge of an $11$-gon. Such configuration forces a partial neighborhood of the $11$-gon that cannot be extended to a neighborhood consistent with observations \ref{A1}--\ref{A2}, as shown in Fig.~\ref{fig:3511k4} ($k_4$-gon $U$). This constraint, combined with Observations \ref{A1}--\ref{A2} force a unique extension of the partial neighborhood of a particular $k_4$-gon adjacent to a type-$4$ pentagon as shown in Fig.~ \ref{3578k} with the $7$-gons replaced by $11$-gons. Thus, using the same argument used for the tuple $[3, 5, 7, k_4]$, it is impossible to form a neighborhood of the $k_4$-gon. Hence there does not exist a tiling of type $[3, 5, 11, k_4]$. The same argument applies to the tuple $[3, 5, 10, 11]$ by treating $10$-gon as $k_4$-gon.
\par
We now present the proof of the \enquote{if} (existence) part of the theorem. This is also carried out in a case by case manner. 
\par
\textbf{Type $[3, 3, 3p, 3q]$.} The existence of a tiling is shown in Corollary \ref{33pq} below.
\par
\textbf{Type $[3, 4, p, p]$, $p \geq 6$.} By Lemma \ref{twoeven}, we may assume $p \geq 7$. In this case, we modify the previous layer-by-layer construction. The possible in-degrees of a vertex on $\partial X_k$ are $0, 1$, or $2$. Vertices of in-degree $2$ occur only in the configurations shown in Fig.~\ref{34pp1} (the vertices $u_i$ and $v$) and in Fig.~\ref{34pp2} (the vertex $v_i$). Observe that no two consecutive vertices on $\partial X_k$ have in-degree $2$. First, for all the vertices of in-degree $2$ vertex on $\partial X_k$ arising from face sequence $(p, 4, p)$, we modify $X_k$ by attaching a triangle, illustrated in Fig.~\ref{34pp1} by the triangle incident to the vertex $v$ with one curved side. We then proceed with the inductive layer construction. In the remaining two cases of vertices of in-degree $2$ arise when a triangle is attached to $\partial X_k$ at a single vertex by faces in $X_k$ corresponding to: (a) $\{p, 3, p\}$ and (b) $\{4, 3, p\}$. 
\par
Case (a) $\{p, 3, p\}$. In this case, as shown in Fig.~\ref{34pp1} (at vertex $u_i$ on $\partial X_{k-2}$), we require that the $4$-gon incident to $u_i$ is preceded and succeeded by $p$-gons. Since $p\geq 7$, there are at least three free vertices $u_{i-3}, u_{i-2}, u_{i-1}$ preceding $u_i$. Consequently, regardless of the face $S$ adjacent to $u_{i-3}$, we can choose the other two faces in the fan around $u_{i-3}$ such that $T$ is a $p$-gon. This yields the desired configuration around $u_i$. 
\par
 Case (b) $[4, 3, p]$. Here, we observe that that the vertices of the triangles are preceded by two vertices of in-degree $2$ as shown in Fig.~\ref{34pp2}. This provides sufficient flexibility to choose the faces appropriately. It is straightforward to verify that no obstacle to the induction exists in any of these cases.. 
\begin{figure}[H] 
\begin{minipage}{.50\textwidth} 
\centering 
\begin{tikzpicture}[scale=0.5,
	x=0.6mm, y=-0.6mm,
	line cap=round, 
	line join=round,
	thick
	]
		
	\def\yTop{208}
	\def\yMid{248}
	\def\yBot{279}
		
	\coordinate (ST) at (228, \yTop);
	\coordinate (ET) at (420, \yTop);
		
	\coordinate (SM) at (229, \yMid);
	\coordinate (EM) at (426, \yMid);
		
	\coordinate (SB) at (266, \yBot);
	\coordinate (EB) at (402, \yBot);
		
	\coordinate (t1) at (249, \yTop);
	\coordinate (t2) at (267, \yTop);
	\coordinate (t3) at (278, \yTop);
	\coordinate (t4) at (304, \yTop);
	\coordinate (t5) at (330, \yTop);
	\coordinate (t6) at (372, \yTop);
	\coordinate (t7) at (395, \yTop);
	\coordinate (t5a) at (350, \yTop);
		
	\coordinate (m1) at (246, \yMid);
	\coordinate (m2) at (270, \yMid);
	\coordinate (m3) at (285, \yMid);
	\coordinate (m4) at (322, \yMid);
	\coordinate (m5) at (347, \yMid);
	\coordinate (m6) at (376, \yMid);
	\coordinate (m7) at (412, \yMid);
		
	\coordinate (b1) at (278, \yBot);
	\coordinate (b2) at (319, \yBot);
	\coordinate (b3) at (368, \yBot);
		
	\draw (ST) -- (ET);
	\draw (SM) -- (EM);
	\draw (SB) -- (EB);
		
	\draw (t1) -- (m2) -- (t2);
	\draw (m4) -- (t5a) -- (m6);
	\draw (b1) -- (m1);
	\draw (m3) -- (t3);
	\draw (m3) -- (t4) -- (m4);
	\draw (b2) -- (m5) -- (b3);
	\draw (m6) -- (t7) -- (m7);
		
	\draw[thin] (t5) .. controls (342, 192) and (360, 190) .. (t6); 
	\draw[dashed] (t5) -- (318, 185);
	\draw[dashed] (t6) -- (382, 185);
		
	\draw (b2) -- +(0, 12);
	\draw (b3) -- +(0, 12);
		
	\node[anchor=east] at (ST) {$\partial X_k$};
	\node[anchor=east] at (SM) {$\partial X_{k-1}$};
	\node[anchor=east] at (SB) {$\partial X_{k-2}$};
		
	\node at (255, 239) {$S$};
	\node at (287, 227) {$T$};
		
	\node[font=\footnotesize] at (305, 239) {$3$};
	\node[font=\footnotesize] at (347, 235) {$4$};
	\node[font=\footnotesize] at (415, 239) {$3$};
		
	\node at (301, 270) {$p$};
	\node at (372, 268) {$p$};
	\node at (325, 230) {$p$};
	\node at (375, 230) {$p$};
	\node at (350, 190) {$p$};
	\node [font=\footnotesize] at (350, 200) {$3$};
	\node [font=\footnotesize] at (320, 203) {$4$};
	\node[font=\footnotesize] at (380, 203) {$4$};
	\node[font=\footnotesize] at (345, 265) {$3$};
		
	\node[font=\footnotesize] at (272, 253) {$u_{i\!-\!3}$};
	\node[font=\footnotesize] at (290, 253) {$u_{i\!-\!2}$}; 
	\node[font=\footnotesize] at (322, 253) {$u_{i\!-\!1}$};
	\node[font=\footnotesize] at (356, 252.5) {$u_{i}$};
		
	\node[below, font=\footnotesize, inner sep = 5pt] at (t5a) {$v$};
		
\end{tikzpicture}

\captionof{figure}{Case (a)}
\label{34pp1} 
\end{minipage}% 
\begin{minipage}{.50\textwidth} 
\centering	
\begin{tikzpicture}[scale=0.8,
	x=0.6mm, y=-0.6mm,
	line cap=round, 
	line join=round,
	thick
	]
		
\def\yTop{90}
\def\yMid{115}
\def\yBot{141}
		
\coordinate (Start_Top) at (223, \yTop);
\coordinate (End_Top) at (329, \yTop);
		
\coordinate (Start_Mid) at (225, \yMid);
\coordinate (End_Mid) at (330, \yMid);
		
\coordinate (Start_Bot) at (232, \yBot);
\coordinate (End_Bot) at (330, \yBot);
		
\coordinate (T1) at (235, \yTop);
\coordinate (T2) at (244, \yTop);
\coordinate (T3) at (257, \yTop);
\coordinate (T4) at (263, \yTop);
\coordinate (T5) at (281, \yTop);
\coordinate (T6) at (324, \yTop);
\coordinate (T7) at (326, \yTop);
		
\coordinate (M1) at (241, \yMid);
\coordinate (M2) at (256, \yMid);
\coordinate (M3) at (271, \yMid);
\coordinate (M4) at (293, \yMid);
\coordinate (M5) at (312, \yMid);
\coordinate (M6) at (323, \yMid);
\coordinate (M7) at (326, \yMid);
		
\coordinate (B1) at (259, \yBot);
\coordinate (B2) at (283, \yBot);
\coordinate (B3) at (299, \yBot);
\coordinate (B4) at (322, \yBot);

\draw (Start_Top) -- (End_Top);
\draw (Start_Mid) -- (End_Mid);
\draw (Start_Bot) -- (End_Bot);
		
\draw (M1) -- (T1);
\draw (M2) -- (T2);
\draw (M2) -- (T3);
\draw (M3) -- (T4);
		
\draw (M3) -- (T5) -- (M4);
		
\draw (B3) -- (M5) -- (B4);
\draw (T7) -- (M7);		
\draw (B2) -- (M4);
\draw (B2) -- (Start_Mid);
		
\draw (B1) -- +(0, 8);
\draw (B3) -- +(0, 8);
\draw (B4) -- +(0, 8);
		
\node[anchor=east] at (Start_Top) {$\partial X_{k+1}$};
\node[anchor=east] at (Start_Mid) {$\partial X_k$};
\node[anchor=east] at (Start_Bot) {$\partial X_{k-1}$};
		
\node[below, inner sep=1pt, font=\footnotesize] at (M1) {$v_{i\!-\!4}$};
\node[below, inner sep=1pt, font=\footnotesize] at (M2) {$v_{i\!-\!3}$};
\node[below, inner sep=1pt, font=\footnotesize] at (M3) {$v_{i\!-\!2}$};
\node[below right, inner sep=1pt, font=\footnotesize] at (M4) {$v_{i\!-\!1}$};
\node[above, inner sep=1pt, font=\footnotesize] at (M5) {$v_i$};
		
\node at (245,105) {$S$};
\node at (280,105) {$T$};
\node at (305,105) {$U$};
		
\node[font=\footnotesize] at (310,127) {$3$};
\node at (325,127) {$p$};
\node at (280,127) {$p$};
\node at (280,145) {$p$};
\node[font=\footnotesize] at (250,135) {$3$};
		
\end{tikzpicture}
\captionof{figure}{Case (b) }
\label{34pp2}
\end{minipage}
\captionof{figure}{Layer construction for type $[3, 4, p, p]$ }
\end{figure}
\textbf{Type $[3, 5, p, p]$ with $p \geq 5$.} The existence of a tiling of type $[3, 5^3]$ was established in \cite{AM20}, and that of type $[3, 5, 6^2]$ follows from Proposition \ref{twoeven}. For $p \geq 7$, we may employ the construction method used for $[3, 4, p, p]$ above without modifying $\partial X_k$, as no vertex of in-degree $2$ (of the type shown in Fig.~\ref{34pp1}) arises on $\partial X_k$.
\par
\textbf{Type $[3, k_2, k_3, k_4]$, $k_2, k_3, k_4 \geq 6$.} The inductive layer construction method used thus far does not readily apply for these tuples. We modify the construction by first attaching triangles to $\partial X_k$ and subsequently constructing neighborhoods of these triangles as illustrated in Fig.~\ref{fig:35big}. As mentioned in the introduction, these modified layers are referred to as \textit{t-layers}. Evidently, in this construction $C_k$ does not contain any triangle; consequently, no vertex on $\partial X_k$ has in-degree $1$ or greater. Since $k_2, k_3, k_4 \geq 6$, each face in $C_k$ possesses at least two free vertices on $\partial X_k$.
 \begin{figure}[H]
\tikzstyle{ver}=[]
\tikzstyle{vert}=[circle, draw, fill=black!100, inner sep=0pt, minimum width=4pt]
\tikzstyle{vertex}=[circle, draw, fill=black!00, inner sep=0pt, minimum width=4pt]
\tikzstyle{edge} = [draw,thick,-]
\centering
\begin{tikzpicture}[scale=0.12]

\draw[edge, thick](5,5)--(67,5);
\draw[edge, thick](5,15)--(67,15);

\draw[edge, thick](11, 10)--(14,5);
\draw[edge, thick](8, -1)--(8,5);
\draw[edge, thick](11, 10)--(8,5);
\draw[edge, thick](11, 10)--(14,15);
\draw[edge, thick](16, 15)--(14,5);
\draw[edge, thick](11, 10)--(8,15);

\draw[edge, thick](23, 5)--(20,10);
\draw[edge, thick](23, 5)--(26,10);
\draw[edge, thick](20, 10)--(26,10);
\draw[edge, thick](26, 10)--(24,15);
\draw[edge, thick](26, 10)--(27,15);
\draw[edge, thick](20, 10)--(18,15);
\draw[edge, thick](20, 10)--(22,15);

\draw[edge, thick](33, 5)--(30,10);
\draw[edge, thick](33, 5)--(36,10);
\draw[edge, thick](30, 10)--(36,10);
\draw[edge, thick](30, 10)--(29,15);
\draw[edge, thick](30, 10)--(32,15);
\draw[edge, thick](36, 10)--(34,15);
\draw[edge, thick](36, 10)--(37,15);

\draw[edge, thick](43, 10)--(46,5);
\draw[edge, thick](43, 10)--(40,5);
\draw[edge, thick](43, 10)--(45,15);
\draw[edge, thick](43, 10)--(41,15);
\draw[edge, thick](39, 15)--(40,5);
\draw[edge, thick](47, 15)--(46,5);

\draw[edge, thick](53, 5)--(50,10);
\draw[edge, thick](53, 5)--(56,10);
\draw[edge, thick](56, 10)--(50,10);
\draw[edge, thick](50, 10)--(49,15);
\draw[edge, thick](50, 10)--(52,15);
\draw[edge, thick](56, 10)--(54,15);
\draw[edge, thick](56, 10)--(57,15);

\draw[edge, thick](61, 10)--(64,5);
\draw[edge, thick](58, -1)--(58,5);
\draw[edge, thick](61, 10)--(58,5);
\draw[edge, thick](61, 10)--(64,15);
\draw[edge, thick](61, 10)--(59,15);

\node[ver] () at (6.5, 2){\scriptsize $R$};
\node[ver] () at (9, 4){\scriptsize $u$};
\node[ver] () at (30, 2){\scriptsize $P$};
\node[ver] () at (57, 4){\scriptsize $v$};
\node[ver] () at (18,8){\scriptsize $Q$};
\node[ver] () at (37, 8){\scriptsize $Q$};
\node[ver] () at (29,8){\scriptsize $R$};
\node[ver] () at (58, 8){\scriptsize $Q$};
\node[ver] () at (49,8){\scriptsize $R$};
\node[ver] () at (60, 2){\scriptsize $R$};
\node[ver] () at (72,5){\scriptsize ${\partial X_{k}}$};
\node[ver] () at (72.5, 15){\scriptsize ${\partial X_{k+1}}$};

\end{tikzpicture}
\captionof{figure}{t-layer for type $[3, k_2, k_3, k_4]$}
\label{fig:35big} 

\end{figure}

The only obstruction that may arise in constructing t-layers is the following: \textbf{($\mathcal{O}$)} a fan around a vertex $u$ (see Fig.~\ref{fig:35big}) of in-degree $1$ can force an invalid fan around the next vertex $v$ of in-degree $1$ along $\partial X_k$, depending on the sequence of $\Delta$-triangles and $\nabla$-triangles attached to the vertices between $u$ and $v$.
\par
By Lemma \ref{twoeven}, there exists a tiling of type $[3, 6, 6, 6]$. Thus, we may assume that one of the $k_i$ is at least $7$. We also assume $k_2, k_3$ and $k_4$ are distinct labels. We begin by constructing fans around the final (in clockwise direction) vertex of a $7$-gon of in-degree $1$ on $\partial X_k$. We observe that for all possible sequences of faces preceding and succeeding a $6$-gon in $C_{k}$ and for all possible placement of the first triangle attached to a $6$-gon along $\partial X_k$, the triangles can be placed suitably (as shown in Fig.~\ref{fig:3678}) in forward or reverse order so that obstruction \textbf{($\mathcal{O}$)} is avoided. The logic applies for $p>6$, completing the the construction of the t-layer and, consequently, the desired tiling.

\begin{figure}[H]
\tikzstyle{ver}=[]
\tikzstyle{vert}=[circle, draw, fill=black!100, inner sep=0pt, minimum width=4pt]
\tikzstyle{vertex}=[circle, draw, fill=black!00, inner sep=0pt, minimum width=4pt]
\tikzstyle{edge} = [draw,thick,-]
\centering
\begin{minipage}{.30\textwidth} 
\begin{tikzpicture}[scale=0.15]
\draw[edge, thick](1,-40)--(18,-40);
\draw[edge, thick](3,-40)--(3,-43);
\draw[edge, thick](13,-40)--(13,-43);
\draw[edge, thick](3,-40)--(4.8,-37.5)--(6.6,-40)--(7.8,-37.5);
\draw[edge, thick](10.5,-40)--(9,-37.5)--(11.5,-37.5)--(10.5,-40);
\draw[edge, thick](13,-40)--(14.8,-37.5)--(16.6,-40);
\node[ver] () at (-1,-40){\scriptsize $\partial X_{k}$};
\node[ver] () at (2,-41.7){\scriptsize ${q}$};
\node[ver] () at (9,-41.7){\scriptsize ${p}$};
\node[ver] () at (14.5,-41.7){\scriptsize ${r}$};
\node[ver] () at (12,-39){\scriptsize ${q}$};
\node[ver] () at (8,-39){\scriptsize ${q}$};
\end{tikzpicture}
\end{minipage}
\begin{minipage}{.30\textwidth} 
\begin{tikzpicture}[scale=0.15]
\draw[edge, thick](20,-40)--(36,-40);
\draw[edge, thick](22,-40)--(22,-43);
\draw[edge, thick](34,-40)--(34,-43);
\draw[edge, thick](22,-40)--(23.8,-37.5)--(25.6,-40)--(26.8,-37.5);
\draw[edge, thick](34,-40)--(32.2,-37.5)--(30.4,-40)--(29.2,-37.5);
\node[ver] () at (21,-41.7){\scriptsize ${q}$};
\node[ver] () at (26.5,-41.7){\scriptsize ${p}$};
\node[ver] () at (35.5,-41.7){\scriptsize ${q}$};
\node[ver] () at (28,-39){\scriptsize ${r}$};
\end{tikzpicture}
\end{minipage}
\begin{minipage}{.30\textwidth} 
\begin{tikzpicture}[scale=0.15]
\draw[edge, thick](4,-55)--(26,-55);
\draw[edge, thick](9,-55)--(9,-58);
\draw[edge, thick](9,-55)--(7.2,-52.5)--(5.4,-55);
\draw[edge, thick](21,-55)--(21,-58);
\draw[edge, thick](21,-55)--(22.8,-52.5)--(24.6,-55);
\draw[edge, thick](12,-55)--(10.2,-52.5)--(13.8,-52.5)--(12,-55);
\draw[edge, thick](17,-55)--(15.2,-52.5)--(18.8,-52.5)--(17,-55);
\node[ver] () at (8,-56.5){\scriptsize ${q}$};
\node[ver] () at (22,-56.5){\scriptsize ${q}$};
\node[ver] () at (16.5,-56.5){\scriptsize ${p}$};
\node[ver] () at (19,-54){\scriptsize ${r}$};
\node[ver] () at (14.5,-54){\scriptsize ${q}$};
\node[ver] () at (9.5,-54){\scriptsize ${r}$};
\end{tikzpicture}
\end{minipage}
\captionof{figure}{Matching in t-layer for $[3, k_2, k_3, k_4]$}
\label{fig:3678} 
\end{figure}

\textbf{Type $[3, 4, 2p, 2q]$, $3<p<q$.} We will follow the $t$-layer construction described above. We have already noted in the proof of the non-existence part of tiling of type $[3, 4, k_3, k_4]$ for $k_3$ or $k_4$ odd, that the $4$-gons are either of type-2 or type-4, and neighborhoods of the $k_3$ and $k_4$-gons are composed of partial neighborhoods of type-1 or type-2 shown in Fig.~\ref{fig:34k3k4}. It follows that there is a unique way one can attach triangles to the vertices of the $4$-gons on $\partial X_k$. Due to the cyclic repetition of faces around each triangle, a face in $C_k$ must reappear after at most three other faces. In particular, there are three faces ($k_3$ and $k_4$-gons) between two $4$-gons. Then we can form the neighborhoods of $k_3$-gon and $k_4$-gons between any two $4$-gons in $C_k$ by composing partial neighborhoods of type-1 or type-2. 
\par
\textbf{Type $[3, 5, k_3, k_4], k_3, k_4 \geq 5:$} For these tuples, we again construct t-layers. However, the presence of pentagon introduces the risk of violating the Observation \ref{A2}. To overcome this, we modify the constructions of t-layers: we first attach triangles to the vertices of the pentagons on $\partial X_k$ in a prescribed manner, then attach triangles to the remaining faces, and finally construct neighborhoods of these triangles. 
\par
By \textit{$n(r)$-gon} in $C_k$, we mean an $n$-gon that has exactly $n-r$ free vertices. In other words, a $n(r)$-gon appears as a $r$-gon when all its edges on $\partial X_k$ are viewed as a single edge. 
\par
\textbf{Type $[3, 5, 5, p]$, $p \geq 5$.} The existence of a homogeneous tiling of type $[3, 5^3]$ is established in \cite{AM20}; we therefore assume $p \geq 6$. 
\par
\textbf{Step 1.} Inductive hypothesis: a $5(r)$-gon appear in $C_{k}$ (up to order reversal) between two consecutive $p$-gons in one of the following sequences:
\[ \{p, 5(5), p\}, \{p, 5(4), p\}, \{p(3), 5(4), p(5)\}, \{p, 5(5), 5(4), p \}, \{p, 5(3), 5(5), p \} \]

We then place the triangles on the $5$-gons as shown in Fig.~\ref{fig:355parts}. The ambiguity in the orientation of triangles on a $5(4)$-gon in the sequence $\{p, 5(4), p\}$ is resolved as follows: since a $5(4)$-gon can appear in either $\{p(3), 5(4), p(4)\}$ or $\{p(3), 5(4), p(5)\}$, the triangle adjacent to both the $5(4)$-gon and the $p(3)$-gon is placed such that it shares an edge with the $p(3)$-gon; see Fig.~\ref{fig:355parts}-(a).

\begin{figure}[H]
\centering
\tikzstyle{ver}=[]
\tikzstyle{vert}=[circle, draw, fill=black!100, inner sep=0pt, minimum width=4pt]
\tikzstyle{vertex}=[circle, draw, fill=black!00, inner sep=0pt, minimum width=4pt]
\tikzstyle{edge} = [draw,thick,-]
 \begin{minipage}{.50\textwidth}
\begin{tikzpicture}[scale=0.15]

\draw[edge, thick](45,20)--(55,20);

\draw[edge, thick](47,20)--(47,16);

\draw[edge, thick](47,16)--(51,16);
\draw[edge, thick](51,16)--(53,18);
\draw[edge, thick](53,18)--(51,20);

\draw[edge, thick](47,20)--(49,23);

\draw[edge, thick](49,23)--(51,20);

\node[ver] () at (46,19){\scriptsize ${p}$};
\node[ver] () at (49,18){\scriptsize ${5}$};
\node[ver] () at (53,19){\scriptsize ${p}$};

\node[ver] () at (49,21.5){\scriptsize ${3}$};

\draw[edge, thick](57,20)--(70,20);
\draw[edge, thick](62,20)--(62,16)--(68,16)--(68,20);
\draw[edge, thick](62,20)--(61,23)--(60,20.1)--(59,22);
\draw[edge, thick](68,20)--(66.5,23)--(65,20.1)--(64,22);
\node[ver] () at (61,18){\scriptsize ${p}$};
\node[ver] () at (65,18){\scriptsize ${5}$};
\node[ver] () at (69,18){\scriptsize ${p}$};
\node[ver] () at (63,21){\scriptsize ${5}$};
\node[ver] () at (61,21){\scriptsize ${3}$};
\node[ver] () at (66.5,21){\scriptsize ${3}$};

\draw[edge, thick](74,20)--(97,20)--(86,20)--(86,15)--(76,20);
\draw[edge, thick](86,15)--(95,17)--(95,20);
\draw[edge, thick](74,25)--(97,25);
\draw[edge, thick](95,20)--(91,22.5)--(89,20.2)--(87,25);
\draw[edge, thick](91,22.5)--(93,25);
\draw[edge, thick](91,22.5)--(88,25);
\draw[edge, thick](86,20)--(83,22.5)--(81,20.2)--(80,25);
\draw[edge, thick](83,22.5)--(85,25);
\draw[edge, thick](83,22.5)--(81,25);
\draw[edge, thick](76,20)--(77,22.5)--(79,20.2)--(79,22.5);
\node[ver] () at (75,21){\scriptsize ${5}$};
\node[ver] () at (76,18){\scriptsize ${p}$};
\node[ver] () at (82,18.5){\scriptsize ${5}$};
\node[ver] () at (90,18){\scriptsize ${5}$};
\node[ver] () at (96,18){\scriptsize ${p}$};
\node[ver] () at (94,22){\scriptsize ${5}$};
\node[ver] () at (91,24){\scriptsize ${p}$};
\node[ver] () at (89,23){\scriptsize ${5}$};
\node[ver] () at (87,21){\scriptsize ${p}$};
\node[ver] () at (83,24.2){\scriptsize ${5}$};
\node[ver] () at (81.5,23){\scriptsize ${5}$};
\node[ver] () at (80,21){\scriptsize ${p}$};
\node[ver] () at (72,20){\scriptsize ${\partial X_{k}}$};
\node[ver] () at (73,24.5){\scriptsize $\partial X_{k+1}$};
\node[ver] () at (77.3,21){\scriptsize ${3}$};
\node[ver] () at (83,21){\scriptsize ${3}$};
\node[ver] () at (91,21){\scriptsize ${3}$};

\draw[edge, thick](74,25)--(97,25);

\end{tikzpicture}
\end{minipage}%
 \begin{minipage}{.20\textwidth}
 \centering
\begin{tikzpicture}[scale=0.15]

\draw[edge, thick](10,20)--(26,20);
\draw[edge, thick](12,20)--(12,16)--(15,15)--(18,16)--(18,20);
\draw[edge, thick](18,16)--(24,16)--(24,20)--(22.3,22.5)--(20.5,20.2)--(19.5,22.5);
\draw[edge, thick](12,20)--(15,22.5)--(18,20);
\node[ver] () at (11.2,18.5){\scriptsize ${p}$};
\node[ver] () at (15,17.5){\scriptsize ${5}$};
\node[ver] () at (19.5,17.5){\scriptsize ${5}$};
\node[ver] () at (25,18){\scriptsize ${p}$};
\node[ver] () at (18.8,21){\scriptsize ${p}$};
\end{tikzpicture}

\end{minipage}%
 \begin{minipage}{.20\textwidth}
\centering
\begin{tikzpicture}[scale=0.12]

\draw[edge, thick](29,20)--(56,20);
\draw[edge, thick](44,20)--(33,11)--(53.5,11)--(51.5,15)--(44,20);
\draw[edge, thick](51.5,15)--(50,20);
\draw[edge, thick](32.9,11)--(30.5,15)--(37.6,15)--(33,20);
\draw[edge, thick](33,20)--(34.5,23.5)--(36.5,20.3)--(37,23.5);
\draw[edge, thick](41,20.1)--(41,24);
\draw[edge, thick](47,20.1)--(47,24);
\draw[thick](47,20.1) arc (0:180:3);
\node[ver] () at (33,16.5){\scriptsize ${p}$};
\node[ver] () at (33.2,13){\scriptsize ${3}$};
\node[ver] () at (44, 21.5){\scriptsize ${3}$};
\node[ver] () at (38.8,18.5){\scriptsize ${5}$};
\node[ver] () at (44,16){\scriptsize ${5}$};
\node[ver] () at (49,18.5){\scriptsize ${p}$};
\node[ver] () at (52.5,17.8){\scriptsize ${5}$};
\node[ver] () at (49,22){\scriptsize ${5}$};
\node[ver] () at (44,24.5){\scriptsize ${5}$};
\node[ver] () at (38.5,21.5){\scriptsize ${p}$};
\node[ver] () at (59,20){\scriptsize ${\partial X_{k}}$};
\end{tikzpicture}

\end{minipage}

\captionof{figure}{Placement of triangles around pentagons for type $[3, 5, 5, p]$}
\label{fig:355parts}
\end{figure}

\textbf{Step 2.} We will now demonstrate that it is possible to attach triangles to the remaining vertices of $p$-gons to form neighborhoods consistent with the Observation \ref{A2}, and further, the inductive hypothesis holds for $k+1$-the layer. The primary obstacle arises when two $\nabla$-triangles are attached to the $p$-gon at the consecutive vertices, and the faces preceding (say $R$) and succeeding (say $S$) the pentagons between the $\nabla$-triangles are the same as llustrated in Fig.~\ref{355ppart1}. 

\begin{figure}[ht!]
\tikzstyle{ver}=[]
\tikzstyle{vert}=[circle, draw, fill=black!100, inner sep=0pt, minimum width=4pt]
\tikzstyle{vertex}=[circle, draw, fill=black!00, inner sep=0pt, minimum width=4pt]
\tikzstyle{edge} = [draw,thick,-]
\centering
\begin{tikzpicture}[scale=0.15]

\draw[edge, thick](7,3)--(32,3);
\draw[edge, thick](7,10)--(32,10);

\draw[edge, thick](10, 3)--(12,6.5)--(15,3);
\draw[edge, thick](12,6.5)--(16,10);
\draw[edge, thick](16,10)--(21,6.5)--(23, 10);
\draw[edge, thick](10,3)--(8, 10);
\draw[edge, thick](12,6.5 )--(12,10);
\draw[edge, thick](24, 3)--(21,6.5)--(27, 6.5)--(24, 3);

\draw[edge, thick](15, 3)--(15, 0);
\draw[edge, thick](28, 3)--(28, 0);
\draw[edge, thick](24, 10)--(27, 6.5)--(29,10);

\node[ver] () at (30, 1.5){\scriptsize $5$};
\node[ver] () at (12, 1.5){\scriptsize $5$};
\node[ver] () at (20, 1.5){\scriptsize $p$};
\node[ver] () at (13, 8.5){\scriptsize $S$};
\node[ver] () at (20, 8.6){\scriptsize $R$};
\node[ver] () at (23.5, 8.5){\scriptsize $S$};
\node[ver] () at (27, 8.6){\scriptsize $R$};
\node[ver] () at (17, 6.5){\scriptsize $5$};
\node[ver] () at (29, 6.5){\scriptsize $5$};

\node[ver] () at (34,10){\scriptsize ${\partial X_{k}}$};
\node[ver] () at (35.5,3){\scriptsize ${\partial X_{k-1}}$};

\end{tikzpicture}
\captionof{figure}{Around a $p$-gon for type $[3, 5, 5, p]$}
\label{355ppart1}

\end{figure}

We will focus on the non-trivial cases of $p=6$ or $7$ to show how this is overcome; the logic extends naturally to higher values of $p$.
\par
In the $t$-layer construction, following step 1., a $p(r)$-gon can appear in $C_{k}$ in one of the following sequences: 
 \begin{multline*}
 \{ 5(3), p(5), 5(4)\}, \{ 5(3), p(4), 5(5)\}, \{ 5(4), p(4), 5(4)\}, \{ 5(4), p(3), 5(4)\}, \\
 \{ 5(3), p(3), 5(4)\}, \{5(3), 5(5), p(3), 5(4)\}
 \end{multline*}

Note that a $p(r)$-gon in $C_k$ must be attached to a $5(r)$-gon in $C_{k-1}$. This requirement rules out the sequences $\{5(3), p(6), 5(3)\}$ and \{ 5(3), p(4), 5(4)\} etc. 
\par
 A $p(5)$-gon in $C_k$ can arise only from the sequence shown in the third and fourth figure in Fig.~\ref{fig:355parts}. For $p=6$, we attach a $\nabla$-triangle to the remaining free vertex, and for $p=7$, we attach a $\Delta$-triangle to the remaining two free vertices. In the $p=6$ case, we must verify that the face $S$ in Fig.~\ref{355ppart1} can be chosen to be a $p$-gon. Indeed, the $5(3)$-gon preceding the $6$-gon must appear in the sequence $\{p, 5(4), 5(3), p \}$ which is again covered in the manner shown in the third figure in Fig.~\ref{fig:355parts}, consequently, $S$ is a $p$-gon. 
\par
The sequence $ \{ 5(3), p(4), 5(5)\}$ arises exclusively from the sequence $ \{p, 5(4), p\}$ in $C_{k-1}$ as shown in Fig.~\ref{fig:355parts}-(b). Following Step 1., we observe $p(4)$-gon has two remaining free vertices (an edge) to cover; we attach a $\nabla$-triangle to this edge, as shown in Fig.~\ref{355p(4)p2}.
\begin{figure}[ht!]
\tikzstyle{ver}=[]
\tikzstyle{vert}=[circle, draw, fill=black!100, inner sep=0pt, minimum width=4pt]
\tikzstyle{vertex}=[circle, draw, fill=black!00, inner sep=0pt, minimum width=4pt]
\tikzstyle{edge} = [draw,thick,-]
\centering
\begin{tikzpicture}[scale=0.15]

\draw[edge, thick](7,4)--(32,4);
\draw[edge, thick](5,10)--(32,10);

\draw[edge, thick](15, 4)--(15,10);
\draw[edge, thick](15, 4)--(6,10);
\draw[edge, thick](25, 4)--(25,10);
\draw[edge, thick](25, 10)--(27,12);
\draw[edge, thick](29, 10)--(27,12);
\draw[edge, thick](29, 4)--(32,7.5);

\draw[edge, thick](32, 7.5)--(29,10);

\draw[edge, thick](16, 12)--(18, 10)--(20,12)--(22, 10)--(24,12);

\draw[edge, thick](9, 12)--(11, 10)--(13,12)--(15, 10);

\node[ver] () at (15.8, 11.2){\scriptsize $5$};
\node[ver] () at (18.2, 11.5){\scriptsize $5$};
\node[ver] () at (22, 11.5){\scriptsize $5$};
\node[ver] () at (24.5, 11.2){\scriptsize $5$};

\node[ver] () at (13, 7.4){\scriptsize $5(3)$};
\node[ver] () at (20, 6.5){\scriptsize $P(4)$};
\node[ver] () at (28, 6.5){\scriptsize $5 (5)$};

\node[ver] () at (34.5,10){\scriptsize ${\partial X_{k}}$};
\node[ver] () at (35.5,4){\scriptsize ${\partial X_{k-1}}$};

\end{tikzpicture}
\captionof{figure}{Around a $p$-gon for type $[3, 5, 5, p]$}
\label{355p(4)p2}

\end{figure}

The last case of $\{5(3), 5(5), p(3), 5(4)\}$ is shown in Fig.~\ref{355pall}. Now it is straightforward to verify that the inductive hypothesis is satisfied at corona $C_{k+1}$. 

\begin{figure}[H] 
\centering
\begin{minipage}{0.5\textwidth} 
	\begin{tikzpicture}[x=0.25mm, y=-0.25mm,
		line cap=round, 
		line join=round,
		thick
		]
		
		%--- Define Y-Levels (Rounded from EPS) ---%
		\def\yTop{488}
		\def\yMid{539}
		\def\yBot{592}
		
		%--- Define Horizontal Boundary Coordinates ---%
		\coordinate (Start_Top) at (146, \yTop);
		\coordinate (End_Top) at (450, \yTop);
		
		\coordinate (Start_Mid) at (143, \yMid);
		\coordinate (End_Mid) at (493, \yMid);
		
		\coordinate (Start_Bot) at (175, \yBot);
		\coordinate (End_Bot) at (407, \yBot);
		
		%--- Define Vertices (Top Layer) ---%
		\coordinate (TL_1) at (188, \yTop);
		\coordinate (TL_2) at (205, \yTop);
		\coordinate (TL_3) at (215, \yTop);
		\coordinate (TL_4) at (252, \yTop);
		\coordinate (TL_5) at (292, \yTop);
		\coordinate (TL_4A0) at (316, \yTop);
		\coordinate (TL_4A) at (328, \yTop);
		\coordinate (TL_4B) at (334, \yTop);
		\coordinate (TL_4C) at (345, \yTop);
		\coordinate (TL_6) at (350, \yTop);
		\coordinate (TL_7) at (357, \yTop);
		\coordinate (TL_8) at (385, \yTop);
		\coordinate (TL_9) at (404, \yTop);
		\coordinate (TL_10) at (422, \yTop);
		
		\coordinate (TM_1) at (183, \yMid);
		\coordinate (TM_2) at (212, \yMid);
		\coordinate (TM_3) at (243, \yMid);
		\coordinate (TM_4) at (302, \yMid);
		\coordinate (TM_5) at (350, \yMid);
		\coordinate (TM_5A) at (332, \yMid);
		\coordinate (TM_6) at (374, \yMid);
		\coordinate (TM_7) at (407, \yMid);
		\coordinate (TM_8) at (438, \yMid);
		
		\coordinate (Float_L) at (202, 516);
		\coordinate (Float_R1) at (362, 520);
		\coordinate (Float_R2) at (428, 520);
		\coordinate (Float_R0) at (324, 520);
		\coordinate (Float_R01) at (340, 520);
		
		\draw (TL_4A)--(Float_R0) -- (Float_R01)--(TL_4B);
		\draw (TL_4A0)--(Float_R0);
		\draw (Float_R01)--(TL_4C);
		%--- Define Vertices (Bottom Layer) ---%
		% Pentagon Vertices
		\coordinate (P_Top) at (276, \yMid);
		\coordinate (P_ML) at (224, 567);
		\coordinate (P_MR) at (322, 565);
		\coordinate (P_BL) at (213, \yBot);
		\coordinate (P_BR) at (329, \yBot);
		
		% Neighbors
		\coordinate (BL_Mid) at (173, 567);
		\coordinate (BL_Top) at (159, \yMid);
		\coordinate (BR_Mid) at (386, 564);
		\coordinate (BR_Top) at (486, \yMid);

		%--- Draw Horizontal Lines ---%
		\draw (Start_Top) -- (End_Top);
		\draw (Start_Mid) -- (End_Mid);
		\draw (Start_Bot) -- (End_Bot);

		%--- Draw Top Layer Structures ---%
		% Left Structure (V and Vertical)
		\draw (TL_1) -- (Float_L) -- (TL_2);
		\draw (TM_1) -- (Float_L);
		\draw (TL_3) -- (TM_2);
		\draw (TM_5) --(Float_R1);
		\draw (TM_2) --(Float_L);
		
		% Center Diagonal & Curve
		\draw (TL_4) -- (TM_3);
		\draw (TL_5) -- (TM_4);
		\draw (TM_3) .. controls (262, 524) and (281, 515) .. (TM_4);
		
		% Right Structure (Vertical, Triangle, V)
		\draw (TL_6) -- (TM_5); % Vertical
		\draw (TL_7) -- (Float_R1) -- (TL_8); % Triangle on Top
		\draw (Float_R1) -- (TM_6); % Connection to Mid
		\draw (TL_9) -- (TM_7); % Vertical
		\draw (TL_10) -- (Float_R2) -- (TM_8); % V connecting to Mid
		\draw (TM_7) -- (Float_R2); % Closing the V on Mid line
		\draw (TM_6)--(P_MR);
		\draw (TM_5A)--(Float_R0); 
		\draw (TM_5A)--(Float_R01); 
		%--- Draw Bottom Layer Structures ---%
		% The Pentagon
		\draw (P_Top) -- (P_ML) -- (P_BL) -- (P_BR) -- (P_MR) -- cycle;
		\draw (TM_1) -- (P_ML);
		% Left Wing
		\draw (BL_Top) -- (BL_Mid) -- (P_ML);
		\draw (BL_Mid) -- (P_BL); % Connect to start of bot line
		
		% Right Wing
		\draw (P_MR) -- (BR_Mid);
		\draw (BR_Mid) -- (P_BR); % Triangle
		\draw (BR_Mid) -- (TM_8);
		\draw (BR_Mid) -- (BR_Top);
		
		% Vertical Ticks/Leafs at Bottom
		\draw (187, \yBot) -- +(0, 24);
		\draw (371, \yBot) -- +(0, 22);

		%--- Labels ---%
		\node[font=\footnotesize] at (227, 513) {$p$};
		\node[font=\footnotesize] at (319, 552) {$p$};
		\node[font=\footnotesize] at (340, 495) {$p$};
			\node[font=\footnotesize] at (323, 495) {$p$};
		\node[font=\footnotesize] at (435, 547) {$p$};
		\node[font=\footnotesize] at (266, 608) {$p$};
		\node[font=\footnotesize] at (365, 502) {$p$};
		\node[font=\footnotesize] at (415, 512) {$p$};
		\node[font=\footnotesize] at (177, 558) {$p$};
		\node[font=\footnotesize] at (200, 495) {$p$};

		\node[font=\footnotesize] at (219, 552) {$5$};
		\node[font=\footnotesize] at (267, 580) {$5$};
		\node[font=\footnotesize] at (372, 556) {$5$};
			\node[font=\footnotesize] at (180, 526) {$5$};
		\node[font=\footnotesize] at (208, 515) {$5$};
		\node[font=\footnotesize] at (372, 584) {$5$};
		\node[font=\footnotesize] at (346, 518) {$5$};
		\node[font=\footnotesize] at (390, 515) {$5$};
		\node[font=\footnotesize] at (310, 530) {$5$};
		\node[font=\footnotesize] at (280, 509) {$5$};
		
		\node[font=\footnotesize] at (274, 532) {$3$};
		\node[font=\footnotesize] at (200, 575) {$3$};
		\node[font=\footnotesize] at (198, 532) {$3$};
		\node[font=\footnotesize] at (362, 532) {$3$};
		\node[font=\footnotesize] at (424, 532) {$3$};
		\node[font=\footnotesize] at (332, 526) {$3$};
		\node[font=\footnotesize] at (340, 575) {$3$};
		
	\end{tikzpicture}
	
\captionof{figure}{Around a $p$-gon for type $[3, 5, 5, 6]$}
\label{355pall} 
\end{minipage}
\end{figure}
All the other cases are relatively straightforward to verify. 
\par
\textbf{Type $[3, 5, k_3, k_4]$, $ 10 \leq k_3 < k_4$, $k_3 \neq 11$.} \textbf{Step 1.} In Fig.~\ref{fig:357part1}, we show all possible sequence of faces preceding and succeeding (in the clockwise direction) a pentagon in $C_{k}$ (up to the symmetry of interchanging $k_3$ and $k_4$). The figure also demonstrate standard placement of triangles. The only exceptions occur when a $5(3)$-gon succeeds a $10(6)$-gon or $13(6)$-gon; in these instances, the triangles are placed as shown in the last figure of Fig.~\ref{fig:357part1}.

\begin{figure}[H]
\tikzstyle{ver}=[]
\tikzstyle{vert}=[circle, draw, fill=black!100, inner sep=0pt, minimum width=4pt]
\tikzstyle{vertex}=[circle, draw, fill=black!00, inner sep=0pt, minimum width=4pt]
\tikzstyle{edge} = [draw,thick,-]
\begin{minipage}{.6\textwidth} 
\centering
\begin{tikzpicture}[scale=0.18]

\draw[edge, thick](0,1)--(16,1);
\draw[edge, thick](2,1)--(3.5,3)--(5,1)--(5,-1.5)--(8,-2.5)--(11,-1.5)--(11,1)--(12.5,3)--(14,1);
\node[ver] () at (8,-0.5){\scriptsize ${5}$};
\node[ver] () at (3.5,-0.3){\scriptsize ${ k_{3}}$};
\node[ver] () at (12.5,-0.3){\scriptsize ${ k_{3}}$};
\node[ver] () at (8,2){\scriptsize ${ k_{4}}$};
\node[ver] () at (-2,1){\textbf{$\partial X_{k}$}};
------------------
\draw[edge, thick](18,1)--(31,1);
\draw[edge, thick](19,1)--(20.5,3)--(22,1)--(22,-2.5)--(27,-2.5)--(27,1)--(28.5,3)--(30,1);
\draw[edge, thick](24.5,1)--(23,3)--(26,3)--(24.5,1);
\node[ver] () at (20.5,-0.5){\scriptsize ${ k_{3}}$};
\node[ver] () at (28.5,-0.5){\scriptsize ${ k_{4}}$};
\node[ver] () at (22.5,2){\scriptsize ${ k_{4}}$};
\node[ver] () at (26.5,2){\scriptsize ${ k_{3}}$};
\node[ver] () at (24.5,-1){\scriptsize ${5}$};
------------------
\draw[edge, thick](33,1)--(43.5,1);
\draw[edge, thick](34,1)--(35.5,3)--(37,1)--(37,-2.5)--(42,-2.5)--(42,1)--(40.5,3)--(39,1)--(38.5,3);
\node[ver] () at (35.5,-0.5){\scriptsize ${ k_{3}}$};
\node[ver] () at (43.5,-0.5){\scriptsize ${ k_{3}}$};
\node[ver] () at (39.5,-1){\scriptsize ${5}$};
\node[ver] () at (37.7,2){\scriptsize ${ k_{4}}$};

------------------
 
\end{tikzpicture}
\end{minipage}%
\begin{minipage}{.6\textwidth} 
\begin{tikzpicture}[scale=0.18]

\draw[edge, thick](0,1)--(16,1);
\draw[edge, thick](1,1)--(2.5,3)--(4,1)--(8,-2.5)--(12,1)--(13.5,3)--(15,1);
\draw[edge, thick](5.5,3)--(6.5,1)--(8,3)--(9.5,1)--(10.5,3);
\node[ver] () at (8,-0.5){\scriptsize ${5}$};
\node[ver] () at (3,-0.5){\scriptsize ${ k_{4}}$};
\node[ver] () at (11.2,2){\scriptsize ${ k_{4}}$};
\node[ver] () at (12.8,-0.5){\scriptsize ${ k_{3}}$};
\node[ver] () at (4.6,2){\scriptsize ${ k_{3}}$};
------------------

\draw[edge, thick](17,1)--(34,1);
\draw[edge, thick](19,1)--(20.5,3)--(22,1)--(23,3);
\draw[edge, thick](19,1)--(24,-2.5)--(29,1)--(30.5,3)
--(32,1);
\draw[edge, thick](26,1)--(24.5,3)--(27.5,3)--(26,1);
\node[ver] () at (18.5,-0.5){\scriptsize ${10(6)}$};
\node[ver] () at (24,-0.5){\scriptsize ${5}$};
\node[ver] () at (28.2,2){\scriptsize ${10}$};
\node[ver] () at (24,2){\scriptsize ${ k_{4}}$};

 \node[ver] () at (30.5, -0.5){\scriptsize ${ k_{4}}$};
 
\end{tikzpicture}
\end{minipage}
\captionof{figure}{Neighborhoods of pentagons for type $[3, 5, k_3, k_4]$}
\label{fig:357part1} 

 \end{figure}

Due to the cyclic repetition of faces around each triangle, a face in $C_k$ must reappear after at most three other faces. Further, by inspecting the t-layer construction, we note that between any two successive pentagons in $C_k$, the sequence of $k_3$-gons or $k_4$-gons must be one of the following up to order reversal and interchanging $k_3$ and $k_4$:
 \[ a) \{5(3), k_3(6), 5(3)\}, \{5(3), k_3(5), 5(4)\}, \{5(4), k_3(4), 5(4)\}; \]
 \[b) \{5(3), k_4(5), k_3(3), 5(4)\}, \{5, k_4(4), k_3(3), 5\}; \]
 \[c) \{5(3), k_4(4), k_3(5), k_4(3), 5(3)\}, \{5(4), k_4(3), k_3(6), k_4(3), 5(4)\} \] 
 
The other possible sequences, e.g., $\{5(4), k_3(6), 5(3)\}$ and $ \{5(3), k_4(4), k_3(4), k_4(4), 5(3)\}$ do not arise because of the manner in which we placed the triangle on the pentagons; nevertheless, our argument below covers such cases should they arise.

\textbf{Step 2.} We now attach triangles to the other faces. We have the following possible neighborhoods of a $10$-gon shown in Fig.~\ref{fig:3512k4}, and one of the possible types of neighborhoods of a $12$-gon is shown in Fig.~$\ref{fig:35aroundk31}$ which are consistent with Observations \ref{A1}--\ref{A2}.

 \begin{figure}[H] 
\centering
\tikzstyle{ver}=[]
\tikzstyle{vert}=[circle, draw, fill=black!100, inner sep=0pt, minimum width=4pt]
\tikzstyle{vertex}=[circle, draw, fill=black!00, inner sep=0pt, minimum width=4pt]
\tikzstyle{edge} = [draw,thick,-]
 \begin{minipage}{.20\textwidth} 
\centering
\begin{tikzpicture}[scale=0.15]

 \draw[edge, thick](-1.6,5)--(1.6,5)--(4.25,3.07)--(5.25,0.0)--(4.25,-3.1)--(1.6,-5.0)--(-1.6,-5.0)--(-4.25,-3.1)--(-5.25,-0.0)--(-4.25,3.07)--(-1.6,5);
 
 \draw[edge, thick](-4.0,5.5)--(-1.6, 5.05)--(-1.8, 6.4);
 \draw[edge, thick](-6,3.2)--(-4.3, 3.08)--(-4.0, 5.5);
 
 \draw[edge, thick](-7,0.8)--(-5.25,-0.0)--(-6.5,-2);
 
 \draw[edge, thick](-6.5,-2)--(-4.25,-3.1)--(-5,-4.5);
 
 \draw[edge, thick] (-0.5,-6.6)--(-1.6,-5.0)--(-3.4,-5.8)--(-0.5,-6.6);
 
 \draw[edge, thick] (2,-6.6)--(1.6,-5.0)--(4,-5.2)--(4.25,-3.07)--(4.25,-3.07)--(5.5,-3.57); 
 
 \draw[edge, thick] (6.5,-0.5)--(5.25,0.0)--(6, 2)--(4.25,3.1)--(5.5,4);
 
 \draw[edge, thick](3.7, 5.5)--(1.6, 5.05)--(0.5, 6.5)--(3.7,5.5);
 
 \node[ver] () at (0,0){\scriptsize ${10}$};

 \node[ver] () at (-5.5,1.5){\scriptsize ${5}$};

 \node[ver] () at (-3.5, -4.5){\scriptsize ${k_{_4}}$};
 \node[ver] () at (0.5, -5.7 ){\scriptsize ${5}$};

 \node[ver] () at (6, -1.5){\scriptsize ${k_{_4}}$};

 \node[ver] () at (3.9, 4.6){\scriptsize $5$};
 \node[ver] () at (-0.1, 6.2){\scriptsize ${k_{_4}}$};
 
 \end{tikzpicture}
 \end{minipage}%
 \begin{minipage}{.20\textwidth} 
 \centering
 \begin{tikzpicture}[scale=0.15]
 
 \draw[edge, thick](-1.6,5)--(1.6,5)--(4.25,3.07)--(5.25,0.0)--(4.25,-3.1)--(1.6,-5.0)--(-1.6,-5.0)--(-4.25,-3.1)--(-5.25,-0.0)--(-4.25,3.07)--(-1.6,5);
 
 \draw[edge, thick](-4.0,5.5)--(-1.6, 5.05)--(-1.8, 6.4);
 \draw[edge, thick](-6,3.2)--(-4.3, 3.08)--(-4.0, 5.5);
 
 \draw[edge, thick](-7,0.8)--(-5.25,-0.0)--(-6.5,-2);
 
 \draw[edge, thick](-6.5,-2)--(-4.25,-3.1)--(-5,-4.5);
 
 \draw[edge, thick] (-0.5,-6.6)--(-1.6,-5.0)--(-3.4,-5.8)--(-0.5,-6.6);
 
 \draw[edge, thick] (2,-6.6)--(1.6,-5.0)--(4,-5.2)--(4.25,-3.07)--(4.25,-3.07)--(5.5,-3.57); 
 
 \draw[edge, thick] (6.5,-1.5)--(5.25,0.0)--(6.5, 1.5)--(6.5,-1.5);
 
 \draw[edge, thick](2.5, 6.2)--(1.6, 5.05)--(4, 5.2)--(4.25, 3.07)--(5.5, 3.5);
 
 \node[ver] () at (0,0){\scriptsize ${10}$};

 \node[ver] () at (-5.5,1.5){\scriptsize ${5}$};

 \node[ver] () at (-3.5, -4.5){\scriptsize ${k_{_4}}$};
 \node[ver] () at (0.5, -5.8 ){\scriptsize ${5}$};

 \node[ver] () at (6, -2){\scriptsize ${k_{_4}}$};

 \node[ver] () at (5.4, 2){\scriptsize $5$};
 \node[ver] () at (-0.2, 6.2){\scriptsize ${k_{_4}}$};
 
 \end{tikzpicture}

\end{minipage}%
 \begin{minipage}{.20\textwidth} 
 \centering
 \begin{tikzpicture}[scale=0.15]
 
 \draw[edge, thick](-1.6,5)--(1.6,5)--(4.25,3.07)--(5.25,0.0)--(4.25,-3.1)--(1.6,-5.0)--(-1.6,-5.0)--(-4.25,-3.1)--(-5.25,-0.0)--(-4.25,3.07)--(-1.6,5);
 
 \draw[edge, thick](-4.0,5.5)--(-1.6, 5.05)--(-1.8, 6.4);
 \draw[edge, thick](-6,3.2)--(-4.3, 3.08)--(-4.0, 5.5);
 
 \draw[edge, thick](-7,0.8)--(-5.25,-0.0)--(-6.5,-2);
 
 \draw[edge, thick](-6.5,-2)--(-4.25,-3.1)--(-5,-4.5);
 
 \draw[edge, thick] (-0.8,-6.8)--(-1.6,-5.0)--(-3.4,-5.8)--(-0.8,-6.8);
 
 \draw[edge, thick] (1,-6.8)--(1.6,-5.0)--(3.6,-5.5)-- (1,-6.8);

 \draw[edge, thick] (6.5, 0.5)--(5.25,0)--(6.25, -1.8)--(4.25,-3.1)--(5.25,-4);
 
 \draw[edge, thick](2.5, 6.2)--(1.6, 5.05)--(4, 5.2)--(4.25, 3.07)--(5.5, 3.5);
 
 \node[ver] () at (0,0){\scriptsize ${10}$};

 \node[ver] () at (-5.9,1.6){\scriptsize ${k_{_4}}$};

 \node[ver] () at (-3.6, -4.6){\scriptsize $5$};
 \node[ver] () at (0.2, -5.9 ){\scriptsize ${k_{_4}}$};

 \node[ver] () at (3.8, -4.5){\scriptsize $5$};

 \node[ver] () at (6, 2){\scriptsize ${k_{_4}}$};
 \node[ver] () at (0, 6){\scriptsize $5$};
 
 \end{tikzpicture}

\end{minipage}

\captionof{figure}{Neighborhoods of $10$-gons for type $[3, 5, 10, k_4]$}
\label{fig:3512k4} 
\end{figure}

\begin{figure}[ht!]
\tikzstyle{ver}=[]
\tikzstyle{vert}=[circle, draw, fill=black!100, inner sep=0pt, minimum width=4pt]
\tikzstyle{vertex}=[circle, draw, fill=black!00, inner sep=0pt, minimum width=4pt]
\tikzstyle{edge} = [draw,thick,-]
\centering
\begin{tikzpicture}[scale=0.15]

\draw[edge, thick](6,-0.2)--(3.1,-1)--(1,-3.1)--(0,-6);
\draw[edge, thick](0,-6)--(1,-8.9)--(3.1,-11)--(6,-11.8);
\draw[edge, thick](6,-0.2)--(8.9,-1)--(11,-3.1)--(12,-6);
\draw[edge, thick](6,-11.8)--(8.9,-11)--(11,-8.9)--(12,-6);
\draw[edge, thick](6.1,1)--(6,-0.2)--(4.2,0.7)--(3.1,-1)--(2.1,0);
\draw[edge, thick](1,-3.1)--(-0.9,-3.3)--(0,-1.4)--(1,-3.1);
\draw[edge, thick](-1.5,-5.7)--(0,-6)--(-1,-8)--(1,-8.9)--(-0.2,-10.2);
\draw[edge, thick](3.1,-11)--(1.1,-11.4)--(3.1,-12.8)--(3.1,-11);
\draw[edge, thick](6,-11.8)--(4.8,-13.4)--(7.4,-13.2)--(6,-11.8);
\draw[edge, thick](9.1,-12.4)--(8.9,-11)--(11.1,-11.1)--(11,-8.9)--(12.5,-9.5);
\draw[edge, thick](12,-6)--(13.4,-7.3)--(13.4,-4.7)--(12,-6);
\draw[edge, thick](9.8,0.3)--(8.9,-1)--(11.2,-0.8)--(11,-3.1)--(12.5, -2.5);

 \node[ver] () at (6,-6){\scriptsize ${12}$};
 \node[ver] () at (-0.7, -4.6){\scriptsize ${k_{_4}}$};

 \node[ver] () at (1.4, -1.2){\scriptsize $5$};

 \node[ver] () at (8, 0.3){\scriptsize ${k_{_4}}$};

 \node[ver] () at (12.3, -4.3){\scriptsize $5$};
 \node[ver] () at (12.7, -8){\scriptsize ${k_{_4}}$};
 \node[ver] () at (8, -12.3){\scriptsize $5$};
 \node[ver] () at (4.2, -12.2){\scriptsize ${k_{_4}}$};
 \node[ver] () at (1, -10.2){\scriptsize $5$};
 
\end{tikzpicture}

\captionof{figure}{A neighborhood of $12$-gons for type $[3, 5, 12, k_4]$}
\label{fig:35aroundk31} 
 \end{figure}

Each of these neighborhoods can also be viewed as a partial neighborhood of a $k_3$-gon (or the $k_4$-gon) for $k_3 \geq 13$ and extended to full neighborhoods consistent with Observations \ref{A1}--\ref{A2} by replacing $\nabla$- and $\Delta$-triangles in the manners shown in Fig.~\ref{fig:3510extend}.

 \begin{figure}[H] 
\centering
\begin{tikzpicture}[ x=0.30mm, y=0.30mm,
				line cap=round, 
				line join=round,
				thick
				]
				
				%------------------------------------------------------------
				% GROUP 1 (Leftmost)
				%------------------------------------------------------------
				\begin{scope}[shift={(0,0)}]
					% Base Line
					\draw (44, 0) -- (91, 0);
					
					% Vertical Ticks (Strictly starting at y=0)
					\draw (55, 0) -- (55, -15);
					\draw (77, 0) -- (77, -16);
					
					% Shape (Trapezoid - Strictly starting/ending at y=0)
					\draw (66, 0) -- (59, 19) -- (76, 19) -- cycle;
					
					% Labels
					\node[font=\footnotesize] at (75, 6) {5};
					\node[font=\footnotesize] at (57, 6) {$k_3$};
					\node[below, font=\footnotesize] at (67, -2) {$k_4$};
				\end{scope}
				
				% Arrow 1
				\draw[->, >=stealth] (97, 0) -- (114, 0);

				%------------------------------------------------------------
				% GROUP 2
				%------------------------------------------------------------
				\begin{scope}[shift={(0,0)}]
					% Base Line
					\draw (123, 0) -- (170, 0);
					
					% Vertical Ticks (Strict y=0)
					\draw (127, 0) -- (127, -17);
					\draw (166, 0) -- (166, -16);
					
					% Shape (Zigzag - Strictly touching y=0)
					\draw (153, 0) -- (144, 19) -- (134, 0) -- (134, 19);
					\draw (153, 0) -- (153, 19);
					
					% Labels
					\node[font=\footnotesize] at (158, 6) {5};
					\node[font=\footnotesize] at (128, 6) {$k_3$};
					\node[below, font=\footnotesize] at (143, -2) {$k_{4}\!+\!1$};
				\end{scope}
				
				% Arrow 2
				\draw[->, >=stealth] (178, 0) -- (196, 0);

				%------------------------------------------------------------
				% GROUP 3
				%------------------------------------------------------------
				\begin{scope}[shift={(0,0)}]
					% Base Line
					\draw (202, 0) -- (281, 0);
					
					% Vertical Ticks
					\draw (208, 0) -- (208, -16);
					\draw (277, 0) -- (277, -17);
					
					% Shapes (All bases corrected to y=0)
					\draw (216, 0) -- (209, 18) -- (226, 18) -- cycle; % Left Tri
					\draw (251, 0) -- (242, 18) -- (232, 0) -- (232, 18); % Mid Zig
					\draw (251, 0) -- (251, 18);
					
					% Right Tri (Fixed from y=1 to y=0)
					\draw (267, 0) -- (259, 18) -- (274, 18) -- cycle; 
					
					% Labels
					\node[font=\footnotesize] at (224, 6) {5};
					\node[font=\footnotesize] at (274, 6) {5};
					
					\node[font=\footnotesize] at (207, 6) {$k_3$};
					\node[below, font=\footnotesize] at (241, -2) {$k_{4}\!+\!3$};
					\node[font=\footnotesize] at (258, 6) {$k_3$};
				\end{scope}
				
				%--- GROUP 3.5 (Middle Small Segment) ---%
				\begin{scope}[shift={(0,0)}]
					% Base Line
					\draw (326, 0) -- (377, 0);
					
					% Vertical Ticks (Fixed from y=1 to y=0)
					\draw (330, 0) -- (330, -15);
					\draw (368, 0) -- (368, -16);
					
					% Shape (Fixed base points to y=0)
					\draw (362, 0) -- (353, 19) -- (342, 0) -- (342, 19);
					\draw (362, 0) -- (362, 19);
					
					% Labels
					\node[font=\footnotesize] at (366, 6) {5};
					\node[font=\footnotesize] at (333, 6) {$k_3$};
					\node[below, font=\footnotesize] at (350, -2) {$k_4$};
				\end{scope}
				
				% Arrow 3
				\draw[->, >=stealth] (382, 0) -- (402, 0);

				%------------------------------------------------------------
				% GROUP 4 (Far Right)
				%------------------------------------------------------------
				\begin{scope}[shift={(0,0)}]
					% Base Line
					\draw (411, 0) -- (509, 0);
					
					% Vertical Ticks
					\draw (418, 0) -- (418, -16);
					\draw (506, 0) -- (506, -16);

					% Shapes (Bases checked for y=0)
					\draw (446, 0) -- (436, 19) -- (425, 0) -- (425, 19); % Left Zig
					\draw (446, 0) -- (446, 19);
					
					\draw (464, 0) -- (455, 20) -- (471, 20) -- cycle; % Middle Trap
					
					\draw (500, 0) -- (490, 21) -- (480, 0) -- (480, 21); % Right Zig
					\draw (500, 0) -- (500, 21);
					
					% Labels
					\node[font=\footnotesize] at (450, 6) {5};
					\node[font=\footnotesize] at (505, 6) {5};
					
					\node[font=\footnotesize] at (418, 6) {$k_3$};
					\node[below, font=\footnotesize] at (448, -2) {$k_{4}\!+\!3$};
					\node[font=\footnotesize] at (472, 6) {$k_3$};
				\end{scope}
				
			\end{tikzpicture}

\captionof{figure}{Extension of partial neighborhoods} 
\label{fig:3510extend} 

\end{figure}
We claim that the existing partial neighborhood of each $k_3$-gon or $k_4$-gon, as well as the neighborhood forced upon it during the process of covering adjacent faces, can be extended to one of the neighborhoods shown in Fig.~\ref{fig:35aroundk3} for $k_3=10$, or in Fig~\ref{fig:3512k4} for $k_3=12$, and to the extensions of these neighborhood as shown in Fig~\ref{fig:3510extend} for $k_3>12$. We verify the claim only for two non-trivial cases, the remaining cases are straightforward.
 \par
For the sequence $\{5(3), k_3(6), 5(3)\}$, we have the following partial neighborhood of $k_3$-gon for $k_3 \neq 10, 13$ shown in Fig.~\ref{fig:35aroundk3} (a), and for $k_3=10, 13$ shown in Fig.~\ref{fig:35aroundk3} (b) (without the faces with dotted edges). For $k_3=10$, the partial neighborhood can be extended to the neighborhood shown in the second figure in Fig.~\ref{fig:3512k4} (also in Fig.~\ref{fig:35aroundk3} (b)). For $k_3=12$, it can be extended to the neighborhood shown in Fig.~\ref{fig:35aroundk31}. For larger values of $k_3$, we can use the extension procedure shown in Fig.~\ref{fig:3510extend}.

\begin{figure}[H]
\tikzstyle{ver}=[]
\tikzstyle{vert}=[circle, draw, fill=black!100, inner sep=0pt, minimum width=4pt]
\tikzstyle{vertex}=[circle, draw, fill=black!00, inner sep=0pt, minimum width=4pt]
\tikzstyle{edge} = [draw,thick,-]
\centering
\subfigure[For $k_3 > 10$]{\begin{tikzpicture}[scale=0.20]
\draw[edge, thick](10,20)--(24,20);
\draw[edge, thick](10,14)--(26,14);
\draw[edge, thick](12,20)--(11,17.5)--(12,14);
\draw[edge, thick](9.5, 18.5)--(11,17.5)--(9.8,16)--(12,14);
\draw[edge, thick](25,22)--(24,20)--(22.5,22)--(21,20)--(20.5,22);
\draw[edge, thick](24,14)--(25,17.5)--(24,20);
\draw[edge, thick](24,14)--(26.5,16.5)--(25,17.5)--(26.5,18.5);
\draw[edge, thick](12,20)--(13.5,22)--(15,20)--(16,22);
\node[ver] () at (11,19){\scriptsize ${5}$};
\node[ver] () at (25.5,19){\scriptsize ${5}$};
\node[ver] () at (17,17){\scriptsize ${k_3}$};
\node[ver] () at (20,21){\scriptsize ${k_4}$};
\node[ver] () at (16.5,21){\scriptsize ${k_4}$};
\node[ver] () at (8.5,20){\scriptsize ${\partial X_{k}}$};
\node[ver] () at (8.5,13.7){\scriptsize ${\partial X_{k-1}}$};
\end{tikzpicture}}
\subfigure[For $k_3=10$]{ \begin{tikzpicture}[scale=0.20]
\draw[edge, thick](10,20)--(27,20);
\draw[edge, thick](10,14)--(26,14);
\draw[edge, thick](12,20)--(11,17.5)--(12,14);
\draw[edge, thick](9.5, 18.5)--(11,17.5)--(9.8,16)--(12,14);
\draw[dotted, thick](22.5,22)--(22.5,20.1)--(21.5,22)--(20.5,20)--(20.5,22);
\draw[dotted, thick](16,22)--(17,20)--(18,22)--(16,22);
\draw[edge, thick](26.5,20)--(25,22)--(24,20);
\draw[edge, thick](24,14)--(25,17.5)--(24,20);
\draw[edge, thick](24,14)--(26.5,16.5)--(25,17.5)--(26.5,18.5);
\draw[edge, thick](12,20)--(13,22)--(14,20)--(14.5,22);
\node[ver] () at (11,19){\scriptsize ${5}$};
\node[ver] () at (25.5,19){\scriptsize ${5}$};
\node[ver] () at (19,21){\scriptsize ${5}$};
\node[ver] () at (17,17){\scriptsize ${k_3}$};
\node[ver] () at (23.5,21){\scriptsize ${k_4}$};
\node[ver] () at (15.5,21){\scriptsize ${k_4}$};
\node[ver] () at (8.5,20){\scriptsize ${\partial X_{k}}$};
\node[ver] () at (8.5,13.7){\scriptsize ${\partial X_{k-1}}$};
\end{tikzpicture}}

\captionof{figure}{Neighborhood of $k_3$-gon for $[3, 5, k_3, k_4]$}
\label{fig:35aroundk3}

\end{figure}

For the sequence $\{5(4), k_3(4), 5(4)\}$, we already have $\Delta$-triangles attached to two consecutive vertices of the $k_3(4)$-gon. Here, we note that the first and the last $5(4)$-gons are preceded and succeeded by a $k_4$-gon, respectively. Therefore, as prescribed above in Fig.~\ref{fig:357part1}(a), the triangles are attached to these $5(4)$-gons so that they are $\nabla$-triangles in the neighborhood of the $k_3(4)$-gon. Thus, we can extend the partial neighborhood of $k_3(4)$-gon to a neighborhood as shown in Fig.~\ref{fig:35aroundk3} for $k_3=10$. For $k_3> 11$, we use the same procedure of extension as before. 
\begin{figure}[H]
\centering
\begin{minipage}{.20\textwidth} 
\includegraphics[scale=0.15]{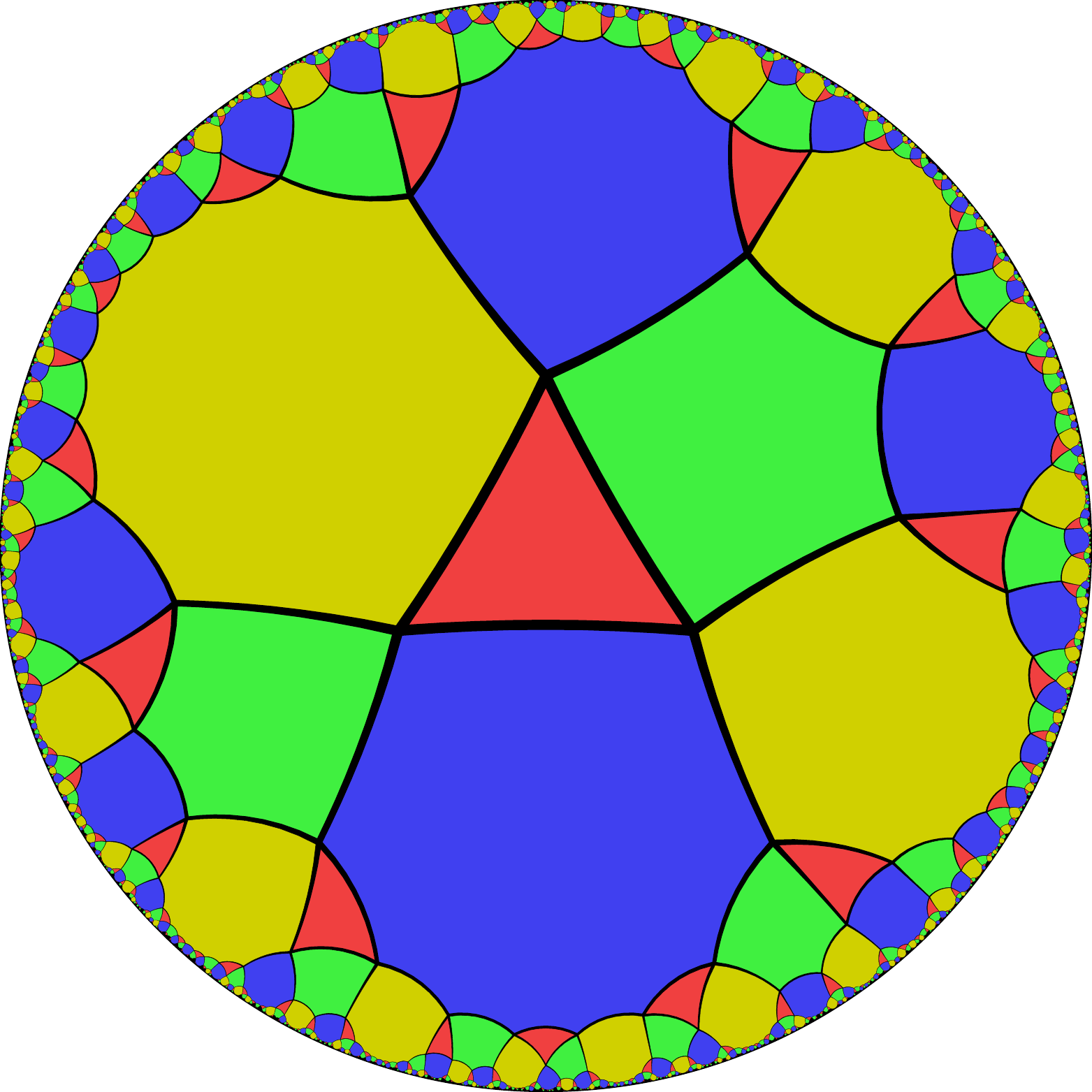}
\end{minipage}
\captionof{figure}{A tiling of type $[3, 5, 10, 12]$}
\label{351012}
\end{figure}
One can now verify that all the tuples have been covered for \enquote{the existence} part of the theorem. This completes the proof.
\end{proof}
For tuples of degree $3$, pseudo homogeneous and homogeneous tiling are the same. We state the following known result for future reference.
\begin{proposition}\label{dg3} \cite{DG18}
A cyclic tuple $k = [k_1, k_2, k_3]$ is the type of a pseudo-homogeneous tiling on $\mathbb{H}^2$ if and only if one of the following holds: \newline
 $\bullet$ $k=[p^3] $ where $p \geq 7$, or \newline
$\bullet$ $ k=[2n, 2n, q]$ where $2n \neq q$, and $\frac{1}{n} +\frac{1}{q}< \frac{1}{2}.$ \newline
$\bullet$ $ k=[2l, 2m, 2n]$ where $l, m, n$ are distinct $\frac{1}{l} +\frac{1}{m} +\frac{1}{n}< \frac{1}{2}.$

 \end{proposition}
 
\begin{remark} [Uniqueness of tilings] \label{unique} 
By a careful study of the classification of tuples that admit tiling given above, we see that more than one non-equivalent ways of constructing neighborhoods around at least one of the faces for all the tuples except the following: 
 \newline
 a) $[k^r]$ with $r \geq 3$ b) $[3, 3, 3p, 3q]$ c) $[3, 4, 4p, 4q]$, d) all the tuples of degree $3$ that admits a tiling. 
 \newline
 This flexibility in constructing neighborhoods ensures the existence of uncountably many tilings for all the remaining types.

\end{remark}

 \subsection{Tilings of $\mathbb{H}^2$ by regular polygons} \label{regtile}
 Let us recall the hyperbolic cosine rule for a hyperbolic regular $n$-gon, $n<\infty$, with side length $a$ and inner angle $\theta$ 
 \begin{equation} \label{eq:sidelength}
 \cosh (a/2)= \frac{\cos (\pi/n)}{ \sin (\theta/2)} 
 \end{equation}
 Consequently, given a tuple $\mathfrak{k}=[k_1, k_2,$ $ \cdots, k_d]$ with $k_i< \infty, \vartheta(\mathfrak{k})>2$, there is a unique length, denoted by $\ell_{\mathfrak{k}}$, such that there exist hyperbolic regular $k_i$-gons with geodesic sides for all $i=1, \cdots, d$, which fits around a vertex, or equivalently, their inner angles $\theta_i$ sum to $2 \pi$, that is, $ \sum^d_{i=1} \theta_i=2 \pi$. Thus, a pseudo-homogeneous tiling produces a tiling of $\mathbb{H}^2$ by regular polygons of finite sizes. 
 \\
We are now ready to establish the final component of the classification process as stated in Theorem \ref{solv1}.
\begin{proof} 

Let $V_{\text{null}}$ denote the set of all such tuples satisfying the angle-sum condition that do not admit a pseudo-homogeneous tiling of $\mathbb{H}^2$. Based on the exhaustive classification established in the preceding sections, we observe that $V_{\text{null}}$ consists exclusively of tuples of degree $3$ and $4$. 
To complete the proof of the theorem, it is sufficient to demonstrate that no global tiling can exist if it contains two adjacent vertices of distinct types $\mathfrak{k}_1, \mathfrak{k}_2 \in V_{\text{null}}$. Recall that by the cosine rule \ref{eq:sidelength}, there exists a unique length $\ell_\mathfrak{k}$ of the sides of the regular polygons of finite sizes forming a fan of type $\mathfrak{k}$. We establish this impossibility by showing a combinatorial contradiction or a mismatch of associated lengths $\ell_{\mathfrak{k}_1} \neq \ell_{\mathfrak{k}_2}$.
 \\ 
 Let us denote by $\theta_{\ell}(r)$, the inner-angles of a hyperbolic regular polygon with side length $\ell$. The following can be verified using the cosine formula: 

 \begin{equation} \label{estimate}
 \cosh(\ell_{[3, 5, k_3, k_4]}/2) \in (1.1, 1.2)
\end{equation}
Since two adjacent vertices of distinct types must have two common polygons, the possibility that $\mathfrak{k}_1$ and $\mathfrak{k}_2$ are both of degree $3$ is ruled out. 
\par
 \textbf{Case A}. $\deg(\mathfrak{k}_1) =3$ and $\deg(\mathfrak{k}_2)=4$; we have the following subcases:.
\par
\[ \textbf{i)} [3, j, k], \mathfrak{k}_2=[3, i, j, l]; \textbf{ii)} \mathfrak{k}_1 =[i, j, k], \mathfrak{k}_2=[3, 5, j, k]; \textbf{iii)} \mathfrak{k}_1 =[i, j, k], \mathfrak{k}_2=[3, 4, j, k].\]
 \textbf{A. i)} Using the formula \ref{eq:sidelength}, one can verify that $\cosh(\ell_{\mathfrak{k}_1 } /2)<\cosh(\ell_{\mathfrak{k}_2}/2)$ for $i=4, 5$ and for all $j, l, k$. Hence, $\ell_{\mathfrak{k}_1} \neq \ell_{\mathfrak{k}_2}$. 
 \newline
 \textbf{A. ii)} We must have $\theta_{\ell}(3) + \theta_{\ell}(5)= \theta_{\ell}(i)$ for $\ell=\ell_{\mathfrak{k}_1}=\ell_{\mathfrak{k}_2}$. Using the side length estimate \ref{estimate}, we can verify that $\theta_{\ell}(3)>50^{\circ}$ and $\theta_{\ell}(5)>84^{\circ}$, but $\theta_{\ell}(i)<134^{\circ}$. Hence, $\ell_{\mathfrak{k}_1} \neq \ell_{\mathfrak{k}_2}$. 
\newline
 \textbf{A. iii)} Starting with two vertices, say of $v_1$ and $v_2$ of types $\mathfrak{k}_1$ and $\mathfrak{k}_2$, respectively, the only possible valid configuration around $i$-gon is shown in Fig.~\ref{fig:34deg3}. Consequently, $i$ must be even. Similarly, for a valid neighborhoods of the $j$ and $k$-gons, $j$ and $k$ must be even. This implies $\mathfrak{k}_1 \notin V_{not}$. A similar argument works for $i =j$.
 \begin{figure}[ht!]
\tikzstyle{ver}=[]
\tikzstyle{vert}=[circle, draw, fill=black!100, inner sep=0pt, minimum width=4pt]
\tikzstyle{vertex}=[circle, draw, fill=black!00, inner sep=0pt, minimum width=4pt]
\tikzstyle{edge} = [draw,thick,-]
\centering
\begin{tikzpicture}[scale=0.15]

\draw[edge, thick](16,12)--(11.5,9)--(11.5,4.8);
\draw[edge, thick](16,12)--(22,10.3)--(22,4.6);
\draw[edge, thick](11.5,4.8)--(14.4,1.7);
\draw[edge, thick](22,4.6)--(17.8,1.7)--(14.4,1.7);
\draw[edge, thick](16,12)--(22,14.5)--(22,10.3);
\draw[edge, thick](22,14.5)--(27.6,16.4)--(27.6,12.2)--(22,10.3);
\draw[edge, thick](27.6,16.4)--(32.5,15)--(27.6,12.2);
\draw[edge, thick](22,4.6)--(24.8,2)--(28.2,2.5)--(30.8,3.8)--(31.5,6.2)--(31.3,9.4)--(27.6,12.2);
\draw[edge, thick](17.8,1.7)--(19,-2.5)--(22.5,-2.5)--(24.7,-1)--(24.8,2);
\draw[edge, thick](16,12)--(11.8,14.5)--(7.4,12)--(11.5,9);
\draw[edge, thick](11.8,14.5)--(9.3,16.6)--(14,16.6)--(11.8,14.5);
\draw[edge, thick](7.4,12)--(5,14)--(5,10)--(7.4,12);
\draw[edge, thick](11.5,9)--(7.4,7)--(11.5,4.8);
\draw[edge, thick](7.4,7)--(5,8.6);
\draw[edge, thick](7.4,7)--(5,4.5);
\draw[edge, thick](11.5,4.8)--(8,2)--(5,4.5);
\node[ver] () at (11.8,15.6){\scriptsize ${3}$};
\node[ver] () at (6,12){\scriptsize ${3}$};
\node[ver] () at (10,7){\scriptsize ${3}$};
\node[ver] () at (19.7,12.1){\scriptsize ${3}$};
\node[ver] () at (29.3,14.5){\scriptsize ${3}$};
\node[ver] () at (11.6,11.5){\scriptsize ${4}$};
\node[ver] () at (8,4.5){\scriptsize ${4}$};
\node[ver] () at (24.5,13.2){\scriptsize ${4}$};
\node[ver] () at (8.7,14.3){\scriptsize ${i}$};
\node[ver] () at (5.6,6.8){\scriptsize ${i}$};
\node[ver] () at (16.5,6.6){\scriptsize ${i}$};
\node[ver] () at (30.2,12.2){\scriptsize ${i}$};
\node[ver] () at (14.8,14.2){\scriptsize ${j}$};
\node[ver] () at (25.5,7){\scriptsize ${j}$};
\node[ver] () at (11,2.6){\scriptsize ${j}$};
\node[ver] () at (7.8,9.3){\scriptsize ${j}$};
\node[ver] () at (20.5,0.3){\scriptsize ${k}$};
\node[ver] () at (22,2.8){\scriptsize ${ v_{1}}$};
\node[ver] () at (20.8,9.5){\scriptsize ${ v_{2}}$};

\end{tikzpicture}
\captionof{figure}{Configuration for mixed type $\{[i, j, k], [3, 4, i, j]\}$}
\label{fig:34deg3} 
\end{figure}
\par
 \textbf{Case B}. $\deg(\mathfrak{k}_1) =4$ and $\deg(\mathfrak{k}_2)=4$. We have the following subcases:
 \[
 \textbf{i)} \ \mathfrak{k}_1= [3, 4, i, j]; \textbf{ii)}\mathfrak{k}_1=[3, 3, j, l]; \textbf{iii)}\mathfrak{k}_1=[3, 4, i, j], \mathfrak{k}_2=[3, 5, i, k] \]
 \textbf{B. i)} We first note that fans around two adjacent vertices of types $\mathfrak{k}_1 =[3, 3, i, j]$ and $\mathfrak{k}_2=[3, 3, i', j']$ must share three polygons, consequently, the fourth as well. Hence our claim follows.
\\
\textbf{B. ii)} Starting with two such adjacent vertex types, the only possible configuration around the $l$-gon is shown in Fig.~\ref{fig:3334}. Here we use the easy to verify fact that a $4$-gon in such a tiling must be of type-2. Consequently, both $j$ and $l$ must be multiples of $3$. But this implies, by Theorem \ref{431prop}, Assertion 1, that $[3, 3, j, l] \notin V_{not}$.

\begin{figure}[H]
\tikzstyle{ver}=[]
\tikzstyle{vert}=[circle, draw, fill=black!100, inner sep=0pt, minimum width=4pt]
\tikzstyle{vertex}=[circle, draw, fill=black!00, inner sep=0pt, minimum width=4pt]
\tikzstyle{edge} = [draw,thick,-]
\centering
\begin{tikzpicture}[scale=0.20]

\draw[edge, thick](11,16.5)--(8.5,18)--(8.5,20)--(10.5,21.4)--(13,21.5)--(15,21)--(18,21)--(20,21.5)--(22,21.5)--(23,20.5)--(23,18.5)--(22,17);
\draw[edge, thick](11,16.5)--(15,17)--(18,17)--(22,17);
\draw[edge, thick](18,17)--(16.5,14)--(15,17);
\draw[edge, thick](11,16.5)--(9.5,15.5)--(9,14)--(10,11.5)--(12.5,10)--(15,10)--(16.5,14)--(18.5,10)--(19.5,10)--(21,10.5)--(22.5,12)--(23,14)--(22,17);
\draw[edge, thick](15,17)--(15,21);
\draw[edge, thick](18,17)--(18,21);
\draw[edge, thick](15,10)--(18.5,10);
\draw[edge, thick](25,13)--(26,12)--(24.5,10.5)--(22.5,12)--(25,13)--(23,14)--(24.5,16)--(22,17);
\draw[edge, thick](24.5,16)--(25.7,17.8)--(23,18.5)--(24.5,20)--(25.7,17.8);
\draw[edge, thick](24.5,16)--(26,15);

\node[ver] () at (12.5,19){\scriptsize ${i}$};
\node[ver] () at (16.5,19){\scriptsize ${4}$};
\node[ver] () at (19.5,19){\scriptsize ${k}$};
\node[ver] () at (19.5,14){\scriptsize ${l}$};
\node[ver] () at (13, 13.5){\scriptsize ${j}$};
\node[ver] () at (16.5, 16){\scriptsize ${3}$};
\node[ver] () at (16.5, 11.5){\scriptsize ${3}$};
\node[ver] () at (23.2, 15.5){\scriptsize ${3}$};
\node[ver] () at (23.5, 13.1){\scriptsize ${3}$};
\node[ver] () at (22.5,11){\scriptsize ${k}$};
\node[ver] () at (24.5,14.5){\scriptsize ${j}$};
\node[ver] () at (24.5,12){\scriptsize ${4}$};
\node[ver] () at (24.5, 19){\scriptsize ${3}$};
\node[ver] () at (24, 17.4){\scriptsize ${4}$};
\end{tikzpicture}
\captionof{figure}{Configuration for mixed type $\{[3, 3, j, k], [3, 4, i, j]\}$}
\label{fig:3334} 
\end{figure}
 \textbf{B. iii)} For tuples $\mathfrak{k}_1= [3, 4, i, j]$ and $\mathfrak{k}_2 =[3, 5, i, k]$, $i< j$, $i < k$, we claim that there can only be type-2 4-gons in such a tiling. First note that the only possible neighborhood of any triangle is of the form $[4, i, l, 5, k, j]$. For a type-3 $4$-gon, we cannot consistently form a neighborhood of the triangles adjacent to it. A type-4 4-gon leads to Fig.~\ref{fig:34deg4}. It is straightforward to see that the $l$- and $m$-gon both can be neither $4$-gons nor $5$-gons. If $m$-gon is a $5$-gon, we cannot form a valid neighborhood around it or around the triangles adjacent to it for any values of $i, j, k$. If $m$ is a $4$-gon then we must have $k=j$ and $l=i$. It follows, by a similar argument used earlier in the proof of Theorem \ref{431prop} for tuples $ [3, 4, i, j]$ (see also Fig.~\ref{fig:34k3k4}), that $i$ and $j$ must be even. This contradicts the fact that $\mathfrak{k}_1 \in V_{not}$, and thus proves our claim.
 \begin{figure}[ht!]
\tikzstyle{ver}=[]
\tikzstyle{vert}=[circle, draw, fill=black!100, inner sep=0pt, minimum width=4pt]
\tikzstyle{vertex}=[circle, draw, fill=black!00, inner sep=0pt, minimum width=4pt]
\tikzstyle{edge} = [draw,thick,-]
\centering
\begin{tikzpicture}[scale=0.15]

\draw[edge, thick](16,12)--(11.5,9)--(11.5,4.8);
\draw[edge, thick](16,12)--(22,10.3)--(22,4.6);
\draw[edge, thick](11.5,4.8)--(14.4,1.7);
\draw[edge, thick](22,4.6)--(17.8,1.7)--(14.4,1.7);
\draw[edge, thick](16,12)--(22,14.5)--(22,10.3);
\draw[edge, thick](22,14.5)--(25.6,16.1);
\draw[edge, thick](22,14.5)--(22.2,16.8);
\draw[edge, thick](27.6,12.2)--(22,10.3);
\draw[edge, thick](22,4.6)--(24.8,2)--(28.2,2.5)--(30.8,3.8)--(31.5,6.2)--(31.3,9.4)--(27.6,12.2);
\draw[edge, thick](16,12)--(11.8,14.5)--(7.4,12)--(11.5,9);
\draw[edge, thick](11.8,14.5)--(9.3,16.6)--(14,16.6)--(11.8,14.5);
\draw[edge, thick](7.4,12)--(5,14)--(5,10)--(7.4,12);
\draw[edge, thick](11.5,9)--(7.4,7)--(11.5,4.8);
\draw[edge, thick](7.4,7)--(5,8.6);
\draw[edge, thick](7.4,7)--(5,4.5);

\node[ver] () at (11.8,15.6){\scriptsize ${3}$};
\node[ver] () at (6,12){\scriptsize ${3}$};
\node[ver] () at (10,7){\scriptsize ${3}$};
\node[ver] () at (19.7,12.1){\scriptsize ${3}$};
\node[ver] () at (11.6,11.5){\scriptsize ${4}$};

\node[ver] () at (24.5,13.2){\scriptsize ${m}$};
\node[ver] () at (8.7,14.3){\scriptsize ${i}$};
\node[ver] () at (16.5,6.6){\scriptsize ${i}$};
\node[ver] () at (15.8,14.2){\scriptsize ${j}$};
\node[ver] () at (25.5,7){\scriptsize ${k}$};
\node[ver] () at (23,16){\scriptsize ${l}$};

\node[ver] () at (7.8,9.3){\scriptsize ${j}$};
\node[ver] () at (22,2.8){\scriptsize ${ v_{1}}$};
\node[ver] () at (20.8,9.5){\scriptsize ${ v_{2}}$};

\end{tikzpicture}
\captionof{figure}{Configuration for mixed type $\{[3, 4, i, j], [3, 5, i, k]\}$}
\label{fig:34deg4} 
\end{figure}
Starting with a type-2 $4$-gon, the only possible neighborhood of the triangle adjacent to it with all the vertex types belonging to $V_{not}$ is shown in Fig.~\ref{fig:3435mixed}. Consequently, there must be two adjacent vertices of types 
$[3, 5, j, m]$ and $[3, 4, 5, m]$ for some $m\geq 5$. Then using formula~\ref{eq:sidelength} one can show that $\cosh(\ell_{[3, 4, 5, k]}/2)<1.1$. Therefore, $\ell_{[3, 5, j, m]}> \ell_{[3, 4, 5, k]}$ for all $i, k, m$ such that $\mathfrak{k}_1, \mathfrak{k}_2 \in V_{not}$.

\begin{figure}[H]
\tikzstyle{ver}=[]
\tikzstyle{vert}=[circle, draw, fill=black!100, inner sep=0pt, minimum width=4pt]
\tikzstyle{vertex}=[circle, draw, fill=black!00, inner sep=0pt, minimum width=4pt]
\tikzstyle{edge} = [draw,thick,-]
\centering
\begin{tikzpicture}[scale=0.15]

\draw[edge, thick](11,17)--(9,18)--(8.5,20)--(10.5,21.5)--(13,21.5)--(15,21)--(18,21)--(20,21.5)--(22,21.5)--(23,20.5)--(23,18.5)--(22,17);
\draw[edge, thick](11,17)--(15,17)--(18,17)--(22,17);
\draw[edge, thick](18,17)--(16.5,14)--(15,17);
\draw[edge, thick](11,17)--(9,15.5)--(9,14)--(10,12.5)--(13,12)--(16.5,14)--(20.5,13)--(22.5,15)--(22,17);
\draw[edge, thick](15,17)--(15,21);
\draw[edge, thick](18,17)--(18,21);
\draw[edge, thick](13,12)--(15, 10)--(18,10)--(19.5,10.5)--(20,11)--(20.5,13);

\node[ver] () at (12,19){\scriptsize ${i}$};
\node[ver] () at (16.5,19){\scriptsize ${4}$};
\node[ver] () at (19.5,19){\scriptsize ${k}$};
\node[ver] () at (19,15){\scriptsize ${5}$};
\node[ver] () at (13, 14.5){\scriptsize ${j}$};
\node[ver] () at (16.5, 16){\scriptsize ${3}$};
\node[ver] () at (16.5, 12){\scriptsize ${m}$};
\end{tikzpicture}

\captionof{figure}{Configuration for mixed type $\{[3, 4, i, j], [3, 5, j, k]\}$}
\label{fig:3435mixed} 

\end{figure}

\end{proof}
\subsection{Summary of Classification of Regular Polygon Protosets}
The primary result of the preceding sections is the complete classification of finite sets of regular polygons (protosets) that admit a tiling of $\mathbb{H}^2$. As established in Theorem~\ref{solv1}, the admissibility of a global tiling for any protoset is equivalent to the existence of a pseudo-homogeneous tiling of a specific type. We summarize these results in Table~\ref{tab:classification_summary}, providing a complete characterization of admissible and forbidden configurations.

\begin{table}[H]
\centering
\small
\renewcommand{\arraystretch}{1.3}
\begin{tabular}{|c|l|l|}
\hline
\textbf{Degree} & \textbf{Admissible Tuples (Tiling Exists)} & \textbf{Forbidden Tuples (No Tiling)} \\ \hline
$d \geq 6$ & All tuples satisfying $\vartheta(\mathfrak{k}) \geq 2$ & None \\ \hline
$d = 5$ & All tuples except $[3^5]$ & $[3^5]$ (Spherical) \\ \hline
$d = 4$ & 1. All tuples with $k_i \geq 4$ & 1. $[3, 3, k_3, k_4]$ ($k_3, k_4 \not\equiv 0 \pmod 3$) \\
        & 2. $[3, 3, 3p, 3q]$ & 2. $[3, 4, 5, 5]$ and $[3, 4, 6, 8]$ \\
        & 3. $[3, 4, 2p, 2q]$ & 3. $[3, 4, k_3, k_4]$ ($k_3, k_4$ not both even) \\
        & 4. $[3, 5, p, p]$ ($p \geq 5$) & 4. $[3, 5, 10, 11]$  \\
        & 5. $[3, 5, 10, k_4]$ ($k_4 \geq 12$) & 5. $[3, 5, k_3, k_4]$ ($k_3 \leq 11, k_3 \neq 10$) \\
        & 6. $[3, 5, k_3, k_4]$ ($k_3 \geq 12, k_4 \geq k_3$) & \\
        & 7. $[3, k_2, k_3, k_4]$ ($k_i \geq 6$) & \\ \hline
$d = 3$ & 1. $[p^3]$ ($p \geq 7$) & All other degree 3 tuples \\
        & 2. $[2n, 2n, q]$ ($\frac{1}{n} + \frac{1}{q} < \frac{1}{2}$) & \\
        & 3. $[2l, 2m, 2n]$ ($\frac{1}{l} + \frac{1}{m} + \frac{1}{n} < \frac{1}{2}$) & \\ \hline
\end{tabular}
\caption{Classification of regular polygon protosets that tile the hyperbolic plane.}
\label{tab:classification_summary}
\end{table}

\subsection{Heesch numbers of protosets of regular polygons} \label{heeschregular}
 For a tuple $\mathfrak{k}$ with $\vartheta(\mathfrak{k})>2$, the \textit{Heesch number} is defined to be the maximal nonnegative integer $r$ such that the tuple admits a pseudo-homogeneous tiling of $r$ complete layers. By convention, if the tuple admits a tiling of the entire plane, its Heesch number is defined to be infinite \cite{BB15}. Similarly, for a protoset $\mathcal{F}$, the Heesch number is the maximal nonnegative integer $r$ such that each prototile from $\mathcal{F}$ can be surrounded by $r$ layers of tiling. The \textit{Heesch problem} asks which integers can occur as Heesch numbers for a given class of tiles. From our classification of tuples that do admit pseudo-homogeneous tilings $\mathbb{H}^2$, it becomes obvious that the Heesch numbers are bounded for all tuples that do not admit a tiling. 
 \par
The set of tuples with finite Heesch number is obviously the set $V_{not}$. The proof of non-existence of tiling for those tuples in essence tracks down the Heesch numbers which we list them below. 
 
\[ \mathfrak{H}([3, 3, k_3, k_4])= 0; \mathfrak{H}([3, 4, 5, 5])= 5; \mathfrak{H}([3, 4, 6, 8])= 3; \mathfrak{H}([3, 4, k_3, k_4])= 1 \]
\[ \mathfrak{H}([3, 5, 6, k_4]) = 1, \mathfrak{H}([3, 5, 9, k_4]) =1, \mathfrak{H}([3, 5, 7, k_4]) = 2, \mathfrak{H}([3, 5, 11, k_4]) = 3\]

To produce the Heesch numbers for any protosets of regular polygons, we need only to consider protosets of regular polygons whose size must form tuple of mixed types. We have shown in Theorem \ref{solv1} that for most of those mixed tuple, there is a length mismatch, that is, $\ell (\mathfrak{k}_1) \neq \ell(\mathfrak{k}_2)$ for any $\mathfrak{k}_1, \mathfrak{k}_2 \in V_{not}$, $\mathfrak{k}_1 \neq \mathfrak{k}_2$. It will be shown in a future article that the same holds for all pair of tuples in $V_{not}$.

\section{Aperiodic tiles} \label{aperiodic}
In the search for an aperiodic set of regular tiles, we focus on one of the protosets that poses the strongest constraint in constructing the tiling, namely, the tuple $[3, 5, k_3, k_4]$ presented above in Theorem \ref{431prop}. 
\\
Let us denote by $\mathcal{A}$ the set of regular polygons of size $[3, 5, k_3, k_4]$, $10 \leq k_3 <k_4$, $k_3 \neq 11$, such that they form a fan around a vertex.
 \par
\subsection{Weak aperiodicity of tile set $\mathcal{A}$} For a tiling of type $[3, 5, k_3, k_4], 10 \leq k_3 <k_4$, $k_3 \neq 11$, let $c_1(T, P)$ and $c_2(T, P)$ denote the number of (triangle, pentagon) pairs so that the intersection $T\cap P$ is a single vertex and an edge, respectively. 
\\
Suppose a strongly periodic tiling (tiling with compact quotient) of this type exist. The total incidence counts for $c_1(T, P)$ and $c_2(T, P)$ can be determined by evaluating the local constraints from two perspectives: 
\begin{itemize}
 \item From the perspective of the triangles: By Observation \ref{A1}, each triangle in such a tiling is attached to exactly one pentagon by a single vertex and exactly one pentagon by an edge, implying the ratio $c_1(T, P): c_2(T, P)$ is $1:1$.
 \item From the perspective of the pentagons: By Observations \ref{A2}, each pentagon is incident with three triangles at distinct vertices, which yields the ratio $c_1(T, P): c_2(T, P)$ is $1:3$.
\end{itemize}
This contradiction shows that no strongly periodic tiling such type exists. It follows that  the set $\mathcal{A}$ of regular polygons is weakly aperiodic.

\begin{proposition}
There exists a tiling of $\mathbb{H}^2$ of type $[3, 5, k_3, k_4]$ with infinite cyclic symmetry. Hence the tile set $\mathcal{A}$ is not strongly aperiodic. 
\end{proposition}
\begin{proof} We will construct an infinite strip of tiling between two ultraparallel piecewise hypercycles so that a tiling of $\mathbb{H}^2$ can be constructed by joining the copies of the strip using repeated application of an isometry of hyperbolic type. 
\\
Let $S$ denote the chain consisting of an infinite sequence (in the forward direction) of regular triangles and regular pentagons (see Figure \ref{chain}) from the tile set $\mathcal{A}$,  along with their neighborhoods satisfying the following adjacency rules: 
\\
(1) for every $i \in \mathbb{Z}$, the triangle $T_i$ and the pentagon $P_i$ intersect at exactly one shared geodesic edge, denoted $e_i$ and the pentagon $P_i$.
\\
 (2) the subsequent triangle $T_{i+1}$ intersect at exactly one shared vertex, denoted $v_i$ and the shared vertex $v_i$ is strictly the vertex of $P_i$ opposite the edge $e_i$.
 \\
 Note that for a tiling of type $[3, 5, k_3, k_4]$, the conditions Observations \ref{A1}--\ref{A2} forces every triangle and pentagon to lie in one such infinite chain.

\begin{figure}[H]
\centering

	\begin{tikzpicture}[scale=0.6, thick, transform shape] 
		
		% =========================================================
		% PARAMETERS
		% =========================================================
		\def\s{1.5} % Main side length
		\def\sBranch{0.85} % Side length for Branch Triangles
		\def\sOut{1} % Length of the outward edges
		
		% NEW: Distance and angles for the local floating labels
		\def\dLabel{0.55} % Distance of the k_i, k_j labels from the spine vertices
		
		% Geometric Constants
		\pgfmathsetmacro{\hTri}{\s*sqrt(3)/2} 
		\pgfmathsetmacro{\rPent}{\s/(2*sin(36))} 
		\pgfmathsetmacro{\aPent}{\s/(2*tan(36))} 
		\pgfmathsetmacro{\wPent}{\s/2 * cot(18)} 
		\pgfmathsetmacro{\xStep}{\wPent + \hTri} 
		
		% =========================================================
		% 1. LEFT UNIT (i-1)
		% =========================================================
		\begin{scope}[shift={(-\xStep, 0)}]
			\coordinate (T-top) at (0, \s/2);
			\coordinate (T-bot) at (0, -\s/2);
			\coordinate (T-left) at (-\hTri, 0);
			
			\coordinate (P-center) at (\aPent, 0);
			\coordinate (P-v1) at ($(P-center) + (144:\rPent)$); 
			\coordinate (P-v2) at ($(P-center) + (72:\rPent)$); 
			\coordinate (P-v3) at ($(P-center) + (0:\rPent)$); 
			\coordinate (P-v4) at ($(P-center) + (288:\rPent)$); 
			\coordinate (P-v5) at ($(P-center) + (216:\rPent)$); 
			
			\coordinate (TopBranch-left) at ($(P-v2) + (102:\sBranch)$);
			\coordinate (TopBranch-right) at ($(P-v2) + (42:\sBranch)$);
			\coordinate (BotBranch-left) at ($(P-v4) + (258:\sBranch)$);
			\coordinate (BotBranch-right) at ($(P-v4) + (318:\sBranch)$);
			
			\coordinate (Out-top) at ($(T-top) + (114:\sOut)$); 
			\coordinate (Out-bot) at ($(T-bot) + (246:\sOut)$); 
			
			% Draw Shapes
			\draw (T-top) -- (T-bot) -- (T-left) -- cycle;
			\draw (P-v1) -- (P-v2) -- (P-v3) -- (P-v4) -- (P-v5) -- cycle;
			\draw (P-v2) -- (TopBranch-left) -- (TopBranch-right) -- cycle;
			\draw (P-v4) -- (BotBranch-left) -- (BotBranch-right) -- cycle;
			\draw (T-top) -- (Out-top);
			\draw (T-bot) -- (Out-bot);
			
			\node at (-\hTri/2+0.2, 0) { $T_{1}$};
			\node at (P-center) { $P_{1}$};
			
			% LOCAL POLYGON LABELS (Tucked into the pockets near the vertices)
			\node at ($(T-top) + (155:\dLabel)$) {$k_3$};
			\node[xshift=0.4cm] at ($(T-top) + (80:\dLabel)$) {$k_4$};
			\node at ($(T-bot) + (205:\dLabel)$) {$k_4$};
			\node [xshift=0.4cm] at ($(T-bot) + (290:\dLabel)$) {$k_3$};
			\node[xshift=-1.6cm] at ($(T-bot) + (50:\dLabel)$) {$v_{1}$};
		\end{scope}
		
		% =========================================================
		% 2. MIDDLE UNIT (i)
		% =========================================================
		\begin{scope}[shift={(0, 0)}]
			\coordinate (T-top) at (0, \s/2);
			\coordinate (T-bot) at (0, -\s/2);
			\coordinate (T-left) at (-\hTri, 0);
			
			\coordinate (P-center) at (\aPent, 0);
			\coordinate (P-v1) at ($(P-center) + (144:\rPent)$); 
			\coordinate (P-v2) at ($(P-center) + (72:\rPent)$); 
			\coordinate (P-v3) at ($(P-center) + (0:\rPent)$); 
			\coordinate (P-v4) at ($(P-center) + (288:\rPent)$); 
			\coordinate (P-v5) at ($(P-center) + (216:\rPent)$); 
			
			\coordinate (TopBranch-left) at ($(P-v2) + (102:\sBranch)$);
			\coordinate (TopBranch-right) at ($(P-v2) + (42:\sBranch)$);
			\coordinate (BotBranch-left) at ($(P-v4) + (258:\sBranch)$);
			\coordinate (BotBranch-right) at ($(P-v4) + (318:\sBranch)$);
			
			\coordinate (Out-top) at ($(T-top) + (114:\sOut)$); 
			\coordinate (Out-bot) at ($(T-bot) + (246:\sOut)$); 
			
			% Draw Shapes
			\draw (T-top) -- (T-bot) -- (T-left) -- cycle;
			\draw (P-v1) -- (P-v2) -- (P-v3) -- (P-v4) -- (P-v5) -- cycle;
			\draw (P-v2) -- (TopBranch-left) -- (TopBranch-right) -- cycle;
			\draw (P-v4) -- (BotBranch-left) -- (BotBranch-right) -- cycle;
			\draw (T-top) -- (Out-top);
			\draw (T-bot) -- (Out-bot);
			
			\node at (-\hTri/2, 0) { $T_{2}$};
			\node at (P-center) { $P_{2}$};
			
			% LOCAL POLYGON LABELS (Alternating logic)
			\node[xshift=-0.4cm] at ($(T-top) + (155:\dLabel)$) {$k_3$};
			\node[xshift=0.2cm] at ($(T-top) + (70:\dLabel)$) {$k_4$};
			\node[xshift=-0.4cm]at ($(T-bot) + (205:\dLabel)$) {$k_4$};
			\node[xshift=0.2cm] at ($(T-bot) + (305:\dLabel)$) {$k_3$};
			\node[xshift=-1.5cm] at ($(T-bot) + (50:\dLabel)$) {$v_{2}$};
		\end{scope}
		
		% =========================================================
		% 3. RIGHT UNIT (i+1)
		% =========================================================
		\begin{scope}[shift={(\xStep, 0)}]
			\coordinate (T-top) at (0, \s/2);
			\coordinate (T-bot) at (0, -\s/2);
			\coordinate (T-left) at (-\hTri, 0);
			
			\coordinate (P-center) at (\aPent, 0);
			\coordinate (P-v1) at ($(P-center) + (144:\rPent)$); 
			\coordinate (P-v2) at ($(P-center) + (72:\rPent)$); 
			\coordinate (P-v3) at ($(P-center) + (0:\rPent)$); 
			\coordinate (P-v4) at ($(P-center) + (288:\rPent)$); 
			\coordinate (P-v5) at ($(P-center) + (216:\rPent)$); 
			
			\coordinate (TopBranch-left) at ($(P-v2) + (102:\sBranch)$);
			\coordinate (TopBranch-right) at ($(P-v2) + (42:\sBranch)$);
			\coordinate (BotBranch-left) at ($(P-v4) + (258:\sBranch)$);
			\coordinate (BotBranch-right) at ($(P-v4) + (318:\sBranch)$);
			
			\coordinate (Out-top) at ($(T-top) + (114:\sOut)$); 
			\coordinate (Out-bot) at ($(T-bot) + (246:\sOut)$); 
			
			% Draw Shapes
			\draw (T-top) -- (T-bot) -- (T-left) -- cycle;
			\draw (P-v1) -- (P-v2) -- (P-v3) -- (P-v4) -- (P-v5) -- cycle;
			\draw (P-v2) -- (TopBranch-left) -- (TopBranch-right) -- cycle;
			\draw (P-v4) -- (BotBranch-left) -- (BotBranch-right) -- cycle;
			\draw (T-top) -- (Out-top);
			\draw (T-bot) -- (Out-bot);
			
			\node at (-\hTri/2+0.2, 0) { $T_{3}$};
			\node at (P-center) {$P_{3}$};
			
			% LOCAL POLYGON LABELS (Alternating logic)
			\node[xshift=-0.4cm] at ($(T-top) + (155:\dLabel)$) {$k_3$};
			\node[xshift=0.3cm] at ($(T-top) + (70:\dLabel)$) {$k_4$};
			\node[xshift=-0.4cm] at ($(T-bot) + (205:\dLabel)$) {$k_4$};
			\node[xshift=0.3cm] at ($(T-bot) + (290:\dLabel)$) {$k_3$};
			\node[xshift=-1.5cm] at ($(T-bot) + (45:\dLabel)$) {$v_3$};
		\end{scope}
		
		% --- Ellipses for the infinite sequence ---
		\node[scale=1.5] at (-\xStep*1.5, 0) {$\dots$};
		\node[scale=1.5] at (\xStep*1.8, 0) {$\dots$};
		
	\end{tikzpicture}
\captionof{figure}{Sequence of alternating triangle and type-2 pentagons along a hypercyle}
\label{chain}
\end{figure}

Recall that a path defined by segments of constant length and constant turning angles traces a curve of constant geodesic curvature. The piecewise geodesic path connecting the single common vertices of triangles and pentagons ( $v_1, v_2, v_3, \cdots$ ) in $S$ circumscribes an infinite hypercycle segment in $\mathbb{H}^2$. 
\\
Let us consider two pentagons $P_1$ and $P_2$ that are adjacent to a common triangle by a vertex and opposite edge, respectively, as illustrated in Figure \ref{strip}. Let $(C_1, C^{'}_1)$ and $(C_2, C_2^{'})$ be pairs of chains of type $S$ (described above) starting from the two vertices of the pentagons $P_1$ and $P_2$ respectively, as shown in the figure. Thus, we have two bi-infinite chains of polygons $[C_1, P_1, C^{'}_1]$ and $[C_2, P_2, C^{'}_2]$ 
\\
We can construct a strip of a tiling of type $[3, 5, k_3, k_4]$ between two chains of polygons $[C_1, P_1, C^{'}_1]$ and $[C_2, P_2, C^{'}_2]$ using the construction method described in Theorem \ref{431prop} as any tiling starting from the vertex $v_2$ is forced to have these chains along the boundary of the strip. 
\\
Recall that two hypercycles are called ultraparallel if their supporting geodesic axes are ultraparallel. Any two hypercycles with the same geodesic curvature are equidistant from their respective axes, and are therefore congruent via a translation of hyperbolic type \cite{B83, JA05}. Consequently, the chains $[C_1, P_1, C^{'}_1]$ and $[C_2, P_2, C^{'}_2]$ are congruent via a translation of hyperbolic type. Thus, an entire tiling of type $[3, 5, k_3, k_4]$ can be constructed by joining the copies of such strips through a translation of hyperbolic type. Hence we obtain a tiling with infinite cyclic symmetry.

\begin{figure}[H]
 \centering

	\begin{tikzpicture}[
		scale=0.6, % Scales the overall size of the diagram
		transform shape, % Scales the text font proportionally
		 thick,
		font=\sffamily\large,
		]
		
		% ==========================================
		% MACRO: Add Pentagon, Triangle, and Continuation Dots
		% ==========================================
		\def\addshapes#1#2{
			% 1. Calculate and define pentagon vertices
			\path let \p1=($(#2)-(#1)$), \n1={veclen(\x1,\y1)}, \n2={atan2(\y1,\x1)} in
			coordinate (p1) at (#1)
			coordinate (p2) at (#2)
			coordinate (p3) at ($ (p2) + (\n2-72:\n1) $)
			coordinate (p4) at ($ (p3) + (\n2-144:\n1) $) % p4 is the outermost tip
			coordinate (p5) at ($ (p4) + (\n2-216:\n1) $)
			
			% 2. Calculate triangle vertices attached to the outermost tip (p4)
			coordinate (t2) at ($ (p4) + (\n2-120:\n1) $)
			coordinate (t3) at ($ (p4) + (\n2-60:\n1) $)
			
			% 3. Calculate coordinates for the continuation dots
			coordinate (d1) at ($ (p4) + (\n2-90:1.3*\n1) $)
			coordinate (d2) at ($ (p4) + (\n2-90:1.6*\n1) $)
			coordinate (d3) at ($ (p4) + (\n2-90:1.9*\n1) $);
			
			% Draw the new pentagon
			\draw[line join=round] (p1) -- (p2) -- (p3) -- (p4) -- (p5) -- cycle;
			
			% Draw the new vertex-attached triangle
			\draw[line join=round] (p4) -- (t2) -- (t3) -- cycle;
			
			% Draw the continuation dots
			\fill (d1) circle (1.5pt);
			\fill (d2) circle (1.5pt);
			\fill (d3) circle (1.5pt);
		}
		
		% ==========================================
		% DEFINING BASE COORDINATES (Directly Rotated)
		% ==========================================
		\coordinate (vi) at (0,0);
		
		% Bottom Pentagon P_{i-1} (Previously Left)
		\coordinate (P1_1) at (-1.2, -1.5);
		\coordinate (P1_2) at (-0.7, -3);
		\coordinate (P1_3) at (0.7, -3);
		\coordinate (P1_4) at (1.2, -1.5);
		
		% Central Triangle T_i
		\coordinate (T_top) at (-0.7, 1.2);
		\coordinate (T_bot) at (0.7, 1.2);
		
		% Top Pentagon P_i (Previously Right)
		\coordinate (P2_1) at (-1.2, 2.5);
		\coordinate (P2_2) at (0, 3.5);
		\coordinate (P2_3) at (1.2, 2.5);
		
		% Small Outer Triangles
		\coordinate (TL1) at (-2.2, -1.8);
		\coordinate (TL2) at (-1.9, -0.6);
		
		\coordinate (BL1) at (2.2, -1.8);
		\coordinate (BL2) at (1.9, -0.6);
		
		\coordinate (TR1) at (-2.2, 2.2);
		\coordinate (TR2) at (-1.9, 3.4);
		
		\coordinate (BR1) at (2.2, 2.2);
		\coordinate (BR2) at (1.9, 3.4);

		% ==========================================
		% DRAWING THE NEW SHAPES & DOTS
		% ==========================================
		\addshapes{TL2}{TL1}
		\addshapes{BL1}{BL2}
		\addshapes{TR2}{TR1}
		\addshapes{BR1}{BR2}

		% ==========================================
		% DRAWING THE ORIGINAL BASE STRUCTURE
		% ==========================================
		% Main Pentagons
		\draw (vi) -- (P1_1) -- (P1_2) -- (P1_3) -- (P1_4) -- cycle;
		\draw (T_top) -- (P2_1) -- (P2_2) -- (P2_3) -- (T_bot) -- cycle;
		
		% Central Triangle T_i
		\draw (vi) -- (T_top) -- (T_bot) -- cycle;
		
		% Small Outer Triangles
		\draw (P1_1) -- (TL1) -- (TL2) -- cycle;
		\draw (P1_4) -- (BL1) -- (BL2) -- cycle;
		\draw (P2_1) -- (TR1) -- (TR2) -- cycle;
		\draw (P2_3) -- (BR1) -- (BR2) -- cycle;
		
		% ==========================================
		% INTERSECTING TICK LINES
		% ==========================================
		\draw (T_top) -- (-1.8, 0.5);
		\draw (T_bot) -- (1.8, 0.5);

		% ==========================================
		% LABELS (Adjusted for native orientation)
		% ==========================================
		% Sequence Labels (Placed at the bottom start of the sequences)
		\node at (-7.2, 0.0) {$C_1$};
		\node at (-7.2, 4) {$C_2$};
		\node at (7.2, 0.0) {$C_1^{'}$};
		\node at (7.2, 4) {$C_2^{'}$};
		
		\node at (0, -1.6) {$P_{1}$};
		\node at (0, 2.2) {$P_{2}$};
		\node at (0, 0.6) {$T_2$};
		\node at (0.4, 0) {$v_2$};
		
		\node at (-1.3, -2.0) {$k_4$};
		\node at (1.3, -2.0) {$k_3$};
		\node at (-1.0, 0.2) {$k_3$};
		\node at (1.0, 0.2) {$k_4$};
		\node at (-1.3, 1.5) {$k_4$};
		\node at (1.3, 1.5) {$k_3$};
		
	\end{tikzpicture}
	
\captionof{figure}{Strip of tiling between two hypercycles}
\label{strip}
\end{figure}

\end{proof}
\textbf{Symmetry of finite orders:} For $k_3$ and $k_4$ prime, none of the tiling of type $[3,5,k_3,k_4]$ admits a rotational symmetry. This is because neither the fans and neighborhoods of any of the polygons has a rotational symmetry for such type. The same can be said about reflectional symmetry. However, for $k_3$ or $k_4$ multiple of $4$, there exists a tiling with rotational symmetry. Such tiling can be constructed starting with a neighbourhood of a $k_3$-gon with rotational symmetry and forced chains of polygons as shown Figure \ref{chain}.

\subsection{Tiling Space} \label{tilingspace}
Let $\Omega([3, 5, k_3, k_4])$ be the space of all tilings of $\mathbb{H}^2$ of type $[3, 5, k_3, k_4]$ endowed with natural local topology \cite{S08}, see \S \ref{appen1} for more details. The tilings of type $[3, 5, k_3, k_4]$ exhibit the following kind of a strong extendability property that shed lights into the topology of the space $\Omega([3, 5, k_3, k_4])$.			
 Let \(\mathcal{B}\) be the set of all admissible (partial or full) neighborhoods of the faces in $[3, 5, k_3, k_4]$; that is, the neighborhoods satisfying the Observations \ref{A1}--\ref{A2}.

\begin{lemma} \label{extn}
Any locally consistent $n$-layer of tiling of type $[3, 5, k_3, k_4]$ for $k_3 \geq 12, k_4 \geq 14$ composed of elements from $\mathcal{B}$ can be extended to a full tiling of $\mathbb{H}^2$.
\end{lemma}

\begin{proof} We proceed by induction, mirroring the existence proof of tiling of type $[3, 5, k_3, k_4]$ in Theorem \ref{431prop} above. To establish extendability of $n$-layer to $n+1$-layer, it suffices to show that any valid placement of triangles around the $5$-gons in $C_k$ results in admissible partial neighbourhoods around the big-gons ($k_3$- and $k_4$-gon). Specifically, we must ensure that $\nabla$-triangles are never attached to three consecutive vertices of a big-gon.\\
In Figure \ref{around5}, we illustrate the neighborhoods of a $k_3$-gon induced by all possible admissible configurations around the $5$-gons. This exhaustive set includes all the possibilities previously illustrated in Fig. \ref{fig:357part1}. It is straightforward to verify that each case yields an admissible partial neighborhood around the $k_3$-gon. The same argument works for $k_4$-gon. 
\\
Finally, one can verify that any admissible neighborhood formation around the $k_3$-gons leads to an admissible partial neighborhood around adjacent the $k_4$, if any, and vice-versa. Because no local choice forces a violation of A1--A2 at the $n+1$ layer, the induction holds, thereby proving the global existence of the tiling. 

\begin{figure}[H]
 \centering
\subfigure[$k_3(6)$]{
\tikzstyle{ver}=[]
\tikzstyle{vert}=[circle, draw, fill=black!100, inner sep=0pt, minimum width=4pt]
\tikzstyle{vertex}=[circle, draw, fill=black!00, inner sep=0pt, minimum width=4pt]
\tikzstyle{edge} = [draw,thick,-]
\centering
\begin{tikzpicture}[scale=0.18]
\draw[edge, thick](30,1)--(48.5,1);
\draw[edge, thick](30,-5)--(48.5,-5);
\draw[edge, thick](34,1)--(35.5,1)--(37,-2.5)--(38,1)--(42,1);
\draw[edge, thick] (32,1)--(34,-2.5)--(35.5, -5)--(37,-2.5)--(34, -2.5);
\draw[edge, thick] (45,-2.5)--(43.5,-5)-- (42,-2.5)--(45,-2.5)--(46,1);
\draw[edge, thick] (43.5, 1)--(42,-2.5)--(41,1);
\node[ver] () at (38,-6){5};
\node[ver] () at (39.5,-1){\scriptsize ${k_3}$};
\node[ver] () at (28,-5){\scriptsize ${\partial X_k}$};
\end{tikzpicture} }
\centering
\subfigure[$k_3(6)$]{
\tikzstyle{ver}=[]
\tikzstyle{vert}=[circle, draw, fill=black!100, inner sep=0pt, minimum width=4pt]
\tikzstyle{vertex}=[circle, draw, fill=black!00, inner sep=0pt, minimum width=4pt]
\tikzstyle{edge} = [draw,thick,-]
\centering
\begin{tikzpicture}[scale=0.18]
\draw[edge, thick](30,1)--(48.5,1);
\draw[edge, thick](30,-5)--(48.5,-5);
\draw[edge, thick](34,1)--(35.5,1)--(37,-2.5)--(38,1)--(42,1);

\draw[edge, thick] (32,1)--(34,-2.5)--(35.5, -5)--(37,-2.5)--(34, -2.5);
\draw[edge, thick] (41.5,-5)--(43, -2.5)-- (45,-5);
\draw[edge, thick] (47, 1)--(43,-2.5)--(43,1);
\draw[edge, thick](41.5,-5)--(41.5,-7);
\node[ver] () at (38,-6){5};
\node[ver] () at (39.5,-2){\scriptsize ${k_3}$};
\end{tikzpicture} }
\centering
\subfigure[$k_3(5)$]{
\tikzstyle{ver}=[]
\tikzstyle{vert}=[circle, draw, fill=black!100, inner sep=0pt, minimum width=4pt]
\tikzstyle{vertex}=[circle, draw, fill=black!00, inner sep=0pt, minimum width=4pt]
\tikzstyle{edge} = [draw,thick,-]
\centering
\begin{tikzpicture}[scale=0.18]
\draw[edge, thick](30,1)--(48.5,1);
\draw[edge, thick](30,-5)--(48.5,-5);
\draw[edge, thick](34,1)--(35.5,1)--(37,-2.5)--(38,1)--(42,1);

\draw[edge, thick] (32,1)--(34,-2.5)--(35.5, -5)--(37,-2.5)--(34, -2.5);
\draw[edge, thick] (41.5,-5)--(43, -2.5)-- (45,-5);
\draw[edge, thick] (47, 1)--(43,-2.5)--(43,1);
\draw[edge, thick](41.5,-5)--(41,1);
\node[ver] () at (38,-6){5};
\node[ver] () at (39.5,-2){\scriptsize ${k_3}$};
\end{tikzpicture} }
\centering
\subfigure[$k_3(4)$]{
\tikzstyle{ver}=[]
\tikzstyle{vert}=[circle, draw, fill=black!100, inner sep=0pt, minimum width=4pt]
\tikzstyle{vertex}=[circle, draw, fill=black!00, inner sep=0pt, minimum width=4pt]
\tikzstyle{edge} = [draw,thick,-]
\centering
\begin{tikzpicture}[scale=0.18]
\draw[edge, thick](30,1)--(48.5,1);
\draw[edge, thick](30,-5)--(48.5,-5);

\draw[edge, thick](31,-5)--(31,-7);
\draw[edge, thick](47,-5)--(47,-7);
\draw[edge, thick](35,-2.5)--(36.5,1)--(42,1);

\draw[edge, thick] (33,1)--(35,-2.5)--(33, -5);
\draw[edge, thick] (35,-2.5)--(37, -5)--(38, 1);
\draw[edge, thick] (41.5,-5)--(43, -2.5)-- (45,-5);
\draw[edge, thick] (47, 1)--(43,-2.5)--(43,1);
\draw[edge, thick](41.5,-5)--(41,1);
\node[ver] () at (38,-6){5};
\node[ver] () at (28,-5){\scriptsize ${\partial X_k}$};
\node[ver] () at (39.5,-2){\scriptsize ${k_3}$};
\end{tikzpicture} }
\centering
\subfigure[$k_3(4)$]{
\tikzstyle{ver}=[]
\tikzstyle{vert}=[circle, draw, fill=black!100, inner sep=0pt, minimum width=4pt]
\tikzstyle{vertex}=[circle, draw, fill=black!00, inner sep=0pt, minimum width=4pt]
\tikzstyle{edge} = [draw,thick,-]
\centering
\begin{tikzpicture}[scale=0.18]
\draw[edge, thick](15,1)--(38,1);
\draw[edge, thick](15,-5)--(38,-5);
\draw[edge, thick](22,1)--(22,-2.5)--(27,-2.5)--(27,1)--(30,1);
\draw[edge, thick] (16,1)--(22,-2.5)--(24.5, -5)--(27,-2.5)--(32, 1);
\node[ver] () at (20.5,-0.5){\scriptsize $5$};
\node[ver] () at (25,-6){5};
\node[ver] () at (28.5,-0.5){\scriptsize ${ k_{4}}$};
\node[ver] () at (24.5,-1){\scriptsize ${k_3}$};
\end{tikzpicture}}
\centering
\subfigure[$k_3(3)$]{
\tikzstyle{ver}=[]
\tikzstyle{vert}=[circle, draw, fill=black!100, inner sep=0pt, minimum width=4pt]
\tikzstyle{vertex}=[circle, draw, fill=black!00, inner sep=0pt, minimum width=4pt]
\tikzstyle{edge} = [draw,thick,-]
\centering
\begin{tikzpicture}[scale=0.18]
\draw[edge, thick](30,1)--(48.5,1);
\draw[edge, thick](30,-5)--(48.5,-5);
\draw[edge, thick](34,1)--(35,1)--(37,-2.5)--(40,1)--(42,1);
\draw[edge, thick](35,1)--(36,2.5)--(37,1);
\draw[edge, thick] (32,1)--(34,-2.5)--(35.5, -5)--(37,-2.5)--(34, -2.5);
\draw[edge, thick] (41.5,-5)--(43, -2.5)-- (45,-5);
\draw[edge, thick] (47, 1)--(43,-2.5)--(43,1);
\node[ver] () at (38,-6){5};
\node[ver] () at (37.5,-0.5){\scriptsize ${k_3}$};
\node[ver] () at (34.5,-0.5){\scriptsize $5$};
\end{tikzpicture} }
\captionof{figure}{Partial eighborhoods of $k_3$-gons for type $[3, 5, k_3, k_4]$}
\label{around5}
\end{figure}

\end{proof}

\begin{corollary} \label{strongexp} The tiling Space $ \Omega=\Omega({3, 5, k_3, k_4})$ is homeomorphic to a Cantor set. 
\end{corollary}
\begin{proof}
We prove the standard characterization: a nonempty, compact, metrizable, perfect, totally disconnected space is homeomorphic to the Cantor set. 
\\
The space $\Omega$ is obviously a nonempty metric space. Since there are only a finite number of tiles meeting edge-to-edge in every tilings (ensures finite local complexity (FLC)), compactness and total disconnectedness follows using standard arguments \cite{S08}. It follows from Lemma \ref{strongexp} that no finite patch of tiling determines an entire tiling, i.e, the space has no isolated points. Thus $\Omega$ is a Cantor set.\\

\end{proof}

The computation and implications of the cohomology of the zeroth Anderson-Putnam approximant of the tiling space associated to the particular type $[3, 5, 12, 14]$ is presented in Appendix \ref{appen}.

\section{Concluding remarks}\label{conclu}
Our results suggest several new directions for further research. We discuss a few of them in this section. 
\par
\textbf{1}. We expect a plethora of weakly aperiodic protosets of regular polygons to exist, in addition to those presented here. In the context of strong aperiodicity, we propose the following conjecture:
\begin{conjecture}
 There does not exist strongly aperiodic protoset of regular polygons in $\mathbb{H}^2$. 
 \end{conjecture}
It can be shown that if a set of six or more regular polygons admit (pseudo-homogeneous) tiling then it admits a tiling with translation symmetry. This can also be shown to be true for many of the $4$- and $5$-tuples. For example, there exist a tiling of type $[3, 5, k_3, k_4]$ with translational symmetry.
\par
\textbf{2}. The domino problem for homogeneous tilings remains an open problem \cite{DG18}. A recent article (to appear soon) by the present author establishes the unboundedness of the Heesch problem in that setting.

\section*{Acknowledgement} The author would like to thank Marek Čtrnáct for some valuable suggestions. The author also thanks Zeno, the creator of HyperRouge, for technical assistance in rendering Fig.~\ref{351012}. The research was partially supported by Science and Engineering Research Board (SERB) grant (CRG/2019/007028) for the year 2020-2021, and Seed Grant DoRDC/730 of TIET for the year 2024-25. 
\bibliographystyle{amsplain}
\bibliography{bibtile}

\appendix \label{appen}
\section{Tiling Space for vertex type $[3, 5, 12, 14]$} \label{appen1}

In this section we study a few topological properties of the tiling space of the aperiodic tile set associated to the tuple $[3, 5, 12, 14]$) \S \ref{aperiodic}. 
\\
\begin{definition}
	Let $\Omega_{\mathcal P}$ denote the set of all full tilings $T$ of $\mathbb H^2$ by $\mathcal P$ for which the distinguished point $o$ lies in some tile of $T$. \end{definition}
	
Note that we keep the location $o$ fixed and do not quotient by global isometries. This restriction defines a canonical transversal of the continuous tiling space.

\paragraph{Metric:} For two tilings $T,S\in\Omega_{\mathcal P}$, define
	\[
	R(T,S):=\sup\{\, R\ge 0:\ \text{the radius-$R$ centered patches of $T$ and $S$ are identical}\,, \]
	\[
	\qquad d(T,S):=2^{-R(T,S)}.
	\} \]
	By FLC the supremum is a nonnegative integer (possibly $\infty$) and $d$ is a metric inducing the cylinder topology \cite{ B13, S08}. For a radius-$R$ centered patch $P$ define the \emph{cylinder set}
		\[
		C(P) \;=\; \{\, T\in\Omega_{\mathcal P}:\ \text{the radius-$R$ centered patch of }T\text{ equals }P \,\}.
		\]
		These cylinder sets (for varying $R$ and $P\in\mathcal P_R$) form a basis for the \emph{local topology} on $\Omega_{\mathcal P}$.

\subsection{Transition Matrix.} We first establish the combinatorial constraints on the neighborhood of $12$-gon. In any tiling of type $[3, 5, 12, 14]$, the $5$- and $14$-gons incident (share edge) to a common $12$-gon must appear in alternating sequence (potentially separated by a triangle). Consequently, their count within the neighborhood must be equal. This implies that an even number of edges of the $12$-gon are shared with $5$- and $14$-gons, which further necessitates that the number of $\Delta$-triangles in the neighborhood is even. On the other hand, counting the common vertices with the triangles in the neighborhood, the total count of $\Delta$-triangles and $\nabla$-triangles must be $12$. Consequently, the number of $\Delta$- and $\nabla$- triangles in the neighborhood both must also be even. A similar argument works for the $14$-gon. 
\\
These parity observations limits the number of admissible neighborhoods around the $12$-gon and $14$-gon to the following configurations:
\begin{figure}[H]
\centering
\resizebox{.8\linewidth}{!}{
\subfigure[$D_1$]{\begin{tikzpicture}[line width=0.75pt, scale=1.33]
\tikzstyle{vert}=[circle, draw, fill=black!100, inner sep=0pt, minimum width=1pt]

% --- draw the 12-gon
\node[regular polygon, regular polygon sides=12, minimum size=3cm, draw] (a) {};
\coordinate (C) at (a.center);
\node at (C) {12};

% --- define vertices of 12-gon
\foreach \i in {1,...,12}{
 \path (a.corner \i) coordinate (P\i);
}

% --- triangles on some edges
\draw (P1) -- ($ (P1)!1!{-60}:(P2) $) -- (P2);
\draw (P3) -- ($ (P3)!1!{-60}:(P4) $) -- (P4);
\draw (P5) -- ($ (P5)!1!{-60}:(P6) $) -- (P6);
\draw (P7) -- ($ (P7)!1!{-60}:(P8) $) -- (P8);
\draw (P9) -- ($ (P9)!1!{-60}:(P10) $) -- (P10);
\draw (P11) -- ($ (P11)!1!{-60}:(P12) $) -- (P12);

 % ---- Outer coordinates (Q1..Q12) on a circle ----
 \pgfmathsetmacro{\polyR}{3/2} % radius of circumscribed circle of 12-gon
 \pgfmathsetmacro{\radius}{\polyR*1.07} % place Qi a bit further out

 % explicitly define outer coordinates
 \path (C) ++(73:\radius) coordinate (Q1);
 \path (C) ++(105:\radius) coordinate (Q2);
 \path (C) ++(135:\radius) coordinate (Q3);
 \path (C) ++(170:\radius) coordinate (Q4);
 \path (C) ++(195:\radius) coordinate (Q5);
 \path (C) ++(225:\radius) coordinate (Q6);
 \path (C) ++(250:\radius) coordinate (Q7);
 \path (C) ++(285:\radius) coordinate (Q8);
 \path (C) ++(315:\radius) coordinate (Q9);
 \path (C) ++(345:\radius) coordinate (Q10);
 \path (C) ++(15:\radius) coordinate (Q11);
 \path (C) ++(45:\radius) coordinate (Q12);

 % ---- Example connections (edit to your needs) ----
 \draw (P1) -- (Q1);
 \draw (P2) -- (Q2);
 \draw (P3) -- (Q3);
 \draw (P4) -- (Q4);
 \draw (P5) -- (Q5);
 \draw (P6) -- (Q6);
 \draw (P7) -- (Q7);
 \draw (P8) -- (Q8);
 \draw (P9) -- (Q9);
 \draw (P10) -- (Q10);
 \draw (P11) -- (Q11);
 \draw (P12) -- (Q12);

% --- labels a bit outside
\pgfmathsetmacro{\polyR}{3/2}
\pgfmathsetmacro{\radius}{\polyR*0.85}

\node[scale=0.75] at ($(C)+(0:\radius)$) {5};
\node[scale=0.75] at ($(C)+(60:\radius)$) {14};

\node[scale=0.75] at ($(C)+(120:\radius)$) {5};
\node[scale=0.75] at ($(C)+(180:\radius)$) {14};
\node[scale=0.75] at ($(C)+(240:\radius)$) {5};
\node[scale=0.75] at ($(C)+(300:\radius)$) {14};

\end{tikzpicture}}
\subfigure[$D_2$]{\begin{tikzpicture}[line width=0.75pt, scale=1.3]

 % draw the 12-gon
 \node[regular polygon, regular polygon sides=12, minimum size=3cm, draw] (a) {};
 \coordinate (C) at (a.center);
 \node at (C) {12};

 % define vertices of 12-gon explicitly
 \path (a.corner 1) coordinate (P1);
 \path (a.corner 2) coordinate (P2);
 \path (a.corner 3) coordinate (P3);
 \path (a.corner 4) coordinate (P4);
 \path (a.corner 5) coordinate (P5);
 \path (a.corner 6) coordinate (P6);
 \path (a.corner 7) coordinate (P7);
 \path (a.corner 8) coordinate (P8);
 \path (a.corner 9) coordinate (P9);
 \path (a.corner 10) coordinate (P10);
 \path (a.corner 11) coordinate (P11);
 \path (a.corner 12) coordinate (P12);

 % triangles on alternate edges, skipping P3-P4 and P9-P10
 \draw (P1) -- ($ (P1) !1! {-60}:(P2) $) -- (P2);
 % skip P3-P4
 \draw (P5) -- ($ (P5) !1! {-60}:(P6) $) -- (P6);

 % skip P9-P10
 \draw (P10) -- ($ (P10)!1! {-60}:(P11)$) -- (P11);
 \draw (P8) -- ($ (P8)!1!{-60}:(P9) $) -- (P9);

 % ---- Outer coordinates (Q1..Q12) on a circle ----
 \pgfmathsetmacro{\polyR}{3/2} % radius of circumscribed circle of 12-gon
 \pgfmathsetmacro{\radius}{\polyR*1.08} % place Qi a bit further out

 % explicitly define outer coordinates
 \path (C) ++(73:\radius) coordinate (Q1);
 
 \path (C) ++(105:\radius) coordinate (Q2);
 \path (C) ++(125:\radius) coordinate (Q3);
 \path (C) ++(145:\radius) coordinate (Q3a);
 \path (C) ++(155:\radius) coordinate (Q4);
 \path (C) ++(175:\radius) coordinate (Q4a);
 \path (C) ++(195:\radius) coordinate (Q5);
 \path (C) ++(225:\radius) coordinate (Q6);
 \path (C) ++(242:\radius) coordinate (Q7);
 \path (C) ++(264:\radius) coordinate (Q7a);
 \path (C) ++(285:\radius) coordinate (Q8);
 \path (C) ++(315:\radius) coordinate (Q9);

 \path (C) ++(345:\radius) coordinate (Q10);

 \path (C) ++(15:\radius) coordinate (Q11);
 \path (C) ++(35:\radius) coordinate (Q12);
 \path (C) ++(55:\radius) coordinate (Q12a);

 % ---- Example connections (edit to your needs) ----
 \draw (P1) -- (Q1);
 \draw (P2) -- (Q2);
 \draw (P3) -- (Q3);
 \draw (P3) -- (Q3a);
 \draw (Q3) -- (Q3a);
 \draw (P4) -- (Q4);
 \draw (P4) -- (Q4a); 
 \draw (Q4) -- (Q4a); 
 \draw (P5) -- (Q5);
 \draw (P6) -- (Q6);
 \draw (P6) -- (Q6);
 \draw (P7) -- (Q7);
 \draw (P7) -- (Q7a);
 \draw (Q7) -- (Q7a);
 \draw (P9) -- (Q9);
 \draw (P8) -- (Q8);

 \draw (P10) -- (Q10);
 
 \draw (P11) -- (Q11);
 \draw (P12) -- (Q12);
 \draw (P12) -- (Q12a);
 \draw (Q12) -- (Q12a);

\pgfmathsetmacro{\polyR}{3/2}
\pgfmathsetmacro{\radius}{\polyR*0.85}

\node[scale=0.75] at ($(C)+(30:\radius)$) {5};
\node[scale=0.75] at ($(C)+(60:\radius)$) {14};
\node[scale=0.75] at ($(C)+(150:\radius)$) {14};
\node[scale=0.75] at ($(C)+(120:\radius)$) {5};
\node[scale=0.75] at ($(C)+(180:\radius)$) {5};
\node[scale=0.75] at ($(C)+(240:\radius)$) {14};
\node[scale=0.75] at ($(C)+(270:\radius)$) {5};
\node[scale=0.75] at ($(C)+(330:\radius)$) {14};

\end{tikzpicture}}
\subfigure[$D_3$]{\begin{tikzpicture}[line width=0.75pt, scale=1.3]

 % draw the 12-gon
 \node[regular polygon, regular polygon sides=12, minimum size=3cm, draw] (a) {};
 \coordinate (C) at (a.center);
 \node at (C) {12};

 % define vertices of 12-gon explicitly
 \path (a.corner 1) coordinate (P1);
 \path (a.corner 2) coordinate (P2);
 \path (a.corner 3) coordinate (P3);
 \path (a.corner 4) coordinate (P4);
 \path (a.corner 5) coordinate (P5);
 \path (a.corner 6) coordinate (P6);
 \path (a.corner 7) coordinate (P7);
 \path (a.corner 8) coordinate (P8);
 \path (a.corner 9) coordinate (P9);
 \path (a.corner 10) coordinate (P10);
 \path (a.corner 11) coordinate (P11);
 \path (a.corner 12) coordinate (P12);

 % triangles on alternate edges, skipping P3-P4 and P9-P10
 \draw (P1) -- ($ (P1) !1! {-60}:(P2) $) -- (P2);
 % skip P3-P4
 \draw (P5) -- ($ (P5) !1! {-60}:(P6) $) -- (P6);
 \draw (P7) -- ($ (P7) !1! {-60}:(P8) $) -- (P8);
 % skip P9-P10
 \draw (P11) -- ($ (P11)!1! {-60}:(P12)$) -- (P12);

 % ---- Outer coordinates (Q1..Q12) on a circle ----
 \pgfmathsetmacro{\polyR}{3/2} % radius of circumscribed circle of 12-gon
 \pgfmathsetmacro{\radius}{\polyR*1.08} % place Qi a bit further out

 % explicitly define outer coordinates
 \path (C) ++(73:\radius) coordinate (Q1);
 
 \path (C) ++(105:\radius) coordinate (Q2);
 \path (C) ++(125:\radius) coordinate (Q3);
 \path (C) ++(145:\radius) coordinate (Q3a);
 \path (C) ++(155:\radius) coordinate (Q4);
 \path (C) ++(175:\radius) coordinate (Q4a);
 \path (C) ++(195:\radius) coordinate (Q5);
 \path (C) ++(225:\radius) coordinate (Q6);
 \path (C) ++(250:\radius) coordinate (Q7);
 \path (C) ++(285:\radius) coordinate (Q8);
 \path (C) ++(305:\radius) coordinate (Q9);
 \path (C) ++(325:\radius) coordinate (Q9a);
 \path (C) ++(335:\radius) coordinate (Q10);
 \path (C) ++(355:\radius) coordinate (Q10a);
 \path (C) ++(15:\radius) coordinate (Q11);
 \path (C) ++(45:\radius) coordinate (Q12);

 % ---- Example connections (edit to your needs) ----
 \draw (P1) -- (Q1);
 \draw (P2) -- (Q2);
 \draw (P3) -- (Q3);
 \draw (P3) -- (Q3a);
 \draw (Q3) -- (Q3a);
 \draw (P4) -- (Q4);
 \draw (P4) -- (Q4a); 
 \draw (Q4) -- (Q4a); 
 \draw (P5) -- (Q5);
 \draw (P6) -- (Q6);
 \draw (P7) -- (Q7);
 \draw (P10) -- (Q10);
 \draw (P11) -- (Q11);
 \draw (P12) -- (Q12);
 \draw (P9) -- (Q9);
 \draw (P9) -- (Q9a);
 \draw (Q9) -- (Q9a);
 \draw (P10) -- (Q10a);
 \draw (Q10) -- (Q10a);
 \draw (P8) -- (Q8);
 
\pgfmathsetmacro{\polyR}{3/2}
\pgfmathsetmacro{\radius}{\polyR*0.85}

\node[scale=0.75] at ($(C)+(0:\radius)$) {5};
\node[scale=0.75] at ($(C)+(60:\radius)$) {14};
\node[scale=0.75] at ($(C)+(150:\radius)$) {14};
\node[scale=0.75] at ($(C)+(120:\radius)$) {5};
\node[scale=0.75] at ($(C)+(180:\radius)$) {5};
\node[scale=0.75] at ($(C)+(240:\radius)$) {14};
\node[scale=0.75] at ($(C)+(300:\radius)$) {5};
\node[scale=0.75] at ($(C)+(330:\radius)$) {14};

\end{tikzpicture}}

\subfigure[$D_4$]{\begin{tikzpicture}[line width=0.75pt, scale=1.3]

 % --- 12-gon (minimum size = 3cm) ---
 \node[regular polygon, regular polygon sides=12, minimum size=3cm, draw] (a) {};
 \coordinate (C) at (a.center);
 \node at (C) {12};

 % --- explicit polygon vertices ---
 \path (a.corner 1) coordinate (P1);
 \path (a.corner 2) coordinate (P2);
 \path (a.corner 3) coordinate (P3);
 \path (a.corner 4) coordinate (P4);
 \path (a.corner 5) coordinate (P5);
 \path (a.corner 6) coordinate (P6);
 \path (a.corner 7) coordinate (P7);
 \path (a.corner 8) coordinate (P8);
 \path (a.corner 9) coordinate (P9);
 \path (a.corner 10) coordinate (P10);
 \path (a.corner 11) coordinate (P11);
 \path (a.corner 12) coordinate (P12);

 % --- triangles on alternate edges (as before) ---
 \draw (P1) -- ($ (P1) !1! {-60}:(P2) $) -- (P2); % edge P1-P2
 \draw (P4) -- ($ (P4) !1! {-60}:(P5) $) -- (P5); % edge P4-P5
 \draw (P7) -- ($ (P7) !1! {-60}:(P8) $) -- (P8); % edge P7-P8
 \draw (P10) -- ($ (P10) !1! {-60}:(P11)$) -- (P11); % edge P10-P11

 % --- place Qi on the same radial line as Pi (1.3 times out) ---
 % poly circumradius = minimum size / 2 = 3/2 cm
 \pgfmathsetmacro{\polyR}{3/2}
 \pgfmathsetmacro{\factor}{1.4} % how far out Qi is relative to Pi
 % scalar for base half-length = 1/sqrt(3)
 \pgfmathsetmacro{\sf}{1/sqrt(3)}

 % Qi for every Pi (we'll use only the Qi for the free vertices P3,P6,P9,P12)
 \path ($(C)!\factor!(P1)$) coordinate (Q1);
 \path ($(C)!\factor!(P2)$) coordinate (Q2);
 \path ($(C)!\factor!(P3)$) coordinate (Q3);
 \path ($(C)!\factor!(P4)$) coordinate (Q4);
 \path ($(C)!\factor!(P5)$) coordinate (Q5);
 \path ($(C)!\factor!(P6)$) coordinate (Q6);
 \path ($(C)!\factor!(P7)$) coordinate (Q7);
 \path ($(C)!\factor!(P8)$) coordinate (Q8);
 \path ($(C)!\factor!(P9)$) coordinate (Q9);
 \path ($(C)!\factor!(P10)$) coordinate (Q10);
 \path ($(C)!\factor!(P11)$) coordinate (Q11);
 \path ($(C)!\factor!(P12)$) coordinate (Q12);

 % --- For each free vertex Pi (i=3,6,9,12) build an equilateral triangle
 % apex at Pi, base midpoint at Qi, base endpoints = Qi +- (1/sqrt(3)) * perp(Qi->Pi)
 % P3:
 \path ($(Q3)!\sf!90:(P3)$) coordinate (B3a);
 \path ($(Q3)!\sf!-90:(P3)$) coordinate (B3b);
 \draw (P3) -- (B3a) -- (B3b) -- cycle;

 % P6:
 \path ($(Q6)!\sf!90:(P6)$) coordinate (B6a);
 \path ($(Q6)!\sf!-90:(P6)$) coordinate (B6b);
 \draw (P6) -- (B6a) -- (B6b) -- cycle;

 % P9:
 \path ($(Q9)!\sf!90:(P9)$) coordinate (B9a);
 \path ($(Q9)!\sf!-90:(P9)$) coordinate (B9b);
 \draw (P9) -- (B9a) -- (B9b) -- cycle;

 % P12:
 \path ($(Q12)!\sf!90:(P12)$) coordinate (B12a);
 \path ($(Q12)!\sf!-90:(P12)$) coordinate (B12b);
 \draw (P12) -- (B12a) -- (B12b) -- cycle;

 \pgfmathsetmacro{\polyR}{3/2} % radius of circumscribed circle of 12-gon
 \pgfmathsetmacro{\radius}{\polyR*1.08} % place Qi a bit further out

 % explicitly define outer coordinates
 \path (C) ++(73:\radius) coordinate (Q1);
 \path (C) ++(105:\radius) coordinate (Q2);
 \path (C) ++(135:\radius) coordinate (Q3);
 \path (C) ++(165:\radius) coordinate (Q4);
 \path (C) ++(195:\radius) coordinate (Q5);
 \path (C) ++(225:\radius) coordinate (Q6);
 \path (C) ++(250:\radius) coordinate (Q7);
 \path (C) ++(285:\radius) coordinate (Q8);
 \path (C) ++(315:\radius) coordinate (Q9);
 \path (C) ++(345:\radius) coordinate (Q10);
 \path (C) ++(15:\radius) coordinate (Q11);
 \path (C) ++(45:\radius) coordinate (Q12);

 % ---- Example connections (edit to your needs) ----
 \draw (P1) -- (Q1);
 \draw (P2) -- (Q2);
 \draw (P5) -- (Q5);
 \draw (P4) -- (Q4);
 \draw (P7) -- (Q7);
 \draw (P8) -- (Q8);
 \draw (P10) -- (Q10);
 \draw (P11) -- (Q11);

\pgfmathsetmacro{\polyR}{3/2}
\pgfmathsetmacro{\radius}{\polyR*0.85}

\node[scale=0.75] at ($(C)+(60:\radius)$) {14};
\node[scale=0.75] at ($(C)+(150:\radius)$) {14};
\node[scale=0.75] at ($(C)+(120:\radius)$) {5};

\node[scale=0.75] at ($(C)+(240:\radius)$) {14};
\node[scale=0.75] at ($(C)+(300:\radius)$) {5};
\node[scale=0.75] at ($(C)+(330:\radius)$) {14};
\node[scale=0.75] at ($(C)+(210:\radius)$) {5};
\node[scale=0.75] at ($(C)+(30:\radius)$) {5};
\end{tikzpicture}}}
\captionof{figure}{Neighborhoods of $12$-gon}
\end{figure}

\begin{figure}[H] 
\centering
\resizebox{.8\linewidth}{!}{

\subfigure[$H_1$]{\begin{tikzpicture}[line width=0.75pt, scale=1]
\tikzstyle{ver}=[]
\tikzstyle{vert}=[circle, draw, fill=black!100, inner sep=0pt, minimum width=1pt]
\tikzstyle{vertex}=[circle, draw, fill=black!00, inner sep=0pt, minimum width=4pt]

 % draw the 14-gon
 \node[regular polygon, regular polygon sides=14, minimum size=3cm, draw] (a) {};
 \coordinate (C) at (a.center);
 \node at (C) {14};

 % define vertices of 14-gon
 \path (a.corner 1) coordinate (P1);
 \path (a.corner 2) coordinate (P2);
 \path (a.corner 3) coordinate (P3);
 \path (a.corner 4) coordinate (P4);
 \path (a.corner 5) coordinate (P5);
 \path (a.corner 6) coordinate (P6);
 \path (a.corner 7) coordinate (P7);
 \path (a.corner 8) coordinate (P8);
 \path (a.corner 9) coordinate (P9);
 \path (a.corner 10) coordinate (P10);
 \path (a.corner 11) coordinate (P11);
 \path (a.corner 12) coordinate (P12);
 \path (a.corner 13) coordinate (P13);
 \path (a.corner 14) coordinate (P14);

 % triangles on alternate edges, minus one more
 \draw (P1) -- ($ (P1)!1!{-60}:(P2) $) -- (P2);
 \draw (P3) -- ($ (P3)!1!{-60}:(P4) $) -- (P4);
 \draw (P7) -- ($ (P7)!1!{-60}:(P8) $) -- (P8);
 \draw (P9) -- ($ (P9)!1!{-60}:(P10) $) -- (P10);
 \draw (P11) -- ($ (P11)!1!{-60}:(P12) $) -- (P12);
 \draw (P13) -- ($ (P13)!1!{-60}:(P14) $) -- (P14);

 % ---- Outer nodes placed individually on a circle ----
 \def\radius{2cm} % radius of circle around 14-gon
\path (C) ++(25:\radius) coordinate (Q2);

 \path (C) ++(0:\radius) coordinate (Q1);
\path (C) ++(25:\radius) coordinate (Q2);
 \path (C) ++(50:\radius) coordinate (Q3);
 \path (C) ++(76:\radius) coordinate (Q4);
 \path (C) ++(103:\radius) coordinate (Q5);
 \path (C) ++(130:\radius) coordinate (Q6);
 \path (C) ++(155:\radius) coordinate (Q7);
 \path (C) ++(172:\radius) coordinate (Q8);
 \path (C) ++(190:\radius) coordinate (Q8a);
 \path (C) ++(215:\radius) coordinate (Q9);
 \path (C) ++(197:\radius) coordinate (Q9a);
 \path (C) ++(228:\radius) coordinate (Q10);
 \path (C) ++(257:\radius) coordinate (Q11);
 \path (C) ++(282:\radius) coordinate (Q12);
 \path (C) ++(309:\radius) coordinate (Q13);
 \path (C) ++(335:\radius) coordinate (Q14);
 \path (C) ++(360:\radius) coordinate (Q15);

 % example: connect two of them
 \draw (P13) --(Q2);
 \draw (P14) --(Q3);
 \draw (P1) --(Q4);
 \draw (P2) --(Q5);
 \draw (P3) --(Q6);
 \draw (P4) --(Q7);
 \draw (P5) --(Q8);
 \draw (P5) --(Q8a);
 \draw (Q8) --(Q8a);
 \draw (P6) --(Q9);
 \draw (P6) --(Q9a);

\draw (Q9) --(Q9a);
 \draw (P7) --(Q10);
 \draw (P8) --(Q11);
 \draw (P9) --(Q12);
 \draw (P10) --(Q13);
 \draw (P11) --(Q14);
 \draw (P12) --(Q15);
 
 \pgfmathsetmacro{\polyR}{4.5/2}
 \pgfmathsetmacro{\radius}{\polyR*0.75} % letters a bit outside

\node[scale=0.75] at ($(C)+(10:\radius)$) {12};

\node[scale=0.75] at ($(C)+(62:\radius)$) {5};
\node[scale=0.75] at ($(C)+(118:\radius)$) {12};
\node[scale=0.75] at ($(C)+(165:\radius)$) {5};
\node[scale=0.75] at ($(C)+(192:\radius)$) {12};
\node[scale=0.75] at ($(C)+(220:\radius)$) {5};

\node[scale=0.75] at ($(C)+(270:\radius)$) {12};
 
\node[scale=0.75] at ($(C)+(325:\radius)$) {5};

\end{tikzpicture}}

\label{fig:nbd14} 
\subfigure[$H_2$]{

\begin{tikzpicture}[line width=0.75pt, scale=1]
\tikzstyle{ver}=[]
\tikzstyle{vert}=[circle, draw, fill=black!100, inner sep=0pt, minimum width=1pt]
\tikzstyle{vertex}=[circle, draw, fill=black!00, inner sep=0pt, minimum width=4pt]

 % draw the 14-gon
 \node[regular polygon, regular polygon sides=14, minimum size=3cm, draw] (a) {};
 \coordinate (C) at (a.center);
 \node at (C) {14};

 % define vertices of 14-gon
 \path (a.corner 1) coordinate (P1);
 \path (a.corner 2) coordinate (P2);
 \path (a.corner 3) coordinate (P3);
 \path (a.corner 4) coordinate (P4);
 \path (a.corner 5) coordinate (P5);
 \path (a.corner 6) coordinate (P6);
 \path (a.corner 7) coordinate (P7);
 \path (a.corner 8) coordinate (P8);
 \path (a.corner 9) coordinate (P9);
 \path (a.corner 10) coordinate (P10);
 \path (a.corner 11) coordinate (P11);
 \path (a.corner 12) coordinate (P12);
 \path (a.corner 13) coordinate (P13);
 \path (a.corner 14) coordinate (P14);

 % triangles on alternate edges, minus one more
 \draw (P1) -- ($ (P1)!1!{-60}:(P2) $) -- (P2);

 \draw (P4) -- ($ (P4)!1!{-60}:(P5) $) -- (P5);
 
 \draw (P7) -- ($ (P7)!1!{-60}:(P8) $) -- (P8);
 \draw (P9) -- ($ (P9)!1!{-60}:(P10) $) -- (P10);
 \draw (P11) -- ($ (P11)!1!{-60}:(P12) $) -- (P12);
 \draw (P13) -- ($ (P13)!1!{-60}:(P14) $) -- (P14);

 % ---- Outer nodes placed individually on a circle ----
 \def\radius{2 cm} % radius of circle around 14-gon
\path (C) ++(25:\radius) coordinate (Q2);

 \path (C) ++(0:\radius) coordinate (Q1);
\path (C) ++(25:\radius) coordinate (Q2);
 \path (C) ++(50:\radius) coordinate (Q3);
 \path (C) ++(76:\radius) coordinate (Q4);
 \path (C) ++(103:\radius) coordinate (Q5);
 \path (C) ++(121:\radius) coordinate (Q6);
 \path (C) ++(140:\radius) coordinate (Q7);
 \path (C) ++(156:\radius) coordinate (Q8b);
 \path (C) ++(180:\radius) coordinate (Q8);
 \path (C) ++(190:\radius) coordinate (Q8a);
 \path (C) ++(215:\radius) coordinate (Q9);
 \path (C) ++(197:\radius) coordinate (Q9a);
 \path (C) ++(228:\radius) coordinate (Q10);
 \path (C) ++(257:\radius) coordinate (Q11);
 \path (C) ++(282:\radius) coordinate (Q12);
 \path (C) ++(309:\radius) coordinate (Q13);
 \path (C) ++(335:\radius) coordinate (Q14);
 \path (C) ++(360:\radius) coordinate (Q15);

 % example: connect two of them
 \draw (P13) --(Q2);
 \draw (P14) --(Q3);
 
 \draw (P1) --(Q4);
 \draw (P2) --(Q5);
 \draw (P3) --(Q6);
 \draw (P3) --(Q7);
 \draw (Q6) --(Q7);
 \draw (P4) --(Q8b);
 
 \draw (P5) --(Q8);
 \draw (P6) --(Q9);
 \draw (P6) --(Q9a);

\draw (Q9) --(Q9a);
 \draw (P7) --(Q10);
 \draw (P8) --(Q11);
 \draw (P9) --(Q12);
 \draw (P10) --(Q13);
 \draw (P11) --(Q14);
 \draw (P12) --(Q15);
 
 \pgfmathsetmacro{\polyR}{4.5/2}
 \pgfmathsetmacro{\radius}{\polyR*0.75} % letters a bit outside

\node[scale=0.75] at ($(C)+(10:\radius)$) {12};

\node[scale=0.75] at ($(C)+(62:\radius)$) {5};

\node[scale=0.75] at ($(C)+(115:\radius)$) {12};
\node[scale=0.75] at ($(C)+(146:\radius)$) {5};

\node[scale=0.75] at ($(C)+(192:\radius)$) {12};
\node[scale=0.75] at ($(C)+(220:\radius)$) {5};

\node[scale=0.75] at ($(C)+(270:\radius)$) {12};
 
\node[scale=0.75] at ($(C)+(325:\radius)$) {5};
 
\end{tikzpicture}}
\label{nbd142}
\subfigure[$H_3$]{
\begin{tikzpicture}[line width=0.75pt, scale=1]
\tikzstyle{ver}=[]
\tikzstyle{vert}=[circle, draw, fill=black!100, inner sep=0pt, minimum width=1pt]
\tikzstyle{vertex}=[circle, draw, fill=black!00, inner sep=0pt, minimum width=4pt]

 % draw the 14-gon
 \node[regular polygon, regular polygon sides=14, minimum size=3cm, draw] (a) {};
 \coordinate (C) at (a.center);
 \node at (C) {14};

 % define vertices of 14-gon
 \path (a.corner 1) coordinate (P1);
 \path (a.corner 2) coordinate (P2);
 \path (a.corner 3) coordinate (P3);
 \path (a.corner 4) coordinate (P4);
 \path (a.corner 5) coordinate (P5);
 \path (a.corner 6) coordinate (P6);
 \path (a.corner 7) coordinate (P7);
 \path (a.corner 8) coordinate (P8);
 \path (a.corner 9) coordinate (P9);
 \path (a.corner 10) coordinate (P10);
 \path (a.corner 11) coordinate (P11);
 \path (a.corner 12) coordinate (P12);
 \path (a.corner 13) coordinate (P13);
 \path (a.corner 14) coordinate (P14);

 % triangles on alternate edges, minus one more

 \draw (P4) -- ($ (P4)!1!{-60}:(P5) $) -- (P5);
 
 \draw (P7) -- ($ (P7)!1!{-60}:(P8) $) -- (P8);
 \draw (P9) -- ($ (P9)!1!{-60}:(P10) $) -- (P10);
 \draw (P11) -- ($ (P11)!1!{-60}:(P12) $) -- (P12);
 \draw (P13) -- ($ (P13)!1!{-60}:(P14) $) -- (P14);
 
 \draw (P2) -- ($ (P2)!1!{-60}:(P3) $) -- (P3);

 % ---- Outer nodes placed individually on a circle ----
 \def\radius{2 cm} % radius of circle around 14-gon
\path (C) ++(25:\radius) coordinate (Q2);

 \path (C) ++(0:\radius) coordinate (Q1);
\path (C) ++(25:\radius) coordinate (Q2);
 \path (C) ++(50:\radius) coordinate (Q3);
 \path (C) ++(66:\radius) coordinate (Q4);
 \path (C) ++(86:\radius) coordinate (Q4a);
 \path (C) ++(103:\radius) coordinate (Q5);
 \path (C) ++(121:\radius) coordinate (Q6);
 \path (C) ++(132:\radius) coordinate (Q7);
 \path (C) ++(156:\radius) coordinate (Q8b);
 \path (C) ++(180:\radius) coordinate (Q8);
 \path (C) ++(190:\radius) coordinate (Q8a);
 \path (C) ++(215:\radius) coordinate (Q9);
 \path (C) ++(197:\radius) coordinate (Q9a);
 \path (C) ++(228:\radius) coordinate (Q10);
 \path (C) ++(257:\radius) coordinate (Q11);
 \path (C) ++(282:\radius) coordinate (Q12);
 \path (C) ++(309:\radius) coordinate (Q13);
 \path (C) ++(335:\radius) coordinate (Q14);
 \path (C) ++(360:\radius) coordinate (Q15);

 % example: connect two of them
 \draw (P13) --(Q2);
 \draw (P14) --(Q3);
 \draw (P1) --(Q4a);
 \draw (P1) --(Q4);
 \draw (Q4) --(Q4a);
 \draw (P2) --(Q5);

 \draw (P3) --(Q7);

 \draw (P4) --(Q8b);
 
 \draw (P5) --(Q8);
 \draw (P6) --(Q9);
 \draw (P6) --(Q9a);

\draw (Q9) --(Q9a);
 \draw (P7) --(Q10);
 \draw (P8) --(Q11);
 \draw (P9) --(Q12);
 \draw (P10) --(Q13);
 \draw (P11) --(Q14);
 \draw (P12) --(Q15);
 
 \pgfmathsetmacro{\polyR}{4.5/2}
 \pgfmathsetmacro{\radius}{\polyR*0.75} % letters a bit outside

\node[scale=0.75] at ($(C)+(11:\radius)$) {5};

\node[scale=0.75] at ($(C)+(60:\radius)$) {12};
\node[scale=0.75] at ($(C)+(91:\radius)$) {5};

\node[scale=0.75] at ($(C)+(145:\radius)$) {12};

\node[scale=0.75] at ($(C)+(192:\radius)$) {5};
\node[scale=0.75] at ($(C)+(220:\radius)$) {12};

\node[scale=0.75] at ($(C)+(270:\radius)$) {5};
 
\node[scale=0.75] at ($(C)+(325:\radius)$) {12};
 
\end{tikzpicture}}
\subfigure[$H_4$]{\begin{tikzpicture}[line width=0.75pt, scale=1]
\tikzstyle{ver}=[]
\tikzstyle{vert}=[circle, draw, fill=black!100, inner sep=0pt, minimum width=1pt]
\tikzstyle{vertex}=[circle, draw, fill=black!00, inner sep=0pt, minimum width=4pt]

 % draw the 14-gon
 \node[regular polygon, regular polygon sides=14, minimum size=3cm, draw] (a) {};
 \coordinate (C) at (a.center);
 \node at (C) {14};

 % define vertices of 14-gon
 \path (a.corner 1) coordinate (P1);
 \path (a.corner 2) coordinate (P2);
 \path (a.corner 3) coordinate (P3);
 \path (a.corner 4) coordinate (P4);
 \path (a.corner 5) coordinate (P5);
 \path (a.corner 6) coordinate (P6);
 \path (a.corner 7) coordinate (P7);
 \path (a.corner 8) coordinate (P8);
 \path (a.corner 9) coordinate (P9);
 \path (a.corner 10) coordinate (P10);
 \path (a.corner 11) coordinate (P11);
 \path (a.corner 12) coordinate (P12);
 \path (a.corner 13) coordinate (P13);
 \path (a.corner 14) coordinate (P14);

 % triangles on alternate edges, minus one more

 \draw (P4) -- ($ (P4)!1!{-60}:(P5) $) -- (P5);
 
 \draw (P7) -- ($ (P7)!1!{-60}:(P8) $) -- (P8);
 \draw (P9) -- ($ (P9)!1!{-60}:(P10) $) -- (P10);
 \draw (P11) -- ($ (P11)!1!{-60}:(P12) $) -- (P12);

 \draw (P14) -- ($ (P14)!1!{-60}:(P1) $) -- (P1);
 
 \draw (P2) -- ($ (P2)!1!{-60}:(P3) $) -- (P3);

 % ---- Outer nodes placed individually on a circle ----
 \def\radius{2 cm} % radius of circle around 14-gon
\path (C) ++(25:\radius) coordinate (Q2);

 \path (C) ++(0:\radius) coordinate (Q1);
\path (C) ++(18:\radius) coordinate (Q2);
\path (C) ++(35:\radius) coordinate (Q2a);
 \path (C) ++(50:\radius) coordinate (Q3);
 \path (C) ++(75:\radius) coordinate (Q4);
 
 \path (C) ++(103:\radius) coordinate (Q5);
 \path (C) ++(121:\radius) coordinate (Q6);
 \path (C) ++(132:\radius) coordinate (Q7);
 \path (C) ++(156:\radius) coordinate (Q8b);
 \path (C) ++(180:\radius) coordinate (Q8);
 \path (C) ++(190:\radius) coordinate (Q8a);
 \path (C) ++(215:\radius) coordinate (Q9);
 \path (C) ++(197:\radius) coordinate (Q9a);
 \path (C) ++(228:\radius) coordinate (Q10);
 \path (C) ++(257:\radius) coordinate (Q11);
 \path (C) ++(282:\radius) coordinate (Q12);
 \path (C) ++(309:\radius) coordinate (Q13);
 \path (C) ++(335:\radius) coordinate (Q14);
 \path (C) ++(360:\radius) coordinate (Q15);

 % example: connect two of them
 \draw (P13) --(Q2a);
 \draw (P13) --(Q2);
 \draw (Q2) --(Q2a);
 
 \draw (P14) --(Q3);
 
 \draw (P1) --(Q4);

 \draw (P2) --(Q5);

 \draw (P3) --(Q7);

 \draw (P4) --(Q8b);
 
 \draw (P5) --(Q8);
 \draw (P6) --(Q9);
 \draw (P6) --(Q9a);

\draw (Q9) --(Q9a);
 \draw (P7) --(Q10);
 \draw (P8) --(Q11);
 \draw (P9) --(Q12);
 \draw (P10) --(Q13);
 \draw (P11) --(Q14);
 \draw (P12) --(Q15);

 \pgfmathsetmacro{\polyR}{4.5/2}
 \pgfmathsetmacro{\radius}{\polyR*0.75} % letters a bit outside

\node[scale=0.75] at ($(C)+(11:\radius)$) {12};
\node[scale=0.75] at ($(C)+(40:\radius)$) {5};
\node[scale=0.75] at ($(C)+(91:\radius)$) {12};

\node[scale=0.75] at ($(C)+(145:\radius)$) {5};

\node[scale=0.75] at ($(C)+(192:\radius)$) {12};
\node[scale=0.75] at ($(C)+(220:\radius)$) {5};

\node[scale=0.75] at ($(C)+(270:\radius)$) {12};
 
\node[scale=0.75] at ($(C)+(325:\radius)$) {5};
 
\end{tikzpicture}}}

\captionof{figure}{$\mathcal{N}$: Neighborhoods of $14$-gon}
\end{figure}

Let $\mathcal{N}$ denote the set of all neighborhood of $14$-gon.

\begin{definition}[Transition Matrix / Adjacency Matrix]
To quantify the connectivity of the tiling space, we treat the distinct face neighborhood types as states in a symbolic dynamical system. We define the transition matrix
\[
M = (M_{ij})_{1 \le i,j \le m},
\]
where $m$ is the number of distinct admissible neighborhood types. An entry $M_{ij}$ represents the number of distinct ways a neighborhood of type $N_i$ can be joined along a common edge to a neighborhood of type $N_j$ such that the resulting configuration satisfies the local consistency conditions A1--A2.
\end{definition}

By a case-by-case verification of all possible neighborhood-adjacencies, constrained by Observations \ref{A1}--\ref{A2}, yields the following transition matrix $M$:

\[M = \quad 
\begin{array}{c|cccccccccc}
 & \mathrm{T} & \mathrm{P} & \mathrm{D1} & \mathrm{D2} & \mathrm{D3} & \mathrm{D4} & \mathrm{H1} & \mathrm{H2} & \mathrm{H3} & \mathrm{H4} \\
\hline
\mathrm{T} & 0 & 1 & 1 & 2 & 1 & 1 & 3 & 6 & 3 & 3 \\
\mathrm{P} & 1 & 0 & 1 & 2 & 1 & 1 & 2 & 4 & 2 & 2 \\
\mathrm{D1} & 1 & 1 & 0 & 0 & 0 & 0 & 1 & 0 & 0 & 0 \\
\mathrm{D2} & 2 & 2 & 0 & 0 & 0 & 0 & 3 & 2 & 2 & 2 \\
\mathrm{D3} & 1 & 1 & 0 & 0 & 0 & 0 & 3 & 1 & 1 & 1 \\
\mathrm{D4} & 1 & 1 & 0 & 0 & 0 & 0 & 0 & 1 & 1 & 1 \\
\mathrm{H1} & 3 & 2 & 1 & 3 & 3 & 0 & 0 & 0 & 0 & 0 \\
\mathrm{H2} & 6 & 4 & 0& 2 & 1 & 1 & 0 & 0 & 0 & 0\\
\mathrm{H3} & 3 & 2 & 0& 2 & 1 & 1 & 0 & 0 & 0 & 0 \\
\mathrm{H4} & 3 & 2 & 0& 2 & 1 & 1 & 0 & 0 & 0 & 0 \\
\end{array}
\]

While an entry $M_{ij} > 1$ indicates multiple local orientations for joining neighborhood $N_i$ to $N_j$, our Extension Lemma (Lemma \ref{extn}) establishes that for every such transition, there exists at least one choice of orientation that is globally consistent. Consequently, while arbitrary choices may lead to configurations that cannot be completed (dead ends), the transition matrix $M$ robustly describes the skeleton of a non-empty, infinite tiling space.
\subsubsection{Cohomology of the Anderson-Putnam Approximant for Type $[3, 5, 12, 14]$}
Let $\Omega (3, 5, 12, 14)$ denote the tiling space of regular polygons of sizes $\{3, 5, 12, 14\}$ so that they form a fan around a vertex.
The transition matrix allows us to compute the cohomology of the zeroth Anderson-Putnam approximant, $\Gamma_0$, associated with the tiling space $\Omega (3, 5, 12, 14)$. This approximant $\Gamma_0$  is constructed as a 2-dimensional CW-complex as follows (\cite{KP00, S08}):
\begin{itemize}
\item \textbf{0-cells}: The set $\mathcal{N}$ consisting of the 10 dual vertex-stars, forming the basis $\langle e_i : i=1,\dots,10\rangle$ of $C_0 = \mathbb{Z}^{10}$.
\item \textbf{1-cells}: Directed edges defined by the transition matrix $M_{\text{red}}$, $ \langle e_{i\to j}^{(k)} : 1\le i,j\le 10,\ 1\le k\le M_{ij}\rangle $ forming a basis of $C_1 = \mathbb{Z}^{110}$.
\item \textbf{2-cells}: Quadrilateral cycles corresponding to dual tiles (primal vertex-stars), grouped into sets. The set of matrix-admissible 2-cell cycles 
\[
\begin{aligned}
\mathcal{S}_1 &= \{(T, P, D_i, H_j) \mid i \in \{2, 3, 4\}, j \in \{1, 2, 3, 4\}, (i=4 \Rightarrow j \neq 1)\}, \\
\mathcal{S}_2 &= \{(T, P, H_j, D_i) \mid i \in \{1, 2, 3, 4\}, j \in \{1, 2, 3, 4\}, (i=1 \Rightarrow j=1), (i=4 \Rightarrow j \neq 1)\}, \\
\mathcal{S}_3 &= \{(T, D_i, P, H_j) \mid i, j \in \{1, 2, 3, 4\}\},
\end{aligned}
\]
One can verify that all the cycles in $\mathcal{S}_1$, $\mathcal{S}_2$ and $\mathcal{S}_3$ except $5$ sequences in $\mathcal{S}_3$ are realizable as valid cycles. With this \[
|\mathcal{S}_1| = 99, \quad |\mathcal{S}_2| = 68, \quad |\mathcal{S}_3| = 279, \quad \text{Total }= 446.
\]

\[C_2=\langle f_\alpha : \alpha=1,\dots,446\rangle,
\]
where each $f_\alpha$ is a labelled oriented $4$-cycle belonging to one of the sets $\mathcal{S}_1, \mathcal{S}_2$ and $\mathcal{S}_3$. 
 
\end{itemize}
The cochain complex is 
\[
0 \to C^0 \xrightarrow{\delta^0} C^1 \xrightarrow{\delta^1} C^2 \to 0, \quad C^0 = \mathbb{Z}^{10}, \quad C^1 = \mathbb{Z}^{110}, \quad C^2 = \mathbb{Z}^{446},
\]
where $\delta^i = (\partial_{i+1})^*$ described below. \\

On basis elements,
\[
\boxed{\ \delta_1\big(e_{i\to j}^{(k)}\big)=e_j-e_i\ }.
\]
Each generator $f_\alpha$ corresponds to a labelled oriented 4-cycle
\[
(i_1\overset{(k_1)}\to i_2\overset{(k_2)}\to i_3\overset{(k_3)}\to i_4\overset{(k_4)}\to i_1),
\]
where the superscript $(k_s)$ picks a specific parallel edge copy among
the $M_{i_s,i_{s+1}}$ choices. With coherent orientation the boundary is
the (signed) sum of the four oriented edges in order:
\[
\boxed{\ \delta_2(f_\alpha)
= e_{i_1\to i_2}^{(k_1)} + e_{i_2\to i_3}^{(k_2)} + e_{i_3\to i_4}^{(k_3)}
 + e_{i_4\to i_1}^{(k_4)}\ }.
\]
A computer assisted computation yields the cohomology $ H^i(\Gamma_0) = \operatorname{ker} \delta^i / \operatorname{im} \delta^{i-1}$ :
 \[ H^0(\Gamma_0) = \mathbb{Z}^1 \quad H^1(\Gamma_0) = \mathbb{Z}^{36}, \quad H^2(\Gamma_0)= \mathbb{Z}^{381} \quad H^k(\Gamma_0) = 0 \quad \text{for} \; k > 2 \]

\begin{remark}[Cohomology of the Anderson-Putnam Approximant]
The complex $\Gamma_0$ constructed above represents the zeroth approximant in the Anderson-Putnam inverse limit sequence for the tiling space $\Omega([3, 5, 12, 14])$. The ranks of the resulting cohomology groups, $H^1(\Gamma_0) = \mathbb{Z}^{36}$ and $H^2(\Gamma_0) = \mathbb{Z}^{381}$, provide explicit combinatorial bounds on the \v{C}ech cohomology of the continuous hull. Geometrically, the generators of $H^1(\Gamma_0)$ correspond to the independent degrees of freedom for shifting the hypercyclic $(T, P)$ chains (as constructed in Proposition 3.2) without violating local consistency. The rank of $H^2(\Gamma_0)$ formally enumerates the fundamental 2-cycles in the local transition graph, reflecting the heavily branched nature of the matching rules. Under the inverse limit $\Omega \cong \varprojlim \Gamma_n$, these values serve as the initial topological invariants for the dynamical system associated with this weakly aperiodic protoset.
\end{remark}

\end{document}